\documentclass[11pt]{article}
\usepackage[toc,title,page]{appendix}
\usepackage[usenames]{color}
\usepackage{amssymb,amsmath,latexsym}
\numberwithin{equation}{section}
\usepackage{amsthm}
\usepackage[active]{srcltx}
\usepackage[mathscr]{eucal}
\usepackage{setspace}
\usepackage[color=yellow,textsize=small]{todonotes}
\usepackage{geometry}
\usepackage{multirow}
\usepackage{graphicx}
\usepackage{float}
\usepackage{amsfonts}
\usepackage{amssymb}
\usepackage{amsthm}
\usepackage{newlfont}
\usepackage{stmaryrd}
\usepackage{mathrsfs}
\usepackage[colorlinks,
            linkcolor=blue,
            anchorcolor=blue,
            citecolor=blue
            ]{hyperref}
\usepackage{indentfirst}
\usepackage{cases}
\usepackage{bookmark}
\usepackage{graphicx}
\usepackage{pdfpages}
\usepackage{setspace}
\usepackage{abstract}
\usepackage{cite}

\newsavebox{\mybox}

\allowdisplaybreaks[4]

\newtheorem{theorem}{Theorem}[section]
\newtheorem{lemma}[theorem]{Lemma}
\newtheorem{corollary}[theorem]{Corollary}
\newtheorem{definition}[theorem]{Definition}

\newtheorem{condition}[theorem]{Condition}

\def\dbE{{\mathbb{E}}}
\def\dbF{{\mathbb{F}}}

\def\dbH{{\mathbb{H}}}

\def\dbL{{\mathbb{L}}}

\def\dbN{{\mathbb{N}}}

\def\dbP{{\mathbb{P}}}

\def\dbR{{\mathbb{R}}}
\def\dbS{{\mathbb{S}}}

\def\g{\gamma}
\def\d{\delta}
\def\e{\varepsilon}

\def\l{\lambda}

\def\t{\tau}
\def\f{\varphi}
\def\th{\theta}
\def\o{\omega}
\def\p{\phi}

\def\G{\Gamma}

\def\O{\Omega}
\def\Om{\Omega}

\def\cD{{\cal D}}

\def\cF{{\cal F}}

\def\cL{{\cal L}}

\def\cQ{{\cal Q}}

\def\cU{{\cal U}}

\def\rp{{\mathrm p}}
\def\rq{{\mathrm q}}

\def\${|\!|\!|}
\def\({\left(}
\def\){\right)}

\newcounter{bean}
\newcommand{\benuma}{\setlength{\labelwidth}{.25in}
	\begin{list}%
		{(\alph{bean})}{\usecounter{bean}}}
	\newcommand{\eenuma}{\end{list}}

\newcommand{\bi}{\begin{itemize}}
	\newcommand{\ei}{\end{itemize}}
\newcommand{\be}{\begin{enumerate}}
	\newcommand{\ee}{\end{enumerate}}
\newcommand{\beqs}{\begin{equation*}}
\newcommand{\eeqs}{\end{equation*}}
\newcommand{\beq}{\begin{equation}}
    \newcommand{\eeq}{\end{equation}}
    \newcommand{\bald}{\begin{aligned}}
        \newcommand{\eald}{\end{aligned}}
\newcommand{\beqys}{\begin{eqnarray*}}
	\newcommand{\eeqys}{\end{eqnarray*}}
\newcommand{\beqy}{\begin{eqnarray}}
\newcommand{\eeqy}{\end{eqnarray}}

\newcommand{\wt}{\widetilde}

\newcommand{\tin}{\quad \text{in}\ }

\mathchardef\mhyphen="2D
\newcommand{\bp}{\begin{pmatrix}}
	\newcommand{\ep}{\end{pmatrix}}
\begin{document}
\title{ \bf Properties for transposition solutions to operator-valued BSEEs, and applications to robust second order necessary conditions for controlled SEEs
 }
\date{}

\author{Guangdong Jing \thanks{ 
Department of Mathematics, 
Beijing Institute of Technology,    Beijing 100081, China. \newline  \indent  {E-mail:} 
{jingguangdong@mail.sdu.edu.cn.}
}  
}

\maketitle
\begin{abstract}
  \noindent
This article is concerned with the second order necessary conditions for the stochastic optimal control problem of stochastic evolution equation with model uncertainty when the traditional Pontryagin-type maximum principle holds trivially and does not provide any information depicting the optimal control. The diffusion terms of the state equations are allowed to be control dependent with convex control constraints. Transposition method is adopted to deal with the adjoint operator-valued backward stochastic evolution equations, especially the correction terms. Besides, weak convergence arguments are used to obtain the optimal uncertainty reference measure, in which the regularities of the state processes, variational processes, and adjoint processes in the transposition sense are characterized. Malliavin calculus is applied to pave the way for differentiation theorem of Lebesgue type to deduce the pointwise robust optimality conditions.

\end{abstract}
 
\bigskip


\bigskip


\section{Introduction}

Denote by $(\O, \cF, \mathbf F, \dbP)$ a completed filtered probability space, with $\mathbf F=\{\cF_t\}_{t\ge0}$ the natural filtration and a one-dimensional standard Brownian motion $\{W(t)\}_{t\ge0}$ on it.  Denote the progressive $\sigma$-field  corresponding to $\mathbf F$ by $\dbF$. Let $T>0$.
Denote $H$ a real separable Hilbert space with norm $|\cdot|_{H}$ and inner product $ \langle \cdot,\cdot \rangle_{H}$. Assume that $A$ is an unbounded linear operator on $H$. Further assume that it generates a $C_0$-semigroup $\{e^{At}\}_{t\geq 0}$. 
Denote by $A^*$ the adjoint operator of $A$.  $D(A)$ is a Hilbert space with the usual graph norm, and $A^*$ is the infinitesimal generator of $\{e^{A^*t}\}_{t\geq 0}$, the adjoint $C_0$-semigroup of $\{e^{At}\}_{t\geq 0}$. 
Denote $H_{1}$ another separable Hilbert space, and $U\subset H_1$ is convex and closed.

This article focuses on studying the robust second order necessary conditions for singular stochastic optimal control problems, where the states are governed by a controlled  stochastic evolution equation (SEE for short) with model uncertainty: 
\begin{equation}\label{incso}
    \left\{
    \begin{aligned}
    & {d}x_\gamma(t)=\left( A x_\gamma(t) +a_\gamma(t,x_\gamma(t),u(t))\right) {d}t
        +b_\gamma(t,x_\gamma(t),u(t)) {d}W(t) \tin   t\in(0,T],       \\
    &x_\gamma (0)=x_0\in H,
    \end{aligned}
    \right.
\end{equation}
where concrete conditions assumed for the coefficients 
$a_\cdot(\cdot,\cdot,\cdot):[0,T] \times H\times U\times \Gamma\to H$ 
and 
$b_\cdot(\cdot,\cdot,\cdot):[0,T]  \times H\times U\times \Gamma\to H$
will be made clear later. 
The cost functional with certain uncertainty parameter $\gamma \in \Gamma$ is formulated in the form of 
\begin{equation} 
    J(u(\cdot);   \gamma) := \mathbb{E}\Big(\int_0^Tg_\gamma(t,x_\gamma(t),u(t)) {d}t
          +h_\gamma(x_\gamma(T))\Big), 
\end{equation}
where 
$g_\cdot(\cdot,\cdot,\cdot):[0,T] \times H\times U\times \Gamma\to \mathbb R$
and 
$h_\cdot( \cdot):   H \times \Gamma\to \mathbb R$.
However, generally speaking, the uncertainty parameter on certain  circumstance may follow different laws, and it cannot be determined in advance. Therefore, the cost functional dedicated to any specific law $\l(\cdot)$ on $\Gamma$ should be denoted as 
\begin{equation} 
    J(u(\cdot); \lambda ):= \int_{\Gamma}   J(u(\cdot);   \gamma)\lambda( {d}\gamma),
\end{equation}
and further we define 
\begin{equation}
    J(u(\cdot)  ):=   \sup_{\lambda\in\Lambda}  J(u(\cdot); \lambda), 
\end{equation}
where $\Lambda $ denotes the set of all the possible laws of uncertainty parameters $\gamma \in \Gamma$.

The optimal control problem is to find any admissible control $\bar u \in \mathcal U^p[0,T]$ (see \eqref{rkju}), such that 
\begin{equation}\label{djfsf}
    J(\bar u(\cdot))=  \inf_{u(\cdot) \in \mathcal U^p[0,T]} J(u(\cdot)).
\end{equation}
Any $\bar u (\cdot) \in \mathcal U^p[0,T]$ satisfying \eqref{djfsf} is called an optimal control. The corresponding state $\bar x_\gamma (\cdot;\bar u(\cdot))$ is called an optimal state, and $(\bar u (\cdot), \bar x_\gamma (\cdot))$ an optimal pair.

In order to obtain robust optimality conditions, in the above formulation, we have taken into account the model uncertainty firstly proposed in \cite{humingshangwangfalei2020siam}. Here $\Gamma$ is a locally compact Polish space, and the parameter $\gamma\in\Gamma$ represents various market conditions.  
This category of uncertainty has wide-ranging implications.
For instance, controlled state systems vary between bull and bear markets, and the corresponding cost functions are also different, denoting respectively by $J_1(u(\cdot))$ and $J_2(u(\cdot))$. 
However, it is challenging to predict the probability for experiencing a bull market in reality.  One should more reasonably characterize the cost functional by \[\sup_{ \mathfrak{i}  \in[0,1]}\left\{ \mathfrak{i}  J_1(u(\cdot))+(1- \mathfrak{i}  )J_2(u(\cdot))\right\}.\]  
Nevertheless, the above  is  merely  nominal formulation, since we actually are not confirmed that whether the set 
\[\Lambda^{u}:=
       \Big\{\widetilde\lambda \in\Lambda \big|  \sup_{\lambda\in\Lambda}  J(u(\cdot); \lambda) =   J(u(\cdot); \widetilde\lambda )\Big\}
\]
is nonempty. It is even unknown if the cost functional is measurable. We formally introduce the formulation here, which will be made rigorous afterwards.

For stochastic optimal control problem with nonconvex control constraints and when the diffusion of the state equation is dependent on the control, there are essential difficulties conpared with the deterministic counterpart, such as the loss of orders when estimating the variational terms of stochastic integral.  
The breakthrough on the study of the  first order necessary conditions of stochastic optimal control problems for the general case was made in \cite{pengshige1},  motivated by the theory of non-linear backward stochastic differential equations \cite{pardouxpengshige} and the additionally introduced adjoint equations \cite{pengshige1}, although some special cases in the early stage can be found in \cite{KushnerSchweppe1964,KushnerHJ}. 
Further, results on the second order necessary conditions for singular stochastic optimal control problems were firstly presented much later in \cite{MahmudovBashirov1997}.

Recently, the second order necessary conditions for singular \emph{stochastic} optimal control problems in finite dimension have attracted rather extensive attention, and we only present a brief description of the context. For example, in the classical stochastic differential equations setting: 
\cite{tangshanjian2010dcds} the diffusion term of the control system is control independent and the control regions are allowed to be nonconvex; 
\cite{zhanghaisenzhangxu2015siam} with convex control constraint and diffusion term contains control;
further \cite{Frankowskazhanghaisenzhangxu2017jde,zhanghaisenzhangxu2017siam,zhanghaisenzhangxu2018siamreview,zhanghaisenzhangxu2016scm} with general control constraint and control dependent diffusion term;
and further \cite{Frankowskazhanghaisenzhangxu2019transams,Frankowskazhanghaisenzhangxu2018siam} with additional state constraints; \cite{BonnansSilva2012amo} with either convex control constraints or finitely many equality and inequality constraints over the final state.
A second order stochastic maximum principle for generalized mean-field singular control problem  \cite{guohanchengxiongjie2018mcrf} with diffusion term independent of control variable.
A second order maximum principle for singular optimal controls with recursive utilities of stochastic systems \cite{dongyuchaomengqinxin2019jota} and stochastic delay systems \cite{haotaomengqingxin2019ejc}, in both of which  the control domain is nonconvex and the diffusion coefficients are independent of control.
In \cite{jing23+}, the author studys the second order necessary conditions for optimal control problem of stochastic differential equations with model uncertainty in finite dimension. 
Readers who are interested in several kinds of the second order necessary optimality conditions for stochastic optimal control problem in detail can refer to \cite{zhanghaisenzhangxu2018siamreview}. 
At last, we recommend  works on the second order necessary conditions for deterministic optimal control problems \cite{GabasovKirillova1972,warga1978jde, cuiqingdenglizhangxu2016crmasp, cuiqingdenglizhangxu2019esaim,denglizhangxu2021jde,FrankowskaHoehener2017jde, FrankowskaOsmolovskii2020scl, FrankowskaOsmolovskii2019amo, FrankowskaOsmolovskii2018siam, Hoehener2013amo,louhongwei2010dcds, louhongweiyongjiongmin2018mcrf, Nguyen2019p, Osmolovskii2020jota, Osmolovskii2018jmaa} among numerous others   to interested readers.

When studying the  necessary optimality condition for optimal control problems of controlled SEEs utilizing the dual analysis method, formally, a second order adjoint equation should be introduced \cite{yongjiongminzhouxunyu1999book,luqzhx21}. However, its correction part of the backward SEE that the adjoint processes should satisfy involves stochastic integral of linear operator-valued stochastic  processes with respect to Brownian motions, about which there is not good enough theory in the literature since the space of the bounded linear operators on a Hilbert space is neither separable nor reflexive.  Even so, it should be mentioned that there are still works (cf. \cite{zhouxy93}) in the early stage dedicated to these problems imposing technical restrictions to avoid difficulties exceeding the scope of stochastic analyais in Hilbert spaces and to  perform classical analysis.

Up to this moment, there are several works \cite{Frankowskaluqi2020jde,Frankowskazhangxu2020spa,luqizhanghaisenzhangxu2021siam,luqi2016conference} presenting results on the second order necessary  conditions for the optimal control problem of SEEs. \cite{Frankowskaluqi2020jde,Frankowskazhangxu2020spa} deal with cases with  nonconvexity control constrains utilizing classical variational calculus. As a comparison, \cite{luqizhanghaisenzhangxu2021siam} requires convex control regions. In return, \cite{Frankowskaluqi2020jde,Frankowskazhangxu2020spa} address necessary conditions of integral type, while \cite{luqizhanghaisenzhangxu2021siam} gives necessary conditions of both integral and pointwise type.

In the seminar work \cite{pengshige1}, Peng also illustrates that in the Pontryagin-type maximum principle, the correction part of the second order adjoint equation do not exert influence explicitly. Recently, there are several important works utilizing this observation. For example, \cite{dumeng13,fuhute13} characterize the stochastic maximum principle by means of the Riesz representation theorem, further \cite{liusongwang23} with state and control delay, and \cite{liutang23} defines a kind of operator-valued conditional expectation so as to absorb the correction part of the second order  adjoint equation into the infinitesimal error term.

However, when the optimal control problem is singular, in other words, when the first order necessary optimality condition given by the Pontryagin-type maximum principle is trivial, we should further study the second order necessary optimality condition. Note that in this case, the correction part do explicitly  join the expression of the necessary condition as shown by \cite{zhanghaisenzhangxu2015siam,jing23+} and depicting the property of the correct part of the second order adjoint equation directly is inevitable. Besides, the information for the correction part play an important role in numerical scheme, which will be useful in application. 
Towards solving this problem,  \cite{luqzhx14,luqzhx15,luqzhx18} introduce a kind of relaxed transposition solution and then $V$-transposition solution with respect to the operator-valued backward SEE in the spirit of Riesz representation and finite dimensional approximation, which adapt the classical duality relationship into the definition of solutions for the adjoint equations. By thorough approximation argument and convergence analysis, the regularity of the correction part can be characterized, benefited from which the Pontryagin maximum principle and further the second order necessary optimality condition can then be presented (see also the monograph \cite{luqzhx21} for a detailed introduction).

The motivation of this paper is to study the singular optimal control problem of controlled  stochastic evolution equations with model uncertainty, to obtain integral type and then pointwise robust second order necessary optimality conditions. 
We adopt the relaxed transposition solution and $V$-transposition solution \cite{luqzhx21} methods mentioned above to characterize the properties of the operator-valued backward stochastic evolution equations especially the correction parts to further perform dual principle arguments and Lebesgue differentiation theorem to obtain the pointwise conditions. Moreover, tools in Malliavin calculus will be utilized as proposed initially in \cite{zhanghaisenzhangxu2015siam} and followed by \cite{luqizhanghaisenzhangxu2021siam,zhanghaisenzhangxu2017siam,zhanghaisenzhangxu2018siamreview} to overcome another main obstacle: the deficit of integrable order to be precise, which prevents the usage of classical Lebesgue differentiation theorem. 
Compared with the models under uncertainty \cite{humingshangwangfalei2020siam,hlw23,jing23+}, we study controlled stochastic evolution systems in infinite dimension setting, making the results more widely applicable to, for example, controlled stochastic heat equations, stochastic Schr\"odinger equations, stochastic Cahn-Hilliard equations, stochastic Korteweg de Vries equations, stochastic Kuramoto-Sivashinsky equations, etc.

Accordingly, as one of the main contributions, we derive the regularities of linearized state processes, vector-valued and operator-valued adjoint processes, in Hilbert spaces and Banach spaces separately, about the continuity dependence with respect to the uncertainty parameter, which is essential for all the terms appearing in the variational expansion of cost functional to be well-defined. Besides, as another main contribution, we obtain the optimality uncertainty reference  measure under certain convexity by the weak convergence argument provided with the above regularities and dependence results, which indicates the robustness of the optimal control regarding the supremum in the cost functionals taken over a family of probability measures. 
With the common reference measure, both integral type and then pointwise second order necessary optimality conditions are established in turn.

The rest of the paper is arranged as follows. Section \ref{sctele} are devoted to notations and preliminaries. 
In Section \ref{sctllse}, we present the main assumed conditions in this paper, and study the properties of linearized state equations, mainly about the order analysis. After that, the continuity dependence of the state and variational equations on the uncertainty parameter are shown in Section \ref{sctlcytb}, while the properties for the adjoint equations are mainly gathered in Section \ref{stlpsez}. The main applications are presented in Section \ref{lsecmair}, the proofs of which are carried out separately in Section \ref{lstjxo} and Section \ref{slpwocs}. Several examples are presented in Section \ref{slme} to demonstrate the motivation of the results. At last, a few elementary lemmas and preliminaries  are put in Section \ref{slyz}.

\section{Notations and preliminaries}\label{sctele}
\subsection{Notations}

Denote $X$ a Banach space, with the norm ${|\cdot|_{X}}$. Denote $L_{\mathcal{F}_{t}}^{p}(\Omega ; X)$ the Banach space of all ${\mathcal{F}_{t}}$-measurable random variables $\xi: \Omega \rightarrow X$ satisfying  $\mathbb{E}|\xi|_{X}^{p}<\infty$, ${t \in[0, T]}$, $p \in[1, \infty)$.  
Besides, denote $D_{\mathbb{F}} ([0, T] ; L^{p}(\Omega ; X) )$ the vector space of all the $X$-valued $\mathbf{F}$-adapted processes $\varphi(\cdot):[0, T] \rightarrow L_{\mathcal{F}_{T}}^{p}(\Omega ; X)$, which are right continuous with left limits exist and equipped with the norm 
$|\varphi(\cdot)|_{D_{\mathbb{F}}([0, T] ; L^{p}(\Omega ; X))}:= \sup _{t \in[0, T)}(\mathbb{E}|\varphi(t)|_{X}^{p})^{1 / p}$. 
Denote $C_{\mathbb{F}}([0, T] ; L^{p}(\Omega ; X))$ the Banach space of all the $X$-valued $\mathbf{F}$-adapted processes $\varphi(\cdot):[0, T] \rightarrow L_{\mathcal{F}_{T}}^{p}(\Omega ; X)$ which are continuous and whose norm are inherited from $D_{\mathbb{F}}([0, T] ; L^{p}(\Omega ; X))$. 
What's more, denote two Banach spaces 
with $p_{1}, p_{2}, p_{3}, p_{4} \in[1, \infty)$: 
\begin{align*}
L_{\mathbb{F}}^{p_{1}} (\Omega ; L^{p_{2}}(0, T ; X) )=  \Big\{&f:(0, T) \times \Omega \rightarrow X \  |\   f(\cdot)  \ \text{is }  \mathbf{F}\mbox{-}\text{adapted and}
\\
&|f|_{L_{\mathbb{F}}^{p_{1}}\left(\Omega ; L^{p_{2}}(0, T ; X)\right)} :=  \big[\mathbb{E}\big(\int_{0}^{T}|f(t)|_{X}^{p_{2}} d t\big)^{\frac{p_{1}}{p_{2}}}\big]^{\frac{1}{p_{1}}}<\infty \Big\},
\end{align*}
\begin{align*}
L_{\mathbb{F}}^{p_{2}}\left(0, T ; L^{p_{1}}(\Omega ; X)\right)=  \Big\{&f:(0, T) \times \Omega \rightarrow X \  |\   f(\cdot)  \ \text{is }  \mathbf{F}\mbox{-}\text{adapted and}
\\ 
&|f|_{L_{\mathbb{F}}^{p_{2}}\left(0, T ; L^{p_{1}}(\Omega ; X)\right)} := \big[\int_{0}^{T}\left(\mathbb{E}|f(t)|_{X}^{p_{1}}\right)^{\frac{p_{2}}{p_{1}}} d t\big]^{\frac{1}{p_{2}}}<\infty \Big \}. 
\end{align*}
If $p_1=p_2=p$, denote the above two spaces as $L_{\mathbb{F}}^{p } (0, T ;  X)$.  
For another Banach space $Y$, denote $\cal L(X,Y)$ the space of bounded linear operators from $X$ to $Y$, which are equipped with the usual operator norm, and, if $X=Y$, suppressed as $\cal L(X)$. Denote $\cal L_2(X,Y)$ the space of all Hilbert-Schmidt operators from $X$ to $Y$. 
$\mathcal{L}_{p d} (L_{\mathbb{F}}^{r_{1}}\left(0, T ; L^{r_{2}}(\Omega ; X)\right) ; L_{\mathbb{F}}^{r_{3}}\left(0, T ; L^{r_{4}}(\Omega ; Y)\right) )$ denotes the vector space of all pointwise defined bounded  linear operators $\cal L$ from $L_{\mathbb{F}}^{r_{1}}\left(0, T ; L^{r_{2}}(\Omega ; X)\right)$ to $L_{\mathbb{F}}^{r_{3}}\left(0, T ; L^{r_{4}}(\Omega ; Y)\right)$. That is to say, for a.e.  $(t,\o) \in [0,T]\times \O$, there exists bounded linear operator $\cal L(t,\o)$ satisfying $(\cal L \f)(t,\o) =  \cal L(t,\o) \f(t,\o)$ for any $\f\in L_{\mathbb{F}}^{r_{1}}\left(0, T ; L^{r_{2}}(\Omega ; X)\right)$. Other pointwise defined operator spaces can be defined similarly. 
For any $p\ge 1$, put the admissible control set 
\begin{equation}\label{rkju}
\cU^{p}[0,T] := \big\{u(\cdot)\in L^{p}_\dbF(0,T;H_{1}) \;\big|\; u(t,\omega)\in U,\; (t,\omega)\in [0,T]\times\Omega\; a.e.\big\}.
\end{equation}

Constants $C$ and $C_L$ may vary from line to line among the whole paper. 
Besides, we would like to note two small distinctions regarding the symbols in this article. In general, $\varepsilon$ is exclusively used to denote the parameter in variational calculus for optimal control, while  $\epsilon$ is more broadly employed in infinitesimal analysis. Besides, $p$ represents integrability index, whereas $\rp$ signifies the first adjoint processes.

\subsection{Preliminaries for stochastic evolution equation}
Consider the following SEEs:
\begin{equation}\label{eqmaror}
    \left\{
\begin{aligned}
&d X(t)=[AX(t)+f(t,X(t))]dt +\widetilde f(t,X(t))dW(t) \quad \mbox{in}\  (0,T] 
\\
&X(0)=X_0, 
\end{aligned}
\right.
\end{equation}
where $X_0: \O \to H$ is an $\cF_0$-measurable random variable, $A$ generates a $C_0$-semigroup $\{S(t)\}_{t\ge0}$ on $H$, and $f(\cdot, \cdot), \widetilde f(\cdot, \cdot): [0,T]\times \O \times H \to H$ are two given  functions satisfying the following conditions:
\begin{condition}\label{eqmarorllccl}
(i) Both $f(\cdot,x)$ and $\widetilde f(\cdot, x)$ are $\mathbf F$-adapted for any given $x\in H$;

(ii) There exist two nonnegative functions $L_1(\cdot)\in L^1(0,T), L_2(\cdot)\in L^2(0,T)$, such that for any given $x,y \in H$ and a.e. $t\in[0,T]$, 
\begin{equation}\label{eqmarorllc}
    \left\{
    \begin{aligned}
       & |f(t,x)-f(t,y)|_H\le L_1(t)|x-y|_H, 
        \\
       & |\widetilde f(t,x)-\widetilde f(t,y)|_H\le L_2(t)|x-y|_H,
    \end{aligned}
    \right.
    \indent   \dbP\mbox{-}a.s. 
\end{equation}
\end{condition}
\begin{definition}
    An $H$-valued $\mathbf F$-adapted continuous stochastic process $X(\cdot)$ is called a mild solution to \eqref{eqmaror} if $f(\cdot, X(\cdot))\in L^1(0,T;H)\ a.s.$, $\widetilde f(\cdot, X(\cdot))\in L_{\mathbb F}^{2,loc}(0,T;H)$, and for any $t\in[0,T]$, 
    \begin{equation*}
        \begin{aligned}
            X(t)=S(t)X_0+\int_0^tS(t-s)f(s,X(s))ds +\int_0^t S(t-s)\widetilde f(s,X(s))dW(s) \quad \dbP\mbox{-}a.s. 
        \end{aligned}
    \end{equation*} 
\end{definition}
\begin{lemma}\label{mtuee}
    Let Condition \ref{eqmarorllccl}   hold, $f(\cdot, 0)\in L_{\mathbb F}^p(\O;L^1(0,T;H))$, $\widetilde f(\cdot, 0)\in L_{\mathbb F}^{ p}(\O;L^2(0,T;H))$ for some $p\ge2$. Then for any $X_0\in L_{\cal F_0}^{ p}(\O;H)$, the above equation \eqref{eqmaror} admits a unique mild solution $X(\cdot)\in C_{\mathbb F}([0,T];L^p(\O;H))$. Moreover, 
\begin{align}\label{mtueer}
    |X(\cdot)|_{C_{\mathbb F}([0,T];L^{ p}(\O;H))} \le C\big(|X_0|_{L_{\cal F_0}^{ p}(\O;H)} +|f(\cdot, 0)|_{L_{\mathbb F}^p(\O;L^1(0,T;H))} +|\widetilde f(\cdot, 0)|_{L_{\mathbb F}^p(\O;L^2(0,T;H))}\big). 
\end{align}
\end{lemma}
The  above Lemma \ref{mtuee} can be found in \cite[Theorem 3.14]{luqzhx21}.

\section{Regularities for linearized state equations}\label{sctllse}

Firstly we introduce the assumptions that will be used later.

{\bf(A1)} Suppose that
$a_\cdot(\cdot,\cdot,\cdot):[0,T]\times H\times U \times \Gamma\to H$ and
$b_\cdot(\cdot,\cdot,\cdot):[0,T]\times H\times U \times \Gamma \to H$ are
two maps satisfying:

(i)
For any $(x,u,\gamma )\in
H\times U \times \Gamma$, both $a_\gamma (\cdot,x,u):[0,T]\to H$ and 
$b_\gamma (\cdot,x,u):[0,T]\to H$ are Lebesgue measurable;

(ii)
There is a constant $C_L>0$ such that, for a.e. $t\in[0,T]$, any $x,\tilde x\in H$, any $\gamma \in \Gamma $, any $u,\tilde u\in U$,
\begin{equation*}
\left\{ 
\begin{aligned}
&|a_\gamma (t,x,u) - a_\gamma (t,\tilde x,\tilde u)|_H+|b_\gamma (t,x,u) - b_\gamma (t,\tilde x,\tilde u)|_H \leq C_L\big(|x-\tilde x|_H+|u-\tilde u|_{H_{1}}\big),
\\
&|a_\gamma (t,0,0)|_H +|b_\gamma (t,0,0)|_H \leq C_L.
\end{aligned}
\right.
\end{equation*}

{\bf(A2)} Suppose that
$g_\cdot (\cdot ,\cdot,\cdot):[0,T]\times H\times U \times \Gamma \to \dbR$
and $h_\cdot(\cdot):H\times \Gamma\to \dbR$ are two functionals
satisfying:

 (i)
For any $(x,u,\gamma )\in H\times U\times \Gamma $,
$g_\gamma(\cdot,x,u):[0,T]\to \dbR$ is Lebesgue
measurable;

(ii)
There is a
constant $C_L>0$ such that, for a.e. $t\in[0,T]$, any $x\in H$ and $u\in U$,
\begin{equation*}
|g_\gamma(t,x,u)| +|h_\gamma(x)|
 \leq C_L(1+|x|_{H}^{2}+|u|_{H_{1}}^{2}).
\end{equation*}

{\bf(A3)}  The maps $a_\gamma (t,\cdot,\cdot)$ and $b_\gamma(t,\cdot,\cdot)$, and the functionals $g_\gamma(t,\cdot,\cdot)$ and $h_\gamma(\cdot)$ are $C^2$ with respect to $x$ and $u$. Moreover, there exists a constant $C_L>0$ such that, for a.e. $t\in[0,T]$ and any $(x,u,\gamma )\in H\times U\times \Gamma $,
\begin{equation*}\label{ab}
\left\{
\begin{aligned} 
& |\partial_x a_\gamma (t,x,u)|_{\cal L(H)}+|\partial_x b_\gamma (t,x,u)|_{\cL(H)} + |\partial_u a_\gamma (t,x,u)|_{\cL(H_1;H)}+ 
|\partial_u b_\gamma (t,x,u)|_{\cL(H_1;H)}\leq C_L,
\\
&|\partial_{xx}a_\gamma (t,x,u)|_{\cL(H\times H;H)}+|\partial_{xx}b_\gamma (t,x,u)|_{\cL(H\times  H;H)}+|\partial_{xu}a_\gamma (t,x,u)|_{\cL(H\times  H_{1};H)}
\\
&+|\partial_{xu}b_\gamma(t,x,u)|_{\cL(H\times H_{1};H)}+|\partial_{uu}a_\gamma(t,x,u)|_{\cL(H_{1}\times  H_{1};H)}+|\partial_{uu}b_\gamma(t,x,u)|_{\cL(H_{1}\times H_{1};H)}
\leq C_L,
\\
&|\partial_x g_\gamma (t,x,u)|_{H} +|\partial_u g_\gamma(t,x,u)|_{H_{1}}+|\partial_x h_\gamma(x)|_{H} \leq C_L(1+|x|_{H}+|u|_{H_{1}}), 
\\
&|\partial_{xx}g_\gamma(t,x,u)|_{\cL(H)}+|\partial_{xu}g_\gamma(t,x,u)|_{\cL(H;H_{1})}+|\partial_{uu}g_\gamma(t,x,u)|_{\cL(H_{1})}+|\partial_{xx}h_\gamma(x)|_{\cL(H)}\leq C_L.
\end{aligned}
\right.
\end{equation*}

{\bf(A4)} $\Gamma$ is a locally compact Polish space with distance $ \mathsf{d} $, where $\gamma\in\Gamma$ is the uncertainty parameter. For each $N>0$, there exists a modulus of continuity $\rho_N:[0,\infty)\to[0,\infty)$ such that
\begin{equation*}
\begin{aligned}
  |\varphi_\gamma(t,x,u)-\varphi_{\gamma'}(t,x,u)| \le \rho_N \left( \mathsf{d} (\gamma, \gamma')\right)
\end{aligned}
\end{equation*}
for any $t\in[0,T], x\in H$ and $|x|\le N, u \in U, \gamma,\gamma'\in\Gamma$, where $\varphi_\gamma$ represents $a_\gamma, b_\gamma, g_\gamma, h_\gamma$ and its derivatives in $(x,u)$ up to second orders separately.

{\bf(A5)} The control region $U(\subset H_1)$ is nonempty and convex.

{\bf(A6)}  $\Lambda$ is a weakly compact and convex set of probability measures on $(\Gamma,\mathcal{B}(\Gamma))$.

Let $\bar u(\cdot)\in \cU^{p}[0,T]$ be an optimal control and ${\bar x}_\gamma (\cdot)$ the corresponding  state of the controlled system \eqref{incso}. Let $u(\cdot)\in\cU^{p}[0,T]$ be another
admissible control with related state $x_\gamma (\cdot)$. 
Set $\d u(\cdot)=u(\cdot)-\bar u(\cdot)$, $\d x_\gamma (\cdot)= x_\gamma (\cdot;{u(\cdot)})-{\bar  x}_\gamma (\cdot;{\bar u(\cdot)})$, $u^\varepsilon (\cdot)=\bar u(\cdot) + \varepsilon (u(\cdot)-\bar u(\cdot))$, $\d  u^\varepsilon (\cdot)= u^\varepsilon (\cdot) - \bar u(\cdot)$, $\d x^\varepsilon_\gamma (\cdot)= x_\gamma (\cdot;{u^\varepsilon (\cdot)})-{\bar  x}_\gamma (\cdot;{\bar u(\cdot)})$. 
Consider the following first and second order linearized evolution equations:
\begin{equation}\label{lisef}
\left\{
\begin{aligned} 
&dy_\gamma = \big(Ay_\gamma + \partial_{x}a_\gamma[t] y_\gamma +  \partial_{u}a_\gamma[t]\d u \big)dt + \big(\partial_{x}b_\gamma[t] y_\gamma + \partial_{u}b_\gamma[t]\d u \big)
dW(t) \quad  \mbox{ in }(0,T],  
\\
&y_\gamma(0)=0,
\end{aligned}
\right.
\end{equation}
\begin{equation}\label{lises}
\left\{
\begin{aligned} 
&dz_\gamma = \big(Az_\gamma + \partial_{x}a_\gamma[t] z_\gamma  + \partial_{xx}a_\gamma[t](y_\gamma,y_\gamma) + 2\partial_{xu}a_\gamma[t](y_\gamma,\d u) + \partial_{uu}a_\gamma[t](\d u,\d u)\big)dt 
\\
& \indent \indent + \big(\partial_{x}b_\gamma[t]z_\gamma + \partial_{xx}b_\gamma[t](y_\gamma,y_\gamma) +2\partial_{xu}b_\gamma[t](y_\gamma,\d u) + \partial_{uu}b_\gamma[t](\d u,\d u) \big)dW(t) \quad \mbox{ in } (0,T],
\\
&z_\gamma(0)=0.
\end{aligned}
\right.
\end{equation}

By the Lemma \ref{mtuee}, under the assumptions given in (A1) and (A3), together with the subset relation 
\begin{equation}\label{subssde}
    {C_{\mathbb F}([0,T];L^p(\O;H))} \subset L_{\mathbb F}^1(0,T;L^p(\O;H)) \subset L_{\mathbb F}^p(\O;L^1(0,T;H)), 
\end{equation} 
the Eq. \eqref{lisef} has a unique solution  in the mild sense, which  has the following  explicit form 
\begin{equation}\label{lbyxfg}
\begin{aligned} 
y_\gamma(t) = \int_{0}^{t} e^{A(t-s)} \big( \partial_{x}a_\gamma[t] y_\gamma +  \partial_{u}a_\gamma[t]\d u \big)dt + \int_{0}^{t}  e^{A(t-s)}  \big(\partial_{x}b_\gamma[t] y_\gamma & + \partial_{u}b_\gamma[t]\d u \big) dW(t),
\\
&  \quad  t\in (0,T], \  \dbP\mbox{-}a.s. 
\end{aligned}
\end{equation}
which further derives that the map 
$\d u\to y_\g^{\d u}$  is linear and for any $\kappa \in(0,1)$, 
\begin{align}\label{loybptg}
    y_\g^{\kappa \d u_1 + (1-\kappa) \d u_2} (t)= \kappa  y_\g^{\d u_1  } (t) + (1-\kappa) y_\g^{  \d u_2}(t), \quad  \   t\in [0,T], \  \dbP\mbox{-}a.s.,  \  \gamma \in \Gamma . 
\end{align}

For $\f=a,b$, and $g$, denote 
\begin{align*}
&\partial_x\f_\gamma [t] := \partial_x\f_\gamma(t,\bar x_\gamma(t),\bar u(t)),  \quad  \partial_u\f_\gamma [t]
:= \partial_u\f_\gamma(t,\bar x_\gamma(t),\bar u(t)) , \\ 
&\partial_{xx}\f_\gamma [t] := \partial_{xx}\f_\gamma(t,\bar x_\gamma(t),\bar u(t)), \quad   
\partial_{uu}\f_\gamma [t] := \partial_{uu}\f_\gamma (t,\bar x_\gamma(t),\bar u(t)), 
\\ & \partial_{xu}\f_\gamma [t] : = \partial_{xu}\f_\gamma(t,\bar x_\gamma(t),\bar u(t)),
\\ & 
\partial_x{\tilde \f}_\gamma [t]  := \int_0^1 \partial_x\f_\gamma (t,\bar x_\gamma(t) + \th \d x_\gamma(t), u(t))d\th, 
\\ & 
\partial_u{\tilde \f}_\gamma [t]  := \int_0^1 \partial_u\f_\gamma (t,\bar x_\gamma(t) , \bar u(t) + \th \delta u(t))d\th ,
\\ & 
\partial_{xx}\tilde{\f}_\gamma^{\e}[t]:=\int_{0}^{1}(1-\theta)\partial_{xx}\f_\gamma (t,\bar x_\gamma(t)+\theta\delta x_\gamma^{\e}(t),\bar u(t)+\theta\e \d u(t))d\theta,
\\  
&\partial_{xx}\tilde{h}_\gamma^{\e}(T):=\int_{0}^{1}(1-\theta)\partial_{xx}h_\gamma(\bar x_\gamma(T)+\theta\delta x_\gamma^{\e}(T))d\theta.
\end{align*}
$\partial_{xu}\tilde{\f}_\gamma^{\e}[t], \partial_{uu}\tilde{\f}_\gamma^{\e}[t]$ can be defined similarly. 
The methods employed in the proof of Lemmas \ref{maests} and \ref{dxyzapel} are somewhat established and standard in the optimal control literature, so we will only sketch some of the details.

\begin{lemma}\label{maests}
Let {(A1)} and {(A3)} hold. Then, for any $p\ge2, \gamma \in \Gamma $, it deduces 
\begin{align*}
   & |\delta x_\gamma|_{C_\dbF([0,T];L^{p}(\O;H))}\le     C|\d u|_{L^{p}_\dbF(\Om;L^2(0,T;H_1))},
    \quad 
    |y_\gamma|_{C_\dbF([0,T];L^{p}(\O;H))}\le C|\d u|_{L^{p}_\dbF(\Om;L^2(0,T;H_1))},
    \\
   & |z_\gamma|_{C_\dbF([0,T];L^{p}(\O;H))}\le C|\d u|^{2}_{L^{2p}_\dbF(\Om;L^4(0,T;H_1))},
    \quad 
    |\delta x_\gamma-y_\gamma|_{C_\dbF([0,T];L^{p}(\O;H))} \le C|\d u|^{2}_{L^{2p}_\dbF(\Om;L^4(0,T;H_1))}. 
\end{align*}
\end{lemma}
\begin{proof}
The process $\d x_\gamma $ solves the following equation 
\begin{equation*}
\left\{
\begin{aligned} 
& d\d x_\gamma = \big[A\d x_\gamma +  \partial_x{\tilde a}_\gamma [t] \d x_\gamma + \partial_u{\tilde a}_\gamma [t]\d u  \big]dt + \big[ \partial_x{\tilde b}_\gamma [t]  \d x_\gamma +  \partial_u{\tilde b}_\gamma [t]\d u \big]dW(t) &\mbox{ in  }(0,T],
        \\
& \d x_\gamma(0)=0.
\end{aligned} \right.
\end{equation*}
By  (A1), (A3) and Lemma \ref{mtuee}, it deduces 
$|\delta x_\gamma|_{C_\dbF([0,T];L^{p}(\O;H))}\le     C|\d u|_{L^{p}_\dbF(\Om;L^2(0,T;H_1))}$. 
The process $y_\gamma$ satisfies \eqref{lisef}. By Lemma \ref{mtuee} again, 
$|y_\gamma |_{C_\dbF([0,T];L^{p}(\O;H))}\le C|\d u|_{L^{p}_\dbF(\Om;L^2(0,T;H_1))}$. 
Similarly, for the process $z_\gamma$ satisfying \eqref{lises}, utilizing Lemma \ref{mtuee} together with the above estimate for $y_\gamma$, it can be derived that 
$|z_\gamma|_{C_\dbF([0,T];L^{p}(\O;H))}\le C|\d u|^{2}_{L^{2p}_\dbF(\Om;L^4(0,T;H_1))}$.

The process $\d x_\gamma(\cdot)-y_\gamma(\cdot)$ satisfies 
\begin{equation*}
\left\{
\begin{aligned}
&d(\d x_\gamma -y_\gamma ) = \big[ A(\d x_\gamma -y_\gamma ) + \partial_x{\tilde a}_\gamma [t] (\d x_\gamma -y_\gamma ) +\big(\partial_x{\tilde a}_\gamma [t]- \partial_x{ a}_\gamma [t]\big)y_\gamma  
\\ & \indent\indent\indent\indent 
+ \big(\partial_u{\tilde a}_\gamma [t]- \partial_u{ a}_\gamma [t]\big)\d u  \big]dt + \big[ \partial_x{\tilde b}_\gamma [t] (\d x_\gamma -y_\gamma ) 
\\
& \indent\indent\indent\indent +\big(\partial_x{\tilde b}_\gamma [t]- \partial_x{ b}_\gamma [t]\big)y_\gamma  + \big(\partial_u{\tilde b}_\gamma [t]- \partial_u{ b}_\gamma [t]\big)\d u  \big] dW(t) \quad  \mbox{ in }(0,T],
\\
&(\d x_\gamma -y_\gamma )(0)=0.
\end{aligned}
\right.
\end{equation*}
By Lemma \ref{mtuee}, the  Assumptions  (A1) and (A3),  together with the estimation for $y_\gamma $, it can be  deduced that  
$|\delta x_\gamma-y_\gamma|_{C_\dbF([0,T];L^{p}(\O;H))} \le C|\d u|^{2}_{L^{2p}_\dbF(\Om;L^4(0,T;H_1))}$.
\end{proof}

\begin{corollary}
    Let {(A1)} and {(A3)} hold. 
For any $p\ge 2$, by Lemma \ref{maests}, for any $\gamma \in \Gamma $, it holds that 
\begin{equation}\label{ytrui}
\begin{aligned}
&|\delta x_\gamma^{\e}|_{C_\dbF([0,T];L^{p}(\O;H))}\le C\e|\d u|_{L^{p}_\dbF(\Om;L^2(0,T;H_1))},
\\
&|\delta x_\gamma^{\e}-\e  y_\gamma|_{C_\dbF([0,T];L^{p}(\O;H))} \le C\e^2|\d u|^{2}_{L^{2p}_\dbF(\Om;L^4(0,T;H_1))}.
\end{aligned}
\end{equation}
\end{corollary}

\begin{lemma}\label{dxyzapel}
Under Assumptions (A1) and (A3), there exists a subsequence $\{\e_{n}\}_{n=1}^{\infty}$,  such that
\begin{equation}\label{dxyzape}
\lim_{n\to \infty}\frac{1}{\e_{n}^2} {\big|\delta x_\gamma^{\e_n}-\e_{n} y_\gamma-\frac{\e_{n}^2}{2}z_\gamma\big|}_{C_\dbF([0,T];L^{2}(\O;H))} =0.
\end{equation}
Besides, together with Assumption (A2), for a subsequence $\{\e_{n}\}_{n=1}^{\infty}$ satisfying $x_\gamma^{\e_{n}}(\cdot)\to \bar x_\gamma(\cdot)$  a.e. $(t,\o)\in [0,T]\times\O$ as $n\to\infty$, it yields  
\begin{align}
&\lim_{n\to \infty}\frac{1}{\e_{n}^2} \dbE \int_{0}^{T} {\big\langle\partial_{xx}\tilde{g}_\gamma^{\e_{n}}[t]\delta x_\gamma^{\e_{n}}(t),\delta x_\gamma^{\e_{n}}(t)\big\rangle}_H-\frac{\e_{n}^2}{2}{\big\langle \partial_{xx}g_\gamma [t]y_\gamma(t), y_\gamma(t)\big\rangle}_H dt=0,  \label{dxyzape1}
\\
&\lim_{n\to \infty}\frac{1}{\e_{n}}
\dbE \int_{0}^{T} \big\langle\partial_{xu}\tilde{g}_\gamma^{\e_{n}}[t]\delta x_\gamma^{\e_{n}}(t), \d u(t)\big\rangle_{H_{1}}-\frac{\e_{n} }{2}\big\langle \partial_{xu}g_\gamma [t]y_\gamma(t), \d u(t)\big\rangle_{H_{1}} dt=0,  \label{dxyzape2}
\\
&\lim_{n\to \infty} \dbE \int_{0}^{T} \big\langle\partial_{uu}\tilde{g}_\gamma^{\e_{n}}[t] \d u(t),\d u(t)\big\rangle_{H_{1}}-\frac{1}{2}\big\langle \partial_{uu}g_\gamma [t]\d u(t), \d u(t)\big\rangle_{H_{1}} dt=0,  \label{dxyzape3}
\\
&\lim_{n\to \infty}\frac{1}{\e_{n}^2} \dbE \big[{\big\langle \partial_{xx}\tilde{h}_\gamma^{\e_{n}}(\bar x_\gamma(T))\delta x_\gamma^{\e_{n}}(T),\delta x_\gamma^{\e_{n}}(T)\big\rangle}_H -\frac{\e_{n}^2}{2}{\big\langle \partial_{xx}h_\gamma (\bar x_\gamma(T))y_\gamma(T), y_\gamma(T)\big\rangle}_H\big]=0.  \label{dxyzape4}
\end{align}
\end{lemma}
\begin{proof}
Let $\varUpsilon^\e(\cdot) =\e^{-2}\big(\delta x_\gamma^{\e}(\cdot)-\e y_\gamma(\cdot)-\dfrac{\e^{2}}{2}z_\gamma(\cdot)\big)$. Then $\varUpsilon^\e(\cdot)$ fulfills 
\begin{equation*} 
\left\{
\begin{aligned}
&   d\varUpsilon^\e=\big[ A\varUpsilon^\e + \partial_x a_\gamma[t] \varUpsilon^\e + \Psi_{a}^{\e,\gamma }(t)  \big]dt + \big[\partial_x b_\gamma[t]\varUpsilon^\e + \Psi_{b}^{\e,\gamma }(t)\big]dW(t) \ \ \mbox{ in }(0,T],
\\
& \varUpsilon^\e(0)=0,
\end{aligned}
\right.
\end{equation*}
where 
\begin{align*}
&\Psi_{a}^{\e,\gamma }(t) = \! \big[ \partial_{xx}\tilde{a}_\gamma^{\e}[t]\big(\frac{\d x_\gamma^\e(t)}{\e},\frac{\d x_\gamma^\e(t)}{\e}\big) \! - \frac{1}{2}\partial_{xx}a_\gamma [t](y_\gamma(t),y_\gamma(t))\big]  \! +  \! \big(\partial_{uu}\tilde{a}_\gamma^{\e}[t]-\frac{1}{2}\partial_{uu}a_\gamma [t]\big)(\d u(t),\d u(t)) 
\\ &\indent  \indent \quad 
+ \big[2\partial_{xu}\tilde{a}_\gamma^{\e}[t]\big(\frac{\d x_\gamma^\e(t)}{\e},\d u(t)\big) -\partial_{xu}a_\gamma [t](y_\gamma(t),\d u(t))\big], 
\\ 
&\Psi_{b}^{\e,\gamma }(t) = \! \big[ \partial_{xx}\tilde{b}_\gamma^{\e}[t]\big(\frac{\d x_\gamma^\e(t)}{\e},\frac{\d x_\gamma^\e(t)}{\e}\big) \! - \frac{1}{2}\partial_{xx}b_\gamma [t](y_\gamma(t),y_\gamma(t))\big] + \big(\partial_{uu}\tilde{b}_\gamma^{\e}[t]-\frac{1}{2}\partial_{uu}b_\gamma [t]\big)(\d u(t),\d u(t)) 
\\ &\indent  \indent \quad 
\! + \!  \big[2\partial_{xu}\tilde{b}_\gamma^{\e}[t]\big(\frac{\d x_\gamma^\e(t)}{\e},\d u(t)\big) -\partial_{xu}b_\gamma [t](y_\gamma(t),\d u(t))\big]. 
\end{align*} 
Applying Lemma \ref{mtuee}, it follows that 
\begin{align}\label{thlyfk}
|\varUpsilon^\e|_{C_{\mathbb F}([0,T];L^{ 2}(\O;H))} \le C\big( |\Psi_{a}^{\e,\gamma }|_{L_{\mathbb F}^2(\O;L^1(0,T;H))} +|\Psi_{b}^{\e,\gamma }|_{L_{\mathbb F}^2(\O;L^2(0,T;H))}\big). 
\end{align}
Moreover, 
\[|\Psi_{a}^{\e,\gamma}|^2_{L_{\mathbb F}^2(\O;L^1(0,T;H))}= \dbE \big(\int_{0}^{T}|\Psi_{a}^{\e,\gamma}(t)|_H dt \big)^2 \le C \dbE  \int_{0}^{T}|\Psi_{a}^{\e,\gamma}(t)|^2_H dt. \]
Besides, by \eqref{ytrui}, there exists a  subsequence $\{\e_{n}\}_{n=1}^{\infty}$ such that $x^{\e_{n}}(\cdot)\to \bar x(\cdot)$ a.e.  $(t,\o)\in \O\times[0,T]$, as $n\to\infty$. Then by simple interpolation, Assumptions  (A1) and (A3), Lebesgue's  dominated convergence theorem, as well as the integrable of $y_\gamma , \d u$, 
\begin{align*}
&\lim_{n\to \infty} \dbE \int_{0}^{T}|\Psi_{a}^{\e_n,\gamma}(t)|^2_H dt 
\\
& \indent \le  C\lim_{n\to\infty}\dbE \int_0^T \Big[\Big|  \partial_{xx}\tilde  a_{\gamma }^{\e_{n}}[t]\big(\frac{\d  x{^{\e_{n}}}(t)}{\e_{n}},\frac{\d  x^{\e_{n}}(t)}{\e_{n}}\big) -  \partial_{xx} \tilde  a_{\gamma }^{\e_{n}}[t](y_{\gamma }(t),y_{\gamma }(t))\Big|_{H}^2
\\
&\indent\quad +\Big| \partial_{xx}\tilde  a_{\gamma }^{\e_{n}}[t] -  \frac{1}{2} \partial_{xx} a_{\gamma } [t]  \Big|_{\cL(H\times H,\;H)}^2\cdot|y_{\gamma }(t)|_{H}^4
\\
&\indent\quad +2\Big|\partial_{xu}\tilde  a_{\gamma }^{\e_{n}}[t] \big(\frac{\d  x^{\e_{n}}(t)}{\e_{n}},\d u(t)\big) - \partial_{xu}\tilde  a_{\gamma }^{\e_{n}}[t] (y_{\gamma }(t),\d u(t))\Big|_{H}^2
\\
&\indent\quad  +\big|2  \partial_{xu}\tilde  a_{\gamma }^{\e_{n}}[t]   - \partial_{xu} a_{\gamma } [t] \big|_{\cL(H\times H_{1},\;H)}^2\cdot|y_{\gamma }(t)|_{H}^2\cdot|\d u(t)|_{H_{1}}^2
\\
&\indent\quad  +\Big| \partial_{uu}\tilde  a_{\gamma }^{\e_{n}}[t] -\frac{1}{2} \partial_{uu} a_{\gamma } [t] \Big|_{\cL(H_{1}\times H_{1},\;H)}^2\cdot|\d u(t)|_{H_{1}}^4 \Big]dt
\\
&\indent =0.
\end{align*}
Similarly, we deduce that 
$\lim_{n\to \infty} |\Psi_{b}^{\e_n,\gamma }|_{L_{\mathbb F}^2(\O;L^2(0,T;H))}=0 $. 
Then recalling \eqref{thlyfk}, it derives that 
$
\lim_{n\to \infty}|\varUpsilon^{\e_n}|_{C_{\mathbb F}([0,T];L^{ 2}(\O;H))} =0$. 
The proof of \eqref{dxyzape} is completed. \eqref{dxyzape1}-\eqref{dxyzape4} can be proved similarly. 
\end{proof}

\section{Continuity dependence for state and variational processes}\label{sctlcytb}
Continuity dependence for state and variational equations as well as adjoint equations is essential in obtaining the measurable of the terms in the variational expansion of the cost functional, and in performing the weak convergence analysis to determine the optimal uncertainty reference measure. We present the corresponding results for state processes and variational processes in turn in this section, while that for the adjoint processes will be given in the next section. 
\begin{lemma}\label{yxclt}
    Under Assumptions (A1), (A3), and (A4), 
    \begin{align*}
        \lim_{\epsilon \to 0} \sup_{\mathsf{d}(\gamma ,\gamma')\le \epsilon }|x_\gamma-x_{\gamma'}|_{C_{\mathbb F}([0,T];L^p(\O;H))} =0.
    \end{align*}
\end{lemma}
\begin{proof}
    It can be verified that $x_\gamma-x_{\gamma'}$ fulfills the following equation 
\begin{equation*} 
    \left\{
    \begin{aligned}
    & {d}(x_\gamma-x_{\gamma'}) =\big[ A (x_\gamma-x_{\gamma'})  +\Psi^{x,\gamma ,\gamma'}_a \big] {d}t
    +\Psi^{x,\gamma ,{\gamma'}}_b {d}W(t) \tin   t\in(0,T],       \\
    &(x_\gamma-x_{\gamma'})(0)= 0,
    \end{aligned}
    \right.
\end{equation*} 
where for $\varphi =a,b$, 
$\Psi^{x,\gamma ,\gamma'}_\varphi =\int_0^1 \partial_x \varphi_\gamma(t, x_{\gamma'}+\theta (x_\gamma-{x_{\gamma'}}), u(t)) d\theta (x_\gamma-{x_{\gamma'}}) +(\varphi_\gamma-{\varphi_{\gamma'}})(t,x_{\gamma'}(t),u(t))$.
By Lemma  \ref{mtuee}, it yields 
\begin{align}\label{cdzbxab}
|x_\gamma-x_{\gamma'}|_{C_{\mathbb F}([0,T];L^p(\O;H))} 
\le & C\big( |(a_\gamma-{a_{\gamma'}})(t,x_{\gamma'} ,u )|_{L_{\mathbb F}^p(\O;L^1(0,T;H))}  
    \\
& \quad  +|(b_\gamma-{b_{\gamma'}})(t,x_{\gamma'} ,u )|_{L_{\mathbb F}^p(\O;L^2(0,T;H))}\big). \notag
\end{align}
By the Assumptions (A1) and (A4), we deduce  
\begin{align}\label{njmgml}
 |(a_\gamma-{a_{\gamma'}})(t,x_{\gamma'} ,u )|_H   \le C \big(|(a_\gamma-{a_{\gamma'}})(t,x_{\gamma'} ,u )\chi_{\{|x_{\gamma'}|_H \le N\}}|_H  +(1+|x_{\gamma'}|_H) \chi_{\{|x_{\gamma'}|_H > N\}}\big)  
\end{align}
and 
\begin{align}\label{skelke}
    |(a_\gamma-{a_{\gamma'}})(t,x_{\gamma'} ,u )\chi_{\{|x_{\gamma'}|_H \le N\}}|_H \le \rho_N(\mathsf{d}(\gamma ,\gamma')).
\end{align}
Besides, due to $x_\gamma,x_{\gamma'}\in {C_{\mathbb F}([0,T];L^p(\O;H))}$, 
together with the subset relation \eqref{subssde},  
it can be derived that 
$\int_{0}^{T} (\mathbb E (1+|x_{\gamma'}|_H)^p)^{\frac{1}{p} } dt <\infty$, 
which further implies that 
\begin{align}\label{jtyuth}
    \lim_{N\to\infty}\int_{0}^{T} (\mathbb E ((1+|x_{\gamma'}|_H)^p  \chi_{\{|x_{\gamma'}|_H > N\}}) )^{\frac{1}{p} } dt =0.
\end{align}
Then 
\begin{align*}
   \lim_{\epsilon \to 0} \sup_{\mathsf{d}(\gamma ,\gamma')\le \epsilon }|(a_\gamma-{a_{\gamma'}})(t,x_{\gamma'} ,u )|_{L_{\mathbb F}^p(\O;L^1(0,T;H))} \! \!  
   & \!\! \overset{\eqref{subssde}\eqref{njmgml}\eqref{skelke}}{\le} \int_{0}^{T} (\mathbb E ((1+|x_{\gamma'}|_H)^p  \chi_{\{|x_{\gamma'}|_H > N\}}) )^{\frac{1}{p} } dt 
   \\
   &\ \   \overset{\eqref{jtyuth} }{\to}0, \quad  N\to\infty. 
\end{align*}
Similarly, it can be deduced that 
$    \lim_{\epsilon \to 0} \sup_{\mathsf{d}(\gamma ,\gamma')\le \epsilon }|(b_\gamma-{b_{\gamma'}})(t,x_{\gamma'} ,u )|_{L_{\mathbb F}^p(\O;L^2(0,T;H))} =0$.
Now the conclusion follows from \eqref{cdzbxab}. 
\end{proof}

\begin{lemma}\label{lcyyz}
    Under Assumptions (A1), (A3), and  (A4), 
    \begin{align*}
        \lim_{\epsilon \to 0} \sup_{\mathsf{d}(\gamma ,\gamma')\le \epsilon }|y_\gamma-y_{\gamma'}|_{C_{\mathbb F}([0,T];L^p(\O;H))} =0.
    \end{align*}
\end{lemma}
\begin{proof}
    The process $y_\gamma-y_{\gamma'}$ fulfills the following equation: 
\begin{equation*}
    \left\{
\begin{aligned}
    &d(y_\gamma-y_{\gamma'}) = \big[A(y_\gamma-y_{\gamma'}) + \partial_x a_\gamma(t, x_{\gamma},u)(y_\gamma-y_{\gamma'})  + \Psi^{y,\gamma ,\gamma'}_a \big]dt   \\
    &\indent\indent\indent\indent + \big[\partial_x b_\gamma(t, x_{\gamma},u)(y_\gamma-y_{\gamma'}) + \Psi^{y,\gamma ,{\gamma'}}_b \big] dW(t) \quad  \mbox{ in }(0,T],  
    \\
    &(y_\gamma-y_{\gamma'})(0)=0, 
\end{aligned}\right.
\end{equation*}
where for $\varphi =a,b$, 
\begin{align*} 
\Psi^{y,\gamma ,\gamma'}_\varphi :=   &
\int_0^1 \partial_{xx} \varphi_\gamma(t, x_{\gamma'}+\theta (x_\gamma-{x_{\gamma'}}), u ) d\theta (x_\gamma-{x_{\gamma'}},y_{\gamma'})  
+ (\partial_x \varphi_\gamma- \partial_x \varphi_{\gamma'})(t, x_{\gamma'},u) y_{\gamma'}
\\ 
& +\int_0^1 \partial_{xu} \varphi_\gamma(t, x_{\gamma'}+\theta (x_\gamma-{x_{\gamma'}}), u ) d\theta (x_\gamma-{x_{\gamma'}},\d  u)
+ (\partial_u \varphi_\gamma- \partial_u \varphi_{\gamma'})(t, x_{\gamma'},u) \d  u . 
\end{align*}
Then applying Lemma \ref{mtuee} deduces that 
\begin{align} \label{cdypz}
    |y_\gamma-y_{\gamma'}|_{C_{\mathbb F}([0,T];L^{ p}(\O;H))} \le C\big( |\Psi^{y,\gamma ,\gamma'}_a|_{L_{\mathbb F}^p(\O;L^1(0,T;H))} +|\Psi^{y,\gamma ,\gamma'}_b|_{L_{\mathbb F}^p(\O;L^2(0,T;H))}\big). 
\end{align}
From (A3), we get 
$  {|\int_0^1 \partial_{xx} \varphi_\gamma(t, x_{\gamma'}+\theta (x_\gamma-{x_{\gamma'}}), u ) d\theta (x_\gamma-{x_{\gamma'}},y_{\gamma'}) |}_H   \le  C|x_\gamma-{x_{\gamma'}}|_H|y_{\gamma'}|_H$, 
and, recalling \eqref{subssde}, 
$ (\dbE  (\int_{0}^{T} |x_\gamma-{x_{\gamma'}}|_H|y_{\gamma'}|_H dt )^p )^{\frac{1}{p} }
\le C {|x_\gamma-{x_{\gamma'}}|}_{C_{\mathbb F}([0,T];L^{2 p}(\O;H))}  {|y_{\gamma'}|}_{C_{\mathbb F}([0,T];L^{2 p}(\O;H))} $. 
Besides, 
\begin{align*}
&\big(\dbE \big(\int_{0}^{T} |(\partial_x \varphi_\gamma- \partial_x \varphi_{\gamma'})(t, x_{\gamma'},u) |_{\cal L(H)} |y_{\gamma'}|_H dt\big)^p\big)^{\frac{1}{p} }  
\\
& \quad \le C |y_{\gamma'}|_{C_{\mathbb F}([0,T];L^{2 p}(\O;H))} \int_{0}^{T}  \big(\dbE \big|(\partial_x \varphi_\gamma- \partial_x \varphi_{\gamma'})(t, x_{\gamma'},u)\big(\chi_{\{|x_{\gamma'}|_H> N\}}+\chi_{\{|x_{\gamma'}|_H\le N\}}\big) \big|^{2p}_{\cal L(H)}\big)^{\frac{1}{2p} }dt. 
\end{align*}
By Assumptions (A1) and (A4), for any $N>0$, it can be derived that 
\begin{align*}
\int_{0}^{T}  \big(\dbE |(\partial_x \varphi_\gamma- \partial_x \varphi_{\gamma'})(t, x_{\gamma'},u) \chi_{\{|x_{\gamma'}|_H\le N\}} |^{2p}_{\cal L(H)} \big)^{\frac{1}{2p} }dt & \le C \rho_N(\mathsf{d}(\gamma ,\gamma')), 
\\
\int_{0}^{T} \! \big(\dbE |(\partial_x \varphi_\gamma- \partial_x \varphi_{\gamma'})(t, x_{\gamma'},u) \chi_{\{|x_{\gamma'}|_H> N\}} |^{2p}_{\cal L(H)} \big)^{\frac{1}{2p} }dt & 
\le  C \! \int_{0}^{T} \! \big(\dbE \big((1+|x_{\gamma'}|_H)^{2p}\chi_{\{|x_{\gamma'}|_H> N\}}\big)\big)^{\frac{1}{2p} } dt . 
\end{align*} 
Taking $\underset{\epsilon \to 0}{\lim} \underset{\mathsf{d}(\gamma ,\gamma')\le \epsilon }{\sup}$ and then $N\to\infty$, together with $|{x_{\gamma'}}|_{C_{\mathbb F}([0,T];L^{2 p}(\O;H))} <\infty$, we derive 
\begin{align*}
\lim_{\epsilon \to 0} \sup_{\mathsf{d}(\gamma ,\gamma')\le \epsilon }\big|&\int_0^1 \partial_{xx} \varphi_\gamma(t, x_{\gamma'}  +\theta (x_\gamma-{x_{\gamma'}}), u ) d\theta (x_\gamma-{x_{\gamma'}},y_{\gamma'})  
\\
&
+ (\partial_x \varphi_\gamma- \partial_x \varphi_{\gamma'})(t, x_{\gamma'},u) y_{\gamma'}\big|_{L_{\mathbb F}^p(\O;L^1(0,T;H))} =0.
\end{align*}
Similarly, it holds that 
\begin{align*}
\lim_{\epsilon \to 0} \sup_{\mathsf{d}(\gamma ,\gamma')\le \epsilon }\big| &\int_0^1 \partial_{xu} \varphi_\gamma(t, x_{\gamma'}   +\theta (x_\gamma-{x_{\gamma'}}), u ) d\theta (x_\gamma-{x_{\gamma'}},\d  u)
\\
&  + (\partial_u \varphi_\gamma- \partial_u \varphi_{\gamma'})(t, x_{\gamma'},u) \d  u \big|_{L_{\mathbb F}^p(\O;L^1(0,T;H))} =0.
\end{align*}
Then we obtain 
$
\lim_{\epsilon \to 0} \sup_{\mathsf{d}(\gamma ,\gamma')\le \epsilon } | \Psi^{y,\gamma ,\gamma'}_a  |_{L_{\mathbb F}^p(\O;L^1(0,T;H))} =0. 
$
Moreover, by similar analysis, we have 
$
\lim_{\epsilon \to 0} \sup_{\mathsf{d}(\gamma ,\gamma')\le \epsilon }  |\Psi^{y,\gamma ,\gamma'}_b |_{L_{\mathbb F}^p(\O;L^2(0,T;H))} =0.
$
Now the conclusion follows from \eqref{cdypz}. 
\end{proof}

\section{Properties for adjoint equations}\label{stlpsez}
Consider the following  $H$-valued backward SEEs:
\begin{equation}\label{apHm}
\left\{
\begin{aligned} 
&d{\mathrm p}(t) = - [ A^* {\mathrm p}(t) - f(t,{\mathrm p}(t),{\mathrm q}(t))]dt + {\mathrm q}(t) dW(t) &\mbox{ in }[0,T),\\
& {\mathrm p}(T) = {\mathrm p}_T.
\end{aligned}
\right.
\end{equation}
Here, $f:[0,T]\times  \O \times H\times H  \to H$,  ${\mathrm p}_T \in L_{\cal F_T}^{p}( \O;H)$,  $p\in(1,2]$. 

In order to introduce the  concept of the transposition solution to \eqref{apHm} and its well-posedness, firstly we should introduce the following SEE 
\begin{equation}\label{apHmt}
\left\{
\begin{aligned} 
& d\f = [A\f+ v_1]ds +  v_2 dW(s) \quad  \mbox{ in }(t,T],
\\
& \f(t)=\eta,
\end{aligned}
\right.
\end{equation}
where $t\in[0,T]$, $\eta\in L^{q}_{\cF_t}(\O;H)$, $v_1\in L^1_{\dbF}(t,T;L^{q}(\O;H))$, $v_2\in L^q_{\dbF}(t,T;H)$ with $p\in(1,2]$, $\frac{1}{p} +\frac{1}{q} =1$.

\begin{definition}\label{definition1}
A pair of processes $(\rp(\cdot), \rq(\cdot)) \in  D_{\dbF}([0,T];L^{p}(\O;H)) \times  L^p_{\dbF}(0,T;H)$ are defined to be transposition solutions to Eq. \eqref{apHm}, if 
\begin{align*}
   & \dbE \big\langle \f(T),\rp_T\big\rangle_{H}
    - \dbE\int_t^T \big\langle \f(s),f(s,\rp(s),\rq(s) )\big\rangle_Hds\\
    &\indent 
 = \dbE \big\langle\eta,\rp(t)\big\rangle_H
    + \dbE\int_t^T \big\langle
    v_1(s),\rp(s)\big\rangle_H ds + \dbE\int_t^T
    \big\langle v_2(s),\rq(s)\big\rangle_H ds,
\end{align*}
where $t\in [0,T]$,  $\eta\in L^{q}_{\cF_t}(\O;H)$,   $v_1(\cdot)\in L^1_{\dbF}(t,T;L^{q}(\O;H))$,  $v_2(\cdot)\in L^q_{\dbF}(t,T; H)$, and the process $\f\in C_{\dbF}([t,T];L^{q}(\O;H))$ is the solution to \eqref{apHmt}. 
\end{definition}

\begin{condition}\label{iclb}
The map $f:[0,T]\times  \O \times H\times H  \to H$ in \eqref{apHm} satisfies 

(i) For any given $h_1,h_2\in H$, $f(\cdot,h_1,h_2)$ is $\mathbf F$-adapted;
    
(ii) There exist two nonnegative functions $L_1(\cdot)\in L^1(0,T), L_2(\cdot)\in L^2(0,T)$, such that for any given $h_1,h_2,\tilde h_1,\tilde h_2\in H $ and a.e. $t\in[0,T]$, 
\begin{equation}\label{iclbe}
\begin{aligned}
    |f(t,h_1,h_2)-f(t,\tilde h_1,\tilde h_2)|_H\leq  L_1( t)|h_1-\tilde h_1|_H+L_2( t)|h_2-\tilde h_2|_H  
\end{aligned}
        \indent  \dbP\mbox{-}a.s. 
\end{equation}
\end{condition}

\begin{lemma}\textup{\cite[Theorem 4.19]{luqzhx21}}\label{lcdpq}
    Let Condition \ref{iclb} holds. Then for $f(\cdot ,0,0) \in
    L^{1}_{\dbF} (0,T;L^{p}(\O;H))$, ${\mathrm p}_T \in L_{\cal F_T}^{p}( \O;H)$, there exists a unique transposition solution  $(\rp(\cdot), \!  \rq(\cdot))  \! \in  \!   D_{\dbF}([0,T];  \! L^{p} (\O;H)) \times  L^p_{\dbF} (0,T; \!  H )$ to \eqref{apHm}. 
    Moreover, 
\begin{align}
 & |(\rp(\cdot), \rq(\cdot)|_{D_{\dbF}([0,T]; L^{p} (\O;H)) \times  L^p_{\dbF} (0,T; H)}   \le C(|{\mathrm p}_T|_{L_{\cal F_T}^{p}( \O;H)} + |f(\cdot ,0,0)|_{L^{1}_{\dbF} (0,T;L^{p}(\O;H))}). 
\end{align}
\end{lemma}

Now we take 
\begin{equation*}\label{zv1}
    \left\{
\begin{aligned} 
& \rp_T =  -\partial_x h_\gamma ( \bar{x}_\gamma(T) ),
\\ 
& f(t,{\rp}(t),{\rq}(t))=  - \partial_x a_\gamma (t,\bar{x}_\gamma (t),\bar{u}(t) )^*\rp(t)-\partial_x b_\gamma (t,\bar{x}_\gamma (t),\bar{u}(t))^*\rq(t)   +\partial_x g_\gamma (t,\bar{x}_\gamma,\bar{u}) 
\end{aligned}
    \right.
\end{equation*}
in \eqref{apHm}, then it becomes 
\begin{equation}\label{yjdbjx}
\left\{
    \begin{aligned}
    d\rp_\gamma(t)&=-A^*\rp_\gamma(t)dt-\big (\partial_x a_\gamma (t,\bar{x}_\gamma (t),\bar{u}(t) )^*\rp_\gamma(t)+\partial_x b_\gamma (t,\bar{x}_\gamma (t),\bar{u}(t))^*\rq_\gamma(t)   \\
     &\quad\ -\partial_x g_\gamma (t,\bar{x}_\gamma,\bar{u})\big)dt+\rq_\gamma(t)dW(t) \tin  t\in[0,T), \\
    \rp_\gamma(T)&=-\partial_x h_\gamma ( \bar{x}_\gamma(T) ).
    \end{aligned}\right.
\end{equation}
In the following lemma, we present the continuous dependence of solution $(\rp_\gamma, \rq_\gamma)$ to \eqref{yjdbjx} with respect to the uncertainty parameter $\gamma \in \Gamma $. 
\begin{lemma}\label{lllypqg}
    Under Assumptions (A1)-(A4), it holds  
\begin{align}\label{cyxpq}
\lim_{\epsilon \to 0} \sup_{\mathsf{d}(\gamma ,\gamma')\le \epsilon }  |(\rp_\gamma-\rp_{\gamma'}, \rq_\gamma-\rq_{\gamma'})|_{D_{\dbF}([0,T]; L^{p} (\O;H)) \times  L^p_{\dbF} (0,T; H)} = 0.
\end{align}
\end{lemma}
\begin{proof}
Consider the following 
\[ \left\{
    \begin{aligned}
&d(\rp_\gamma-\rp_{\gamma'}) =-A^*(\rp_\gamma-\rp_{\gamma'}) dt  - \partial_x a_\gamma (t,\bar{x}_\gamma (t),\bar{u}(t) )^*(\rp_\gamma -\rp_{\gamma'})dt 
\\
& \indent\indent\indent\indent  - \partial_x b_\gamma (t,\bar{x}_\gamma (t),\bar{u}(t) )^*(\rq_\gamma-\rq_{\gamma'})dt + \Psi^{\rp,\gamma ,\gamma'}_f  dt  +(\rq_\gamma-\rq_{\gamma'})(t)dW(t) \tin  t\in[0,T), 
     \\
&(\rp_\gamma-\rp_{\gamma'})(T)=
\Psi^{\rp,\gamma ,\gamma'}_h,
    \end{aligned}\right.
\]
where 
\[\begin{aligned}
\Psi^{\rp,\gamma ,\gamma'}_f
&  = - (\partial_x a_\gamma (t,\bar{x}_\gamma (t),\bar{u}(t) )^* - \partial_x a_\gamma (t,\bar{x}_{\gamma'} (t),\bar{u}(t) )^*)\rp_{\gamma'}
\\
&\quad  - (\partial_x a_\gamma (t,\bar{x}_{\gamma'} (t),\bar{u}(t) )^* - \partial_x a_{\gamma'} (t,\bar{x}_{\gamma'} (t),\bar{u}(t) )^*)\rp_{\gamma'}
\\
&\quad  - (\partial_x b_\gamma (t,\bar{x}_\gamma (t),\bar{u}(t) )^* - \partial_x b_\gamma (t,\bar{x}_{\gamma'} (t),\bar{u}(t) )^*)\rq_{\gamma'}
\\
&\quad   - (\partial_x b_\gamma (t,\bar{x}_{\gamma'} (t),\bar{u}(t) )^* - \partial_x b_{\gamma'} (t,\bar{x}_{\gamma'} (t),\bar{u}(t) )^*)\rq_{\gamma'}
\\
&\quad 
+\partial_x g_\gamma (t,\bar{x}_\gamma,\bar{u})-\partial_x g_{\gamma'} (t,\bar{x}_{\gamma'},\bar{u}),
\\
\Psi^{\rp,\gamma ,\gamma'}_h &= -(\partial_x h_\gamma ( \bar{x}_\gamma(T) )-\partial_x h_{\gamma'} ( \bar{x}_{\gamma'}(T) ) ).
\end{aligned}\]
Applying Lemma \ref{lcdpq}, it follows that 
\begin{align}
    & |(\rp_\gamma-\rp_{\gamma'}, \rq_\gamma-\rq_{\gamma'})|_{D_{\dbF}([0,T]; L^{p} (\O;H)) \times  L^p_{\dbF} (0,T; H)}  \le C(|\Psi^{\rp,\gamma ,\gamma'}_h|_{L_{\cal F_T}^{p}( \O;H)} + |\Psi^{\rp,\gamma ,\gamma'}_f|_{L^{1}_{\dbF} (0,T;L^{p}(\O;H))}). 
\end{align}
Then by the Assumption (A1)-(A4), the results can be proved utilizing the similar methods as that in Lemma \ref{lcyyz}, together with the regularities of $\bar x_\gamma, \bar u, \rp_\gamma, \rq_\gamma$ and the continuity of $\bar x_{\g}$ with respect to the uncertainty parameter $\g$.  
\end{proof}

To perform the duality principle, a standard method in the literature is to, aside from \eqref{yjdbjx}, (formally) introduce  the  $\cL(H)$-valued  backward SEE: 
\begin{equation}\label{lhsae}
\left\{
\begin{aligned}
&dP  =  - [(A^*  + J^* )P + P(A + J ) + K^*PK + (K^* Q +  Q K) -F ]dt
  +  Q dW(t) \quad \mbox{ in } [0,T),
\\
&P(T) = P_T,
\end{aligned}
\right.
\end{equation}
where
\begin{equation}\label{saccf}
F\in L^1_\dbF(0,T;L^2(\O;\cL(H))), \quad 
J,K\in L^4_\dbF(0,T; L^\infty(\O;\cL(H))), 
\quad 
P_T\in L^2_{\cF_T}(\O;\cL(H)). 
\end{equation} 
The solutions to \eqref{lhsae} are defined in certain weak sense with the help of test processes fulfilling separately two SEEs in the form:
\begin{equation}\label{lhsaec1} 
\left\{
\begin{aligned}
&d\p_1 = [(A+J)\p_1 + u_1]ds + [K\p_1  + v_1 ]dW(s) &\mbox{ in } (t,T],
\\
&\p_1(t)=\xi_1,
\end{aligned}
\right.
\end{equation}
and
\begin{equation}\label{lhsaec2} 
\left\{
\begin{aligned}
 &d\p_2 = [(A+J)\p_2 + u_2]ds + [K\p_2 + v_2 ] dW(s) &\mbox{ in } (t,T],
 \\
 &\p_2(t)=\xi_2,
\end{aligned}
\right.
\end{equation}
where we assume that $\xi_1,\xi_2 \in L^4_{\cF_t}(\O;H)$, $u_1,u_2 \in L^2_\dbF(t,T;L^4(\O;H))$, $ v_1,v_2\in L^2_\dbF(t,T;L^4(\O;H))$.

Now we introduce a concept of weak solution for the $\cal L(H)$-valued backward SEE \eqref{lhsae} (cf. \cite{luqzhx21,luqzhx14}). Let $H$ be a Hilbert space and $H^*$ its dual. Let  $V$ be another Hilbert space, such that $V\subset H$ and the embedding operator is Hilbert-Schmidt. Then for its dual space $V^*$ it follows that $H^*\subset V^*$, and  the embedding operator from $H^*$ to $V^*$ is also Hilbert-Schmidt.  Denote 
\[\begin{aligned}
D_{\mathbb{F}, w} & \left([0, T] ; L^2(\Omega ; \mathcal{L}(H))\right) 
\\
:= & \Big\{ P \mid P \in \mathcal{L}_{p d} \big(L_{\mathbb{F}}^2\left(0, T ; L^4(\Omega ; H)\right), L_{\mathbb{F}}^2\big(0, T ; L^{\frac{4}{3}}(\Omega ; H)\big) \big), P \xi \in D_{\mathbb{F}} ([t, T] ; L^{\frac{4}{3}}(\Omega ; H) ), \\
&\indent  |P \xi|_{D_{\mathbb{F}} ([t, T] ; L^{\frac{4}{3}}(\Omega ; H) )} \leq C|\xi|_{L_{\mathcal{F}_t}^4(\Omega ; H)},  t \in[0, T], \xi \in L_{\mathcal{F}_t}^4(\Omega ; H)\Big\},
\end{aligned}\] 
and 
\begin{align*}
    & \mathcal{Q}[0, T] := \Big\{\big(Q^{(\cdot)}, \widehat{Q}^{(\cdot)}\big) \ \big| \   Q^{(t)}, \widehat{Q}^{(t)}  \text { are all bounded linear operators:}  
    \\ 
    &\indent\indent \indent\  L_{\mathcal{F}_t}^4(\Omega ; H) \times L_{\mathbb{F}}^2\left(t, T ; L^4(\Omega ; H)\right) \times L_{\mathbb{F}}^2\left(t, T ; L^4(\Omega ; H)\right)  \to L_{\mathbb{F}}^2 (t, T ; L^{\frac{4}{3}}(\Omega ; H) ),
    \\ &\indent \indent\indent \  
 Q^{(t)}(0,0, \cdot)^*=\widehat{Q}^{(t)}(0,0, \cdot), t \in[0, T]\Big\}. 
\end{align*}

\begin{definition}
    Processes $\big(P(\cdot),\big(Q^{(\cdot)}, \widehat{Q}^{(\cdot)}\big)\big) \in D_{\mathbb{F}, w}\left([0, T] ; L^2(\Omega ; \mathcal{L}(H))\right) \times \mathcal{Q}[0, T]$ are defined to be relaxed transposition solutions to the Eq. \eqref{lhsae}, if it satisfy 
\begin{equation}\label{zjdbr}
    \begin{aligned}
    & \mathbb{E}\left\langle P_T \phi_1(T), \phi_2(T)\right\rangle_H-\mathbb{E} \int_t^T\left\langle F(s) \phi_1(s), \phi_2(s)\right\rangle_H d s 
    \\
    &\indent= \mathbb{E}{\left\langle P(t) \xi_1, \xi_2\right\rangle}_H+\mathbb{E} \int_t^T\left\langle P(s) \phi_1(s), u_2(s)\right\rangle_H d s +\mathbb{E} \int_t^T\left\langle P(s) u_1(s), \phi_2(s)\right\rangle_H d s
    \\
    &\indent \quad+\mathbb{E} \int_t^T\left\langle v_1(s), \widehat{Q}^{(t)}\left(\xi_2, u_2, v_2\right)(s)\right\rangle_H d s+\mathbb{E} \int_t^T\left\langle Q^{(t)}\left(\xi_1, u_1, v_1\right)(s), v_2(s)\right\rangle_H d s
    \\
    & \indent\quad+\mathbb{E} \int_t^T\left\langle P(s) K(s) \phi_1(s), v_2(s)\right\rangle_H d s+\mathbb{E} \int_t^T\left\langle P(s) v_1(s), K(s) \phi_2(s)+v_2(s)\right\rangle_H d s,
    \end{aligned}
\end{equation}
where  $u_1(\cdot), u_2(\cdot) \in L_{\mathbb{F}}^2\left(t, T ; L^4(\Omega ; H)\right)$, $ \xi_1, \xi_2 \in L_{\mathcal{F}_t}^4(\Omega ; H)$, $v_1(\cdot), v_2(\cdot) \in L_{\mathbb{F}}^2\left(t, T ; L^4(\Omega ; H)\right)$, $t \in[0, T]$, $\phi_1(\cdot)$ $\phi_2(\cdot)$ are that in \eqref{lhsaec1}  and \eqref{lhsaec2}. 
\end{definition}

\begin{lemma}\textup{\cite[Theorem 12.9]{luqzhx21}}\label{ltPre}
    Suppose that $L^{2}_{\cF_T}(\Om)$  is a separable Banach space and \eqref{saccf} holds. 
    Then there uniquely exists a relaxed transposition solution  to 
    \eqref{lhsae}: $\big(P(\cdot),(Q^{(\cdot)},\widehat   Q^{(\cdot)})\big) \in  D_{\dbF,w}([0,T];  L^{2}(\O;\cL(H)))$ $\times \cQ[0,T]$. Besides, 
    \begin{align*} 
    &   |P|_{D_{\dbF,w}([0,T];L^{2}(\Om; \cL(H)))} + \big|\big(Q^{(\cdot)},\widehat  Q^{(\cdot)}\big)\big|_{\cQ [0,T]}
 \leq C\big(|F|_{L^1_\dbF(0,T; L^{2}(\Om;\cL(H)))} + |P_T|_{L^{2}_{\cF_T}(\Om; \cL(H))}\big).
    \end{align*}
\end{lemma}

Besides, we should introduce another weak solution for the $\cal L(H)$-valued backward SEE \eqref{lhsae}. 
\begin{definition}\label{dvzjz}
    The pair of processes $ (P(\cdot),Q(\cdot)) \! \in  \! D_{\mathbb F,w}([0,T];L^{2}(\Omega;{\cal L}(H)))\times L^2_{\mathbb F}(0,T;{\cal L}_2(H;V^*))$ are defined to be $V$-transposition solutions to  \eqref{lhsae},  if it satisfy 
\begin{align} \label{eqdefsol11}
    &\mathbb E  \big\langle P_T
    \phi_1(T),\phi_2(T)\big\rangle_{H}
     - \mathbb E \int_t^T \big\langle F(s)\phi_1(s),\phi_2(s)\big\rangle_{H}ds 
     \\   & \indent
     = \mathbb E  {\big\langle
    P(t)\xi_1,\xi_2\big\rangle}_{H} +
    \mathbb E \int_t^T {\big\langle P(s)\phi_1(s),u_2(s)\big\rangle}_{H} ds  + 
    \mathbb E \int_t^T {\big\langle
    P(s)u_1(s),\phi_2(s)\big\rangle}_{H} ds  \notag 
    \\  &\indent \quad 
     + \mathbb E \int_t^T {\big\langle
    P(s) K(s)\phi_1(s), v_2(s)\big\rangle}_{H} ds  + 
    \mathbb E \int_t^T \big\langle P(s)v_1 (s),
    K(s)\phi_2(s) + v_2(s)\big\rangle_{H} ds \notag 
    \\  &\indent \quad 
     + \mathbb E \int_t^T \big\langle
    v_1(s),Q^*(s)\phi_2(s)\big\rangle_{V ,V^*} ds +
    \mathbb E \int_t^T \big\langle
    Q(s)\phi_1(s),v_2(s)\big\rangle_{V^*,V} ds,  \notag 
\end{align}
for $t\in [0,T]$, $\xi_1,\xi_2\in L^{4}_{{\cal F}_t}(\Omega;H)$,    $u_1(\cdot), u_2(\cdot)\in L^2_{\mathbb F}(t,T;L^{4}(\Omega;H))$, $v_1(\cdot), v_2(\cdot)\in  L^2_{\mathbb F}(t,T; L^{4}(\Omega;V))$, 
    where $\phi_1(\cdot)$, $\phi_2(\cdot)$ are that in \eqref{lhsaec1}  and \eqref{lhsaec2}.
\end{definition}
    
In order to guarantee the well-posedness of  \eqref{lhsae} in the sense of the $V$-transposition solution  in Definition \ref{dvzjz}, more technical restrictions need to be imposed. 
    Set
    \[ \cal L_{HV} :=\Big\{\mathfrak O \in\cal L(H)| \mbox{ The restriction of $\mathfrak O $ on $V$ belongs to }\cal L(V) \Big\} \]
    with the norm
    $ |\cdot|_{\cal L_{HV}} = |\cdot|_{\cal L(H)} + |\cdot|_{\cal L(V)}. $
Assume 

    {\bf (A7)} { $A$ generates a $C_0$-semigroup on $V $ and $J,K\in L^\infty_\dbF(0,T; \cL_{HV })$. } 
    
\begin{lemma}\textup{\cite[Theorem 3.3]{luqzhx18}}\label{jgjvzle}  
    Provided with \eqref{saccf} and {(A7)}, the equation  \eqref{lhsae} admits a unique $V$-transposition solution $ (P(\cdot),Q(\cdot)) \in D_{\mathbb F,w}([0,T];L^{2}(\Omega;{\cal L}(H)))\times L^2_{\mathbb F}(0,T;{\cal L}_2(H;V^*))$. Moreover,  
\begin{equation}\label{jgjvzl}
    \begin{aligned} 
    &|(P, Q)|_{ D_{\dbF,w}([0,T];L^{2}(\Omega;\cL(H)))\times L^2_{\dbF}(0,T;\cL_2(H;V^*))}
 \leq  C\big(|F|_{L^1_\dbF(0,T;L^{2}(\Om;\cL(H)))} 
    + |P_T|_{L^{2}_{\cF_T}(\Om;\cL(H))}\big).
    \end{aligned}
\end{equation}
\end{lemma}
The proof of Lemma \ref{jgjvzle} relies, among others, on the representation of multiple linear operators, see \cite[Proposition 12.5]{luqzhx21}.

Denote $(P_\gamma ,Q_\gamma)$ the $V$-transposition solution to the $\cal L (H)$-valued backward SEE \eqref{lhsae} where $P_T$, $J(\cdot)$, $K(\cdot)$ and $F(\cdot)$ are given by
\begin{align}\label{lhsaefx}
    P_T=-\partial_{x x}h_\gamma (\bar{x}_\gamma (T)), \quad J(t)=\partial_{x}a_\gamma[t], \quad K(t)=\partial_{x}b_\gamma[t], \quad F(t)=-\partial_{x x}\mathbb{H}_\gamma [t], 
\end{align}
where 
\begin{align*} \mathbb{H} (t,x,u,\rp,\rq)= \langle \rp,a(t,x,u) \rangle_H + \langle \rq,b(t,x,u) \rangle_H -g(t,x,u) , 
     \\
(t,x,u,\rp,\rq)\in
[0,T]\times H \times U\times H\times H, 
\end{align*}
and 
\begin{equation*}
    \left\{
\begin{aligned} 
& \partial_x \dbH_\gamma  [t]:=\partial_x\dbH (t,\bar x_\gamma (t),\bar u(t),\rp_\gamma (t),\rq_\gamma (t)),  
\\
& \partial_u \dbH_\gamma  [t]:=\partial_u \dbH (t,\bar x_\gamma (t),\bar u(t),\rp_\gamma (t),\rq_\gamma (t)),
 \\
& \partial_{xx} \dbH_\gamma  [t] :=\partial_{xx} \dbH (t,\bar x_\gamma (t),\bar  u(t),\rp_\gamma (t),\rq_\gamma (t)),
 \\
& \partial_{xu}\dbH_\gamma  [t]:= \partial_{xu}\dbH (t,\bar x_\gamma (t),\bar  u(t), \rp_\gamma (t), \rq_\gamma (t)),
 \\
& \partial_{uu} \dbH_\gamma  [t]:= \partial_{uu} \dbH (t,\bar x_\gamma (t),\bar  u(t), \rp_\gamma (t), \rq_\gamma (t)).
\end{aligned}
\right.
\end{equation*}
Note that \eqref{lhsaefx} works for both Eq. \eqref{lhsae} and test processes in \eqref{lhsaec1}-\eqref{lhsaec2}.  
Put
\begin{equation}\label{dpqSo}
\dbS_\gamma(t):=\partial_{xu}\dbH_\gamma [t] + \partial_{u}a_\gamma[t]^* P_\gamma(t) +
\partial_{u}b_\gamma[t]^*Q_\gamma(t) +\partial_{u}b_\gamma[t]^*P_\gamma(t)\partial_{x}b_\gamma[t].
\end{equation}

Note that there are several kinds of measurability for Banach space-valued random variables in the literature (cf. \cite{hnvw}): strongly measurable (the variable can be approximated by a sequence of simple measurable functions), measurable (the preimage of each Borel set is measurable) and weakly measurable (the composition with any element in the dual space or in a proper subspace of the dual space is a real-valued measurable function). In non-separable space such as $\cal L(H)$, these measurable notions are
quite different. In this paper, we adopt the weak measurable meaning.

\begin{lemma}\label{PQcl}
    Under conditions (A1)-(A6), the relaxed transposition solution   $\big(P_{\gamma}(\cdot),\big(Q_{\gamma}^{(\cdot)}, \widehat{Q}_{\gamma}^{(\cdot)}\big)\big)$  to  \eqref{lhsae} are measurable in the weak sense with respect to the uncertainty parameter. 
\end{lemma}
\begin{proof}
Let $\big(P_{\gamma_1}(\cdot),\big(Q_{\gamma_1}^{(\cdot)}, \widehat{Q}_{\gamma_1}^{(\cdot)}\big)\big)$ and $\big(P_{\gamma_2}(\cdot),\big(Q_{\gamma_2}^{(\cdot)}, \widehat{Q}_{\gamma_2}^{(\cdot)}\big)\big)$ be solutions to  \eqref{lhsae}  corresponding to uncertainty parameters $\gamma_1$  and $\gamma_2$ separately.  From \eqref{zjdbr}, it derives that 
\begin{align}\label{zjdbrgi}
    &0=  - \mathbb{E}\left\langle {P_{T,\gamma_i}} \phi_{1,\gamma_i}(T), \phi_{2,\gamma_i}(T)\right\rangle_H + \mathbb{E} \int_t^T\left\langle F_{\gamma_i}(s) \phi_{1,{\gamma_i}}(s), \phi_{2,{\gamma_i}}(s)\right\rangle_H d s 
    \\
    &\indent + \mathbb{E}\left\langle P_{\gamma_i}(t) \xi_{1,{\gamma_i}}, \xi_{2,{\gamma_i}}\right\rangle_H  +  \mathbb{E} \int_t^T\left\langle P_{\gamma_i}(s) u_{1,{\gamma_i}}(s), \phi_{2,{\gamma_i}}(s)\right\rangle_H d s  \notag  
    \\
    & \indent 
    +\mathbb{E} \int_t^T\left\langle P_{\gamma_i}(s) \phi_{1,{\gamma_i}}(s), u_{2,{\gamma_i}}(s)\right\rangle_H d s +\mathbb{E} \int_t^T\left\langle P_{\gamma_i}(s) K_{\gamma_i}(s) \phi_{1,{\gamma_i}}(s), v_{2,{\gamma_i}}(s)\right\rangle_H d s  \notag  
    \\
    & \indent 
    +\mathbb{E} \int_t^T{\left\langle P_{\gamma_i}(s) v_{1,{\gamma_i}}(s), K_{\gamma_i}(s) \phi_{2,{\gamma_i}}(s)+v_{2,{\gamma_i}}(s)\right\rangle}_H d s   \notag  
    \\
    &\indent +\mathbb{E} \int_t^T{\big\langle v_{1,{\gamma_i}}(s), \widehat{Q}_{\gamma_i}^{(t)}\left(\xi_{2,{\gamma_i}}, u_{2,{\gamma_i}}, v_{2,{\gamma_i}}\right)(s)\big\rangle}_H d s  \notag  
    \\
    &\indent   +\mathbb{E} \int_t^T{\big\langle Q_{\gamma_i}^{(t)}\left(\xi_{1,{\gamma_i}}, u_{1,{\gamma_i}}, v_{1,{\gamma_i}}\right)(s), v_{2,{\gamma_i}}(s)\big\rangle}_H d s, \quad i=1,2.  \notag  
\end{align}
When this lemma is used among the duality analysis  to derive the necessary optimality condition, actually we take the test processes  $\phi_{1,\gamma_i}(\cdot)=y_{\gamma_i}(\cdot)$,  $\phi_{2,\gamma_i}(\cdot)=y_{\gamma_i}(\cdot)$, $i=1,2$. 
By Assumptions (A1)-(A5), Lemmas \ref{yxclt}, \ref{lcyyz}, \eqref{lhsaefx}, together with the  integrability of $x_\gamma (T;u), y_\gamma (T;u)$ for $u\in \mathcal U^p[0,T], p\ge4$, from tedious calculus of simple interpolation, it yields 
\begin{align}\label{tymh}
    \lim_{\epsilon \to 0} \sup_{\mathsf{d}(\gamma_1 ,\gamma_2)\le \epsilon } \big|\mathbb{E}{\left\langle {P_{T,\gamma_1}} \phi_{1,\gamma_1}(T), \phi_{2,\gamma_1}(T)\right\rangle}_H-\mathbb{E}{\left\langle {P_{T,\gamma_2}} \phi_{1,\gamma_2}(T), \phi_{2,\gamma_2}(T)\right\rangle}_H\big|=0. 
\end{align}
Similarly, by Assumptions (A1)-(A5), Lemmas \ref{yxclt}, \ref{lcyyz}, \ref{lllypqg}, \eqref{lhsaefx}, 
it can be derived that 
\begin{align}\label{dnrm}
    \lim_{\epsilon \to 0} \sup_{\mathsf{d}(\gamma_1 ,\gamma_2)\le \epsilon } \Big|\mathbb{E} \int_t^T{\left\langle F_{\gamma_1}(s) \phi_{1,{\gamma_1}}(s), \phi_{2,{\gamma_1}}(s)\right\rangle}_H d s - \mathbb{E} \int_t^T{\left\langle F_{\gamma_2}(s) \phi_{1,{\gamma_2}}(s), \phi_{2,{\gamma_2}}(s)\right\rangle}_H d s \Big|=0, 
\end{align}
where the difference term for $\langle \rp_\cdot, \partial_{xx} a_\cdot(t,x_\cdot,u) (\cdot,\cdot) \rangle_H \in \mathcal L(H\times H;\mathbb R)$ is estimated as 
\begin{align*}
    & \lim_{\epsilon \to 0} \sup_{\mathsf{d}(\gamma_1 ,\gamma_2)\le \epsilon } \Big|\mathbb{E} \int_t^T{\left\langle \rp_{\gamma_1}(s),\partial_{xx}a_{\gamma_1}(s,x_{\gamma_1},u)(\phi_{1,{\gamma_1}}(s),\phi_{2,{\gamma_1}}(s)) \right\rangle}_H d s 
    \\
    & \indent\indent\indent\indent 
    - \mathbb{E} \int_t^T{\left\langle \rp_{\gamma_2}(s),\partial_{xx} a_{\gamma_2}(s,x_{\gamma_2},u)(\phi_{1,{\gamma_2}}(s), \phi_{2,{\gamma_2}}(s)) \right\rangle}_H d s \Big|
    \\
    & \indent 
    \le  \sup_{\gamma_2\in \Gamma } \big(\mathbb{E} \sup_{s\in[t,T]}|\rp_{\gamma_2}(s)|_H^2\big)^{\frac{1}{2} } \sup_{\gamma_2\in \Gamma } \mathbb{E} \sup_{s\in[t,T]}|\partial_{xx} a_{\gamma_2}(s,x_{\gamma_2},u)|_{\cal L(H\times H;H)} 
    \\ &  \indent\indent 
    \cdot\Big[\sup_{\gamma_2\in \Gamma }\big(\mathbb{E} \sup_{s\in[t,T]}| \phi_{1,{\gamma_2}}(s)|_H^4\big)^{\frac{1}{4} } + \sup_{\gamma_1\in \Gamma }\big(\mathbb{E} \sup_{s\in[t,T]}| \phi_{2,{\gamma_1}}(s)|_H^4\big)^{\frac{1}{4} } \Big]
    \\  & \indent\indent 
    \cdot\lim_{\epsilon \to 0} \sup_{\mathsf{d}(\gamma_1 ,\gamma_2)\le \epsilon } \Big[\big(\mathbb{E} \sup_{s\in[t,T]} |\phi_{2,{\gamma_1}}(s)-\phi_{2,{\gamma_2}}(s)|_H^4\big)^{\frac{1}{4} } + \big(\mathbb{E} \sup_{s\in[t,T]} |\phi_{1,{\gamma_1}}(s)-\phi_{1,{\gamma_2}}(s)|_H^4\big)^{\frac{1}{4} } \Big]
    \\ & \indent \quad 
    + \sup_{\gamma_1\in \Gamma }\big(\mathbb{E} \sup_{s\in[t,T]}| \phi_{1,{\gamma_1}}(s)|_H^4\big)^{\frac{1}{4} }  
    \sup_{\gamma_1\in \Gamma }\big(\mathbb{E} \sup_{s\in[t,T]}| \phi_{2,{\gamma_1}}(s)|_H^4\big)^{\frac{1}{4} }
    \\ 
    & \indent \indent \  
    \cdot \Big[ \sup_{\gamma_1\in \Gamma }\mathbb{E} \sup_{s\in[t,T]}|\partial_{xx} a_{\gamma_1}(s,x_{\gamma_1},u)|_{\cal L(H\times H;H)}  \cdot \lim_{\epsilon \to 0} \sup_{\mathsf{d}(\gamma_1 ,\gamma_2)\le \epsilon }  \big(\mathbb{E} \sup_{s\in[t,T]} |\rp_{ {\gamma_1}}(s)-\rp_{ {\gamma_2}}(s)|_H^2\big)^{\frac{1}{2} }
    \\ 
    & \indent \indent \quad   
    + \sup_{\gamma_2\in \Gamma }\big( \mathbb{E} \sup_{s\in[t,T]}| \rp_{ {\gamma_2}}(s)|_H^4\big)^{\frac{1}{4} } \cdot \big[ \sup_{\gamma_1\in \Gamma }\mathbb{E} \sup_{s\in[t,T]}|\partial_{xx} a_{\gamma_1}(s,x_{\gamma_1},u)|_{\cal L(H\times H;H)} 
    \\
    & \indent \indent \indent    
    \cdot \lim_{\epsilon \to 0} \sup_{\mathsf{d}(\gamma_1 ,\gamma_2)\le \epsilon } \big( \mathbb{E} \sup_{s\in[t,T]}|x_{ {\gamma_1}}(s)-x_{ {\gamma_2}}(s)|_H^4\big)^{\frac{1}{4} } 
    \\
    & \indent \indent \indent    
    + \lim_{\epsilon \to 0} \sup_{\mathsf{d}(\gamma_1 ,\gamma_2)\le \epsilon } (\mathbb{E} \sup_{s\in[t,T]} (|\partial_{xx} a_{\gamma_1}(s,x_{\gamma_2},u) - \partial_{xx} a_{\gamma_2}(s,x_{\gamma_2},u)|_{\cal L(H\times H;H)})^4)^{\frac{1}{4} } \big]\Big] 
    \\ 
    & \indent 
    =0, 
\end{align*}
where we utilized $\sup_{\gamma \in \Gamma}\mathbb{E}\sup_{s\in[t, T]}{|p_{\gamma}(s)|}_{H}^{4}<\infty$ and $\sup_{\gamma \in \Gamma}\mathbb{E}\sup_{s\in[t, T]}{|p_{\gamma}(s)|}_{H}^{2}<\infty$ for $u\in \mathcal{U}^4[0,T]$ by (A3) and Lemma \ref{lcdpq}, as well as (A4) and Lemmas \ref{yxclt}, \ref{lllypqg}, together with the integrability of test processes $\phi $ and its continuity dependence with respect to the uncertainty parameter which can be proved similarly. 
The terms with respect to $\langle \rq_\cdot, \partial_{xx} b_\cdot(t,x_\cdot,u) (\cdot,\cdot) \rangle_H$ can be handled similarly but with its corresponding norms for $\rq_\cdot$, and so is that for $ \partial_{xx}g_\cdot(t,x_\cdot,u) (\cdot,\cdot)$.

Then, by \eqref{zjdbrgi},  taking $u_{1,\gamma_i}=0$, $v_{1,\gamma_i}=0$, $u_{2,\gamma_i}=0$, $v_{2,\gamma_i}=0$, $i=1,2$, together with the above two estimates \eqref{tymh} and \eqref{dnrm},  for any $t\in[0,T], \xi_{1,\gamma_i}, \xi_{2,\gamma_i}\in L_{\mathcal{F}_t}^4(\Omega ; H)$, we get 
\begin{align}\label{lcyPgc}
    \lim_{\epsilon \to 0} \sup_{\mathsf{d}(\gamma_1 ,\gamma_2)\le \epsilon } \big|\mathbb{E}\left\langle P_{\gamma_1}(t) \xi_{1,{\gamma_1}}, \xi_{2,{\gamma_1}}\right\rangle_H - \mathbb{E}\left\langle P_{\gamma_2}(t) \xi_{1,{\gamma_2}}, \xi_{2,{\gamma_2}}\right\rangle_H  \big|=0. 
\end{align} 
Moreover, on the one hand, recalling \eqref{zjdbrgi} again and taking $t=0$, \eqref{lcyPgc} further implies that 
\begin{align}\label{loyPgcQ}
    &      \lim_{\epsilon \to 0} \sup_{\mathsf{d}(\gamma_1 ,\gamma_2)\le \epsilon } \Big|
    \mathbb{E} \int_0^T{\big\langle v_{1,{\gamma_1}}(s), \widehat{Q}_{\gamma_1}^{(t)}\left(\xi_{2,{\gamma_1}}, u_{2,{\gamma_1}}, v_{2,{\gamma_1}}\right)(s)\big\rangle}_H d s  
    \\  \notag 
    & \indent\indent\indent\quad
    +\mathbb{E} \int_0^T{\big\langle Q_{\gamma_1}^{(t)}\left(\xi_{1,{\gamma_1}}, u_{1,{\gamma_1}}, v_{1,{\gamma_1}}\right)(s), v_{2,{\gamma_1}}(s)\big\rangle}_H d s
    \\ \notag 
    &  \indent\indent\indent\quad
    -  \mathbb{E} \int_0^T  {\big\langle v_{1,{\gamma_2}}(s), \widehat{Q}_{\gamma_2}^{(t)}\left(\xi_{2,{\gamma_2}}, u_{2,{\gamma_2}}, v_{2,{\gamma_2}}\right)(s)\big\rangle}_H d s   
    \\  \notag 
    & \indent\indent\indent\quad 
    - \mathbb{E} \int_0^T{\big\langle Q_{\gamma_2}^{(t)}\left(\xi_{1,{\gamma_2}}, u_{1,{\gamma_2}}, v_{1,{\gamma_2}}\right)(s), v_{2,{\gamma_2}}(s)\big\rangle}_H  d s \Big| =0  .
\end{align}
Additionally, taking $v_{2,{\gamma_i}} = 0$, $i=1,2$,  it follows that 
\begin{equation}\label{lcyPgcQ}
\begin{aligned}
    \lim_{\epsilon \to 0} \sup_{\mathsf{d}(\gamma_1 ,\gamma_2)\le \epsilon } \Big| & \mathbb{E} \int_0^T{\big\langle v_{1,{\gamma_1}}(s), \widehat{Q}_{\gamma_1}^{(t)}\left(\xi_{2,{\gamma_1}}, u_{2,{\gamma_1}}, 0 \right)(s)\big\rangle}_H d s 
    \\
    & \indent\indent  - \mathbb{E} \int_0^T{\big\langle v_{1,{\gamma_2}}(s), \widehat{Q}_{\gamma_2}^{(t)}\left(\xi_{2,{\gamma_2}}, u_{2,{\gamma_2}}, 0 \right)(s)\big\rangle}_H d s \Big|  = 0 .  
\end{aligned}
\end{equation}
On the other hand, by \eqref{lcyPgc}, for any $t\in[0,T], \xi_{1 }, \xi_{2 }\in L_{\mathcal{F}_t}^4(\Omega ; H)$, it derives 
\begin{align*}
    \lim_{\epsilon \to 0} \sup_{\mathsf{d}(\gamma_1 ,\gamma_2)\le \epsilon } \big|\mathbb{E}{\left\langle (P_{\gamma_1}(t)-P_{\gamma_2}(t)) \xi_{1 }, \xi_{2 }\right\rangle}_H  \big|=0. 
\end{align*}
By taking $\xi_{2,k,(\ge)}= ( \langle (P_{\gamma_1}(t)-P_{\gamma_2}(t)) \xi_{1 },e_k\rangle )^{\frac{1}{3} } e_k \chi_{\{\langle (P_{\gamma_1}(t)-P_{\gamma_2}(t)) \xi_{1 },e_k\rangle \ge0\}}$ and $\xi_{2,k,(<)}= ( \langle (P_{\gamma_1}(t)-P_{\gamma_2}(t)) \xi_{1 },e_k\rangle )^{\frac{1}{3} } e_k \chi_{\{\langle (P_{\gamma_1}(t)-P_{\gamma_2}(t)) \xi_{1 },e_k\rangle < 0\}}$ by turns with $\{e_k\}_{k = 1}^{\infty} $ a sequence of orthonormal basis of $H$, it yields that for any $t\in[0,T]$,  
$
    \lim_{\epsilon \to 0} \sup_{\mathsf{d}(\gamma_1 ,\gamma_2)\le \epsilon } \mathbb{E}  {|(P_{\gamma_1}(t)-P_{\gamma_2}(t)) \xi_{1 } |}_H  =0 
$. 
It further deduces that for any $t\in[0,T]$, 
$
    \lim_{\epsilon \to 0} \sup_{\mathsf{d}(\gamma_1 ,\gamma_2)\le \epsilon }  | {\left\langle (P_{\gamma_1}(t)-P_{\gamma_2}(t)) \xi_{1 }, \xi_{2 }\right\rangle}_H | =0 \  a.s. 
$
Then we get 
\begin{align*}
    \lim_{\epsilon \to 0} \sup_{\mathsf{d}(\gamma_1 ,\gamma_2)\le \epsilon }  \int_{0}^{T} \mathbb{E}\ | {\left\langle (P_{\gamma_1}(t)-P_{\gamma_2}(t)) \xi_{1 }, \xi_{2 }\right\rangle}_H |  dt =0 .
\end{align*}
Since $\Gamma $ is a locally compact Polish space, we can choose a compact subset $\mathsf{S}^N \subset \Gamma $ for each $N\in \mathbb{N}_+$, such that $\lambda (\gamma \not\in \mathsf{S}^N)<\frac{1}{N} $. 
Further we can get a sequence of open neighborhoods $\{B(\gamma_i,\frac{1}{2N} )\}_{i=1}^{\varsigma_N}$ such that $\mathsf{S}^N \subset \bigcup_{i=1}^{\varsigma_N} B(\gamma_i,\frac{1}{2N} )$. 
And by partitions of unity (Lemma \ref{partunith}), there exists a sequence of continuous functions $\rho_i: \Gamma \to [0,1]\subset \mathbb{R} $ such that $\rho_i(\gamma )=0$ for $\gamma \not\in B(\gamma_i,\frac{1}{2N} ), i=1,\cdots,\varsigma_N$, and $\sum_{i = 1}^{\varsigma_N}\rho_i(\gamma )=1  $ for $\gamma \in \mathsf{S}^N$.
Now we choose the following  joint measurable process $ {P}^N_\gamma(\cdot)$ with $\gamma_i^*$ satisfying $\rho_i(\gamma_i^*)>0$: 
\[ {P}^N_\gamma (t):= \sum_{i = 1}^{\varsigma_N} P_{\gamma_i^*}(t) \rho_i(\gamma )\chi_{\{\gamma \in \mathsf{S}^N\}}.\] 
It is sufficient to prove 
\begin{align}\label{lcyPgc2} 
\lim_{N\to\infty}\int_\Gamma \int_{0}^{T} \mathbb{E}\ \big| {\left\langle (P^N_\gamma (t)-P_{\gamma}(t)) \xi_{1 }, \xi_{2 }\right\rangle}_H \big| dt \lambda (d\gamma ) =0 .
\end{align}
Indeed, 
\begin{align*}
   & \int_{0}^{T} \mathbb{E}\ \big| {\left\langle (P^N_\gamma (t)-P_{\gamma}(t)) \xi_{1 }, \xi_{2 }\right\rangle}_H \big| dt 
    \\
    & \quad 
    \le 
    \sum_{i = 1}^{\varsigma_N} \int_{0}^{T} \mathbb{E}\ \big| {\left\langle (P_{\gamma_i^*} (t)-P_{\gamma}(t)) \xi_{1 }, \xi_{2 }\right\rangle}_H \big|  dt \rho_i(\gamma )\chi_{\{\gamma \in \mathsf{S}^N\}}
    + \int_{0}^{T} \mathbb{E}\ \big| {\left\langle ( P_{\gamma}(t)) \xi_{1 }, \xi_{2 }\right\rangle}_H \big|  dt  \chi_{\{\gamma \not\in \mathsf{S}^N\}}
    \\
    & \quad 
    \le \sup_{\mathsf{d}(\gamma_1 ,\gamma_2)\le \frac{1}{N}  }  \int_{0}^{T} \mathbb{E}\ \big| {\left\langle (P_{\gamma_1}(t)-P_{\gamma_2}(t)) \xi_{1 }, \xi_{2 }\right\rangle}_H \big|  dt
    + \sup_{\gamma \in \Gamma } \int_{0}^{T} \mathbb{E}\ | {\left\langle ( P_{\gamma}(t)) \xi_{1 }, \xi_{2 }\right\rangle}_H |  dt \chi_{\{\gamma \not\in \mathsf{S}^N\}}. 
\end{align*}
Then 
\begin{align*} 
    &
    \lim_{N\to\infty}\int_\Gamma \int_{0}^{T} \mathbb{E}\ \big| {\left\langle (P^N_\gamma (t)-P_{\gamma}(t)) \xi_{1 }, \xi_{2 }\right\rangle}_H \big| dt \lambda (d\gamma ) 
    \\
&
\quad 
\le \lim_{N\to\infty} \Big\{ \sup_{\mathsf{d}(\gamma_1 ,\gamma_2)\le \frac{1}{N}  }  \int_{0}^{T} \mathbb{E}\ \big| {\left\langle (P_{\gamma_1}(t)-P_{\gamma_2}(t)) \xi_{1 }, \xi_{2 }\right\rangle}_H \big|  dt + \frac{C}{N} \Big\}
=0. 
\end{align*}
\indent  
Similarly, by \eqref{lcyPgcQ}, for $s, t \in[0, T],\   \xi_2 \in L_{\mathcal{F}_t}^4(\Omega ; H)$, $ u_2(\cdot) \in L_{\mathbb{F}}^2 (t, T ; L^4(\Omega ; H) )$, and $v_1(\cdot) \in L_{\mathbb{F}}^2 (t, T ; L^4(\Omega ; H) )$, it yields that 
\begin{align}\label{lcyPgcQ1}
    \lim_{\epsilon \to 0} \sup_{\mathsf{d}(\gamma_1 ,\gamma_2)\le \epsilon } \int_{t}^{T} \mathbb{E}\  \big|   {\big\langle v_{1}(s), \widehat{Q}_{\gamma_1}^{(t)}\left(\xi_{2,{\gamma_1}}, u_{2,{\gamma_1}}, 0 \right)(s)  - \widehat{Q}_{\gamma_2}^{(t)}\left(\xi_{2,{\gamma_2}}, u_{2,{\gamma_2}}, 0 \right)(s) \big\rangle}_H \big| dt 
      = 0  .  
\end{align}
Besides, by taking  $\xi_{1,\gamma_i}=0, \xi_{2,\gamma_i}=0$,  $u_{1,\gamma_i}=0$,  $u_{2,\gamma_i}=0$,  $i=1,2$ in \eqref{loyPgcQ}, and recalling that by the definition, $Q^{(t)}(0,0, \cdot)^*=\widehat{Q}^{(t)}(0,0, \cdot)$, it can be derived that 
\begin{align}\label{lcyPgcQ2}
    \lim_{\epsilon \to 0} \sup_{\mathsf{d}(\gamma_1 ,\gamma_2)\le \epsilon } \int_{t}^{T} \mathbb{E}\  \big| {\big\langle v_{1}(s), \widehat{Q}_{\gamma_1}^{(t)}\left(  0,0,v_{2,{\gamma_1}} \right)(s)  - \widehat{Q}_{\gamma_2}^{(t)}\left(  0 ,0, v_{2,{\gamma_2}} \right)(s) \big\rangle}_H \big| dt 
      = 0 . 
\end{align}
Applying \eqref{lcyPgcQ1}-\eqref{lcyPgcQ2}, by the arbitrariness of $v_1(\cdot), v_2(\cdot)$ $\in L_{\mathbb{F}}^2 (t, T ; L^4 (\Omega ; H) )$, $t\in[0,T]$ (with indicator functions involving both temporal variable and sample path),  
utilizing standard argument with the partitions of unity theorem (Lemma \ref{partunith}) as done in \eqref{lcyPgc2}, together with \eqref{lcyPgc2}, we obtain the measurable of the processes $\big(P_{\gamma}(\cdot),\big(Q_{\gamma}^{(\cdot)}, \widehat{Q}_{\gamma}^{(\cdot)}\big)\big)$ in the weak sense with respect to the uncertainty parameter. 
\end{proof}

Recalling the expression in \eqref{dpqSo}, $\dbS_\g$ is independent of $u $. Besides,  $y_\g^{\d u}$ is linear in $\d u$ as shown by \eqref{loybptg}, which further derives that the integrand in \eqref{bsyuiegz} is nonhomogeneous bilinear in $u $. 
Denote $\mathscr O_\gamma (\tilde v_1,\tilde v_2):= \mathbb{E} \int_0^T  {\big\langle y_\gamma(t;\tilde v_1), \dbS_\gamma(t)^* \tilde v_2\big\rangle}_{H_{1}} d t$ and $\mathscr O_\gamma (\tilde v):=\mathscr O_\gamma (\tilde v,\tilde v)$. 
To obtain the necessary optimality conditions, the following technical conditions are needed: 
\begin{equation*} \text{\bf (A8)} \quad 
    (i) \quad 
\begin{aligned}
    &
\bar{u}(\cdot)\in\dbL_{2,\dbF}^{1,2}(H_{1}), 
\dbS_\gamma(\cdot)^*\in  \dbL_{2,\dbF}^{1,2}(\cL_2(H_{1};H))\cap  L^{\infty}([0,T]\times\Omega;\cL_2(H_{1};H)),
\\
&
\cD_{\cdot}\dbS_\gamma(\cdot)^*\in  L^{2}(0,T;L^{\infty}([0,T]\times\Omega;\cL_2(H_{1};H))), \ \  \text{uniformly for}\ \gamma \in \Gamma.
\end{aligned}
\end{equation*}
\vspace*{-0.6cm}
\begin{align*} 
    (ii)\quad 
- \int_\Gamma \mathscr O_\gamma(\cdot) \lambda (d\gamma ) \  \text{is convex for any}\   \lambda \in \Lambda^{\bar{u}} \  \text{in}\  L^{4}_\dbF(0,T;H_{1}). \indent \ \   
\end{align*}
\vspace*{-0.7cm}
\begin{align*}
 \indent\indent  \ (iii) \quad 
\lim _{\epsilon  \rightarrow 0} \sup _{\mathrm{d}\left(\gamma, \gamma^{\prime}\right) \leq \epsilon } \int_{0}^{T} \int_{0}^{T} \mathbb{E}{\left|\mathcal{D}_{s} \mathbb{S}_{\gamma}(t)^* -\mathcal{D}_{s} \mathbb{S}_{\gamma^{\prime}}(t)^* \right|}_{\cal L_2(H_1;H)}^{2} d s d t=0 .\indent\indent\quad 
\end{align*}

The conditions (A8), presented following the literature of second order variational inequalities, are restrictive and essentially about the regularity of the coefficients just as those in (A3) and (A4), which indicates that it is natural to some extent from the aspect that more refined conclusions require more refined conditions. On the other hand, we are currently unsure how to weaken it.

For V-transposition solutions which involve more restrictive conditions compared with relaxed transposition solutions, similar analyses as \eqref{lcyPgcQ1} and \eqref{lcyPgcQ2} also hold. We present Lemma \ref{PQcl} with this slightly more general setting, and the variational inequalities later involve V-transposition solutions. 
Besides, by \eqref{cyxpq}, \eqref{dpqSo}, (A8) (i), \eqref{lcyPgc}, \eqref{lcyPgcQ1} and \eqref{lcyPgcQ2}, by simple interpolation as done in \eqref{dnrm}, it derives 
\begin{align}\label{clsg}
\lim_{\epsilon  \rightarrow 0} \sup_{\mathrm{d}\left(\gamma, \gamma^{\prime}\right) \leq \epsilon } \mathbb{E}\int_{0}^{T} {\left| \mathbb{S}_{\gamma}(t)^* - \mathbb{S}_{\gamma^{\prime}} (t)^* \right|}_{\cal L_2(H_1;H)} ds =0 . 
\end{align}
By standard argument with the partitions of unity theorem (Lemma \ref{partunith}) as done in \eqref{lcyPgc2}, we obtain the measurable of $\mathbb{S}_\cdot(\cdot)^*$. Similarly by (A8) (iii), it derives the measurable of $\mathcal{D}_{\cdot} \mathbb{S}_{\cdot}(\cdot)^*$.

\section{Applications}\label{lsecmair}
In this section, we present the main applications. 
Firstly, let us introduce the exact meaning of singular optimal control in the classical sense involved in this article. 
\begin{definition}\label{dscif}
    An optimal control $\bar u(\cdot)\in \cal U^2[0,T]$ is called a singular optimal control in the classical sense, if for any $\lambda \in \Lambda^{\bar u}$ ,   the following conditions hold for a.e. $(t,\omega )\in [0,T]\times \Omega $ and any $ v\in  U$:
    \begin{equation*}
        \left\{
        \begin{aligned}
    &\int_\Gamma {\big\langle \partial_u\dbH_\gamma [t], v-\bar u(t) \big\rangle}_{H_{1}} \lambda (d\gamma )=0 , 
    \\
    & \int_\Gamma {\big\langle \big(\partial_{uu} \dbH_\gamma [t]+\partial_u b_\gamma[t]^{*}P_\gamma (t)\partial_u b_\gamma[t]\big) (v-\bar u(t)), v-\bar u(t) \big\rangle}_{H_{1}}  \lambda (d\gamma )=0. 
        \end{aligned}
    \right.
    \end{equation*}
\end{definition}

The variational inequality in the integral form is provided in the following theorem.
\begin{theorem}\label{fmrt}
Assume that $x_0 \in L_{\mathcal{F}_0}^2(\Omega ; H)$ and $L_{\mathcal{F}_T}^2(\Omega)$ is separable. Let Assumptions (A1)-(A7) and (A8) (i) hold, and let $\bar{u}(\cdot) \in \mathcal{U}^4[0, T]$ be a singular optimal control in Definition \ref{dscif} and $\bar{x}_\gamma (\cdot)$ the corresponding optimal state. Then the following second order necessary conditions hold:  
\begin{equation}\label{bsyuiegz}
\begin{aligned}
    \inf_{u(\cdot)\in \cal U^4[0,T]} \Big\{ - \int_\Gamma \mathbb{E} \int_0^T &  {\big\langle y_\gamma(t), \dbS_\gamma(t)^* (u(t)-\bar u(t))\big\rangle}_{H_{1}} d t \lambda^* (u; d\gamma )\Big\} \geq 0 , 
\end{aligned}
\end{equation}  
where $y_\cdot(\cdot)$ is the solution to \eqref{lisef} corresponding to $\delta u(\cdot)=u(\cdot)-\bar{u}(\cdot)$ and  $\dbS_\cdot(\cdot)$ in  \eqref{dpqSo}. 
\end{theorem}
The proof of Theorem \ref{fmrt} is presented in Section \ref{lstjxo}.

The results in \eqref{bsyuiegz} can be improved to derive the following variational inequalities with common reference measure.
\begin{corollary}\label{ltbs}
Under the conditions of Theorem \ref{fmrt}, together with (A8) (ii), there exists $\lambda^* \in \Lambda^{\bar u}$, such that the following second order necessary conditions hold:  
    \begin{equation}\label{bsyuiegzj}
    \begin{aligned}
        \inf_{u(\cdot)\in \cal U^4[0,T]} \Big\{ - \int_\Gamma \mathbb{E} \int_0^T &  {\big\langle y_\gamma(t), \dbS_\gamma(t)^* (u(t)-\bar u(t))\big\rangle}_{H_{1}} d t \lambda^* (d\gamma )\Big\} \geq 0 .
    \end{aligned}
    \end{equation}
\end{corollary}

The proof of Corollary \ref{ltbs} can be found at the end of Section \ref{lstjxo}. 
Provided with the variational inequality \eqref{bsyuiegzj}, 
the pointwise second order necessary optimality conditions are presented in the following theorem. 
\begin{theorem}\label{fmrte}
    Assume that $x_0 \in L_{\mathcal{F}_0}^2(\Omega ; H)$ and $L_{\mathcal{F}_T}^2(\Omega)$ is separable. Let Assumptions (A1)-(A8) hold. 
    Let $\bar{u}(\cdot) \in \mathcal{U}^4[0, T]$ be a singular optimal control in Definition \ref{dscif}.
    Then the following pointwise second order necessary conditions hold for any $v \in U$ and for a.e. $\tau \in[0,T]$:  
    \begin{align*}
        &  \int_{\Gamma}  {\big\langle \partial_{u}a_\gamma[\tau ] (v -\bar{u}(\tau )) ,   \mathbb{S}_\gamma(\tau )^* (v -\bar{u}(\tau ) )\big\rangle}_H   \lambda^*( {d}\gamma)
            \\
        & \indent    +   \int_{\Gamma} {\big\langle   \partial_{u}b_\gamma[\tau ] (v -\bar{u}(\tau )) ,   \nabla\mathbb{S}_\gamma(\tau )^{*}(v -\bar{u}(\tau ))   \big\rangle}_H    \lambda^*( {d}\gamma) 
        \\
        & \indent  -  \int_{\Gamma} {\big\langle  \partial_{u}b_\gamma[\tau] (v -\bar{u}(\tau)) ,   \dbS_\gamma(\tau )^{*}  \nabla\bar{u}(\tau)  \big\rangle}_H \lambda^*( {d}\gamma) 
          \\
        & \le 0  \quad \dbP\mbox{-}a.s. 
    \end{align*}
where $\dbS_\cdot(\cdot)$ is defined in \eqref{dpqSo}, $\lambda^*$ is the common reference probability measure derived from weak convergence arguments, and $\nabla$ terms mean certain approximations of Malliavin derivatives. 
\end{theorem}
The proof of Theorem \ref{fmrte} is presented in Section \ref{slpwocs}.

\section{Integral type second order necessary optimality conditions}\label{lstjxo}

In this section, we firstly give the proof of Theorem \ref{fmrt} about the variational inequality in the integral form. 
Before performing Taylor form expansion on the cost functional and subsequent analysis of weak convergence, it is essential to firstly examine the following  regularities. 
\begin{lemma}
Given Assumptions (A1)-(A5), for any $\gamma \in \Gamma $, $\lambda \in \Lambda $,  and $u, u_1, u_2 \in \cal U^2[0,T]$, the following estimates for the cost functional hold: 
\begin{align}
    &\sup_{\gamma \in\Gamma }  |J(u(\cdot); \gamma)|<\infty , 
    \label{cdjb}
    \\ 
    &\lim_{\epsilon \to 0} \sup_{\mathsf{d}(\gamma ,\gamma')\le \epsilon }|J(u(\cdot); \gamma)-J(u(\cdot); \gamma')|=0 , 
    \label{cdjg}
\end{align}
\begin{subequations}
    \begin{equation}    \label{cdju}
    \lim_{|u_1-u_2|_{ \cal U^2[0,T]}\to 0} |J(u_1(\cdot); \lambda )-J(u_2(\cdot); \lambda )|=0 ,
\end{equation}
\begin{equation}    \label{cdju2}
    \lim_{|u_1-u_2|_{ \cal U^2[0,T]}\to 0}\sup_{\gamma \in\Gamma } |J(u_1(\cdot); \gamma  )-J(u_2(\cdot); \gamma )|=0 .
\end{equation}
\end{subequations}
\end{lemma}
\begin{proof}
    \eqref{cdjb} follows directly from the Assumptions (A1)-(A3) and the integrable of $x_\gamma , u$.  \eqref{cdjg} can be proved similarly as that in Lemma \ref{lcyyz} together with the integrable of $x_\g$. 
Note that from (A1) and (A3), for $u \in \cal U^2[0,T]$,  by \eqref{subssde} and Lemma \ref{mtuee}, we have $\sup _{\gamma \in \Gamma} \mathbb{E} \sup _{s \in[0, T]} |x_{\gamma}^{u}(s) |_{H}^{2}<\infty$.

Besides, by the mentioned given conditions and the first estimate in Lemma \ref{maests}, 
    \begin{align*}
    &|J(u_1(\cdot); \lambda )-J(u_2(\cdot); \lambda )|^2  
    \\
    &\indent  \le \Big(\int_\Gamma \dbE \Big[\int_{0}^{T} |g_\gamma (t,x_\g^{u_1},u_1) - g_\gamma (t,x_\g^{u_2},u_1)| + |g_\gamma (t,x_\g^{u_2},u_1) - g_\gamma (t,x_\g^{u_2},u_2)| dt  
    \\
    & \indent\indent+ |h_\g (x_\gamma^{u_1}(T)) - h_\g (x_\gamma^{u_2}(T))| \Big]\lambda (d\g) \Big)^2 
    \\
    &\indent \le C \Big(  \sup_{\g\in \Gamma }|x_\g^{u_1}-x_\g^{u_2}|^2_{C_\dbF([0,T];L^2(\O;H))} \sup_{\g\in \Gamma } \int_{0}^{T} \dbE \big(1+|x_\g^{u_1}|_H^2+|x_\g^{u_2}|_H^2+|u_1|_{H_1}^2\big) dt 
    \\
    & \indent\indent+ |u_1-u_2|^2_{L_{\dbF}^2 ([0,T];H_1)} \sup_{\g\in \Gamma }  \dbE \int_{0}^{T} \big( 1+|x_\g^{u_2}|^2_H+|u_1|^2_{H_1}+|u_2|^2_{H_1}\big) dt 
    \\
    & \indent\indent + \sup_{\g\in \Gamma }  \dbE |x_\g^{u_1}(T)-x_\g^{u_2}(T)|_H^2 \sup_{\g\in \Gamma }  \dbE \big( 1+|x_\g^{u_1}(T)|^2_H+|x_\g^{u_2}(T)|^2_H\big) \Big)
    \\
    & \indent  \to 0 \quad \mbox{as} \quad |u_1-u_2|_{ \cal U^2[0,T]}\to 0. 
    \end{align*}
This completes the proof of \eqref{cdju}, and it is a corollary to obtain \eqref{cdju2} from the above proof.  
\end{proof}

Define the set of feasible probability measures for any $u\in \cal U^2[0,T]$ as follows
\[\Lambda^{u}:=
       \Big\{\widetilde\lambda \in\Lambda \big|   J(u(\cdot); \widetilde\lambda ) = \sup_{\lambda\in\Lambda}  J(u(\cdot); \lambda)  \Big\}. 
\]
\begin{lemma}\label{jfkyu}
    Under Assumptions (A1)-(A6),  $\Lambda^{u}$ is nonempty for any $u\in \cal U^2[0,T]$. Besides, $\Lambda^{u}$ is convex and weakly compact for any $u\in \cal U^2[0,T]$. 
\end{lemma}
\begin{proof}
    By the definition of $\Lambda^u$, there exist $\{\lambda_N(d\gamma), N\in \dbN \}\subset \Lambda $, such that 
\[ \sup_{\lambda\in\Lambda}  J(u(\cdot); \lambda) -\frac{1}{N} \le J(u(\cdot); \lambda_N)= \int_{\Gamma } J(u(\cdot); \gamma ) \lambda_N(d\gamma)  \le  \sup_{\lambda\in\Lambda}  J(u(\cdot); \lambda).\]
By \eqref{cdjb}-\eqref{cdjg} and (A6), there exists $\widetilde\lambda \in\Lambda$ such that 
$J(u(\cdot); \widetilde\lambda ) = \sup_{\lambda\in\Lambda}  J(u(\cdot); \lambda)$ and $\Lambda^u $ is well-defined. The convexity of $\Lambda^u $ is obvious and the weakly compactness can be proved similarly. 
\end{proof}

Recall that $u^{\varepsilon }=\varepsilon (u-\bar u)+\bar u=\varepsilon u+(1-\varepsilon )\bar u$. 
Now we present 
\begin{proof}[The proof of Theorem \ref{fmrt}]
For any given uncertainty parameter $\gamma $, it holds that 
\begin{align} \label{tfceo}
J(u^\varepsilon (\cdot); \gamma)-J(\bar{u} (\cdot); \gamma) 
= \mathbb{E}\big[\int_0^T g_\gamma(t,x^\varepsilon_\gamma(t),u^\varepsilon(t))-g_\gamma(t,\bar {x} _\gamma(t),\bar u(t)) {d}t  +h_\gamma(x^\varepsilon_\gamma(T))-h_\gamma(\bar{x} _\gamma(T))\big]. 
\end{align}
Moreover, by Taylor's expansion formula, it can be verified that 
\begin{align}\label{tfceg}
&g_\gamma (t,x_\gamma^\e(t),u^\e(t)) - g_\gamma(t,\bar x_\gamma (t),
\bar u(t))  \notag 
\\
&\indent  = \big\langle \partial_x g_\gamma [t],  \delta x_\gamma^{\e}(t)
\big\rangle_H + \e\big\langle \partial_u g_\gamma [t],  \d u(t)
\big\rangle_{H_{1}} + \big\langle \partial_{xx} \tilde{g}^{\e}_\gamma [t] \delta x_\gamma^{\e}(t), \delta x_\gamma^{\e}(t) \big\rangle_H 
\\ 
&\indent \quad   +2\e \big\langle
\partial_{xu} \tilde{g}^{\e}_\gamma [t]\delta x_\gamma^{\e}(t), \d u(t)
\big\rangle_{H_{1}} + \e^2\big\langle 
\partial_{uu} \tilde{g}^{\e}_\gamma [t] \d u(t), \d u(t)
\big\rangle_{H_{1}}  \notag
\end{align}
and
\begin{equation}\label{tfceh}
h_\gamma(x_\gamma^\e(T)) - h_\gamma(\bar x_\gamma(T)) = \big\langle \partial_x h_\gamma (\bar x_\gamma(T)), \d x_\gamma^\e(T) \big \rangle_H 
+ \big\langle \partial_{xx}\tilde{h}_\gamma^{\e}(T)\d x_\gamma^\e(T), \d x_\gamma^\e(T) \big\rangle_H.
\end{equation}
Taking \eqref{tfceg} and \eqref{tfceh} into \eqref{tfceo}, together with \eqref{dxyzape}-\eqref{dxyzape4}, it derives
\begin{align} \label{tfceo1}
&J(u^{\varepsilon_n} (\cdot); \gamma)-J(\bar{u} (\cdot); \gamma) 
\\
&=  \mathbb E\int_0^T \Big[\e_{n} \big\langle \partial_x g_\gamma [t],
y_\gamma(t)\big\rangle_H + \frac{\e_{n}^2}{2} \big\langle 
\partial_x g_\gamma [t], z_\gamma(t)\big\rangle_H + \e_{n} \big\langle
\partial_u g_\gamma [t], \d u(t)\big\rangle_{H_{1}}  \notag 
 \\
& \indent \indent \quad + \frac{\e_{n}^2}{2}\(\big\langle \partial_{xx} g_\gamma [t] y_\gamma(t), y_\gamma(t)\big\rangle_H + 2\big\langle    \partial_{xu} g_\gamma [t] y_\gamma(t), \d u(t)\big\rangle_{H_{1}} +
 \big\langle \partial_{uu} g_\gamma [t]\d u(t), \d u(t)\big\rangle_{H_{1}} \) \Big]dt  \notag 
\\
& \indent \indent + \mathbb E\Big( \e_{n}\big\langle
\partial_x h_\gamma(\bar x_\gamma(T)), y_\gamma(T) \big\rangle_H  + 
\frac{\e_{n}^2}{2}\big\langle \partial_x h_\gamma(\bar x_\gamma(T)), z_\gamma(T) \big\rangle_H   \notag 
\\
& \indent \indent \indent  +  \frac{\e_{n}^2}{2}\big\langle \partial_{xx} h_\gamma(\bar x_\gamma(T)) y_\gamma(T),y_\gamma(T)\big\rangle_H \Big)  
+  o(\e_{n}^2).  \notag 
\end{align}
In general, we can not apply It\^o's formula to ${\langle \rp_\gamma , y_\gamma\rangle}_H $ or ${\langle \rp_\gamma , z_\gamma\rangle}_H $ directly in the sense of mild solution owing to the deficiency of enough regularities. However, by the definition of the transposition solution to \eqref{apHm}, it holds that 
\begin{align} \label{dfxd1}
&\mathbb E\big\langle \partial_x h_\gamma (\bar x_\gamma (T)), y_\gamma (T) \big\rangle_H
\\
&\indent =-\mathbb E\int_0^T\big[\big\langle \rp_\gamma(t),\partial_u a_\gamma[t]\d u(t) \big\rangle_H + \big\langle \rq_\gamma(t),\partial_u b_\gamma[t]\d u(t) \big\rangle_H +\big\langle \partial_x g_\gamma[t], y_\gamma (t) \big\rangle_H \big]dt \notag 
\end{align}
and
\begin{align}\label{dfxd2}
&\mathbb E\big\langle \partial_x h_\gamma (\bar x_\gamma (T)), z_\gamma (T) \big\rangle_H
\\
&\indent =-\mathbb E\int_0^T\big[ \big\langle \rp_\gamma(t),\partial_{xx}a_\gamma [t](y_\gamma(t),y_\gamma(t)) \big\rangle_H + 2\big\langle \rp_\gamma(t),\partial_{xu}a_\gamma [t](y_\gamma(t),\d u(t)) \big\rangle_H  \notag 
\\ 
&\indent\quad + \big\langle \rp_\gamma(t),\partial_{uu}a_\gamma [t](\d u(t),\d u(t)) \big\rangle_H + \big\langle \rq_\gamma(t),\partial_{xx}b_\gamma [t](y_\gamma(t),y_\gamma(t)) \big\rangle_H  \notag 
\\ 
&\indent\quad  + 2\big\langle \rq_\gamma(t),\partial_{xu}b_\gamma [t](y_\gamma(t),\d u(t)) \big\rangle_H + \big\langle \rq_\gamma(t),\partial_{uu}b_\gamma [t](\d u(t),\d u(t)) \big\rangle_H + \big\langle \partial_x g_\gamma[t], z_\gamma(t) \big\rangle_H \big]dt. \notag 
\end{align}
Similarly,  by the definition of the $V$-transposition solution to \eqref{lhsae}, it deduces 
\begin{align}\label{dfxd3}
&\mathbb E\big\langle \partial_{xx}h_\gamma (\bar x_\gamma(T))y_\gamma(T),y_\gamma(T) \big\rangle_H
\\
&\indent=-\mathbb E\int_0^T\big[\big\langle P_\gamma(t)y_\gamma(t),\partial_u a_\gamma[t]\d u(t) \big\rangle_H + \big\langle y_\gamma(t),P_\gamma(t)\partial_u a_\gamma[t]\d u(t) \big\rangle_H  \notag 
\\
&\indent\quad + \big\langle P_\gamma(t)\partial_x b_\gamma[t]y_\gamma(t), \partial_u b_\gamma[t]\d u(t) \big\rangle_H + \big\langle \partial_x b_\gamma[t]y_\gamma(t), P_\gamma(t)\partial_u b_\gamma[t]\d u(t) \big\rangle_H  \notag 
\\
&\indent\quad + \big\langle P_\gamma(t)\partial_u b_\gamma[t]\d u(t), \partial_u b_\gamma[t]\d u(t) \big\rangle_H + 2 \big\langle  \partial_u b_\gamma[t]^* Q_\gamma (t)y_\gamma (t), \d u(t) \big\rangle_{H_1}  \notag 
\\
&\indent\quad   - {\big\langle \partial_{xx}\mathbb H_\gamma [t]y_\gamma(t),y_\gamma(t) \big\rangle}_H \big]dt. \notag 
\end{align}
Utilizing \eqref{dfxd1}-\eqref{dfxd3} to replace the corresponding terms in \eqref{tfceo1}, we obtain 
\begin{align}\label{zczzdz}
    &  J(u^{\e_{n}}(\cdot);\gamma )-  J(\bar u(\cdot);\gamma ) 
       \\  
    & \indent   =- \mathbb E \int_0^T  \Big[ \e_{n}  \Big(\big \langle \rp_\gamma(t), \partial_u a_\gamma[t]  \d u(t) \big\rangle_H + \big\langle \rq_\gamma(t),\partial_u b_\gamma[t] \d u(t) \big\rangle_H - \big\langle \partial_u g_\gamma[t] , \d u(t)\big\rangle_{H_{1}} \Big)   \notag 
       \\ 
    &  \indent \indent  +  \frac{ \e_{n}^2 }{2}\Big( \big\langle \rp_\gamma(t), \partial_{u u}a_\gamma[t]\big(\d u(t),\d u(t)\big)\big\rangle_H + \big\langle \rq_\gamma(t), \partial_{u u}b_\gamma[t]\big(\d u(t),\d u(t)\big)\big\rangle_H   \notag 
    \\  
    &  \indent \indent  - \big\langle \partial_{u u}g_\gamma[t]\d u(t), \d u(t)\big\rangle_{H_{1}}\Big) 
       + \frac{ \e_{n}^2 }{2}\big\langle P_\gamma(t)\partial_u b_\gamma[t]\d u(t), \partial_u b_\gamma[t] \d u(t)\big\rangle_H   \notag 
       \\
    &   \indent \indent 
       +  \e_{n}^2 \Big( -\big\langle \partial_{x u}g_\gamma[t]y_\gamma(t),\d u(t)\big\rangle_{H_{1}} + \big\langle \rp_\gamma(t), \partial_{x u}a_\gamma[t](y_\gamma,\d u)\big\rangle_H 
       + \big\langle \rq_\gamma(t), \partial_{x u}b_\gamma[t] (y_\gamma, \d u) \big\rangle_H   \notag 
       \\
   &   \indent \indent 
       + \big\langle \partial_u a_\gamma[t]^*P_\gamma(t)y_\gamma(t), \d u(t)\big\rangle_{H_{1}}     + \big\langle \partial_u b_\gamma[t]^*P_\gamma(t)\partial_x b_\gamma[t]y_\gamma(t), \d u(t)\big\rangle_{H_{1}}   \notag 
 \\
 &  \indent \indent  
       + \big\langle  \partial_u b_\gamma[t]^* Q_\gamma (t)y_\gamma (t), \d u(t) \big\rangle_{H_1}   \Big) \Big]dt + o( \e_{n}^2 )  \notag 
       \\
    &  \indent  =-\mathbb E \int_0^T  \Big[ \e_{n}  \big\langle \partial_u\dbH_\gamma (t), \d u(t) \big\rangle_{H_{1}} + \frac{  \e_{n}^2  }{2}\big\langle \partial_{uu}\dbH_\gamma [t]\d u(t), \d u(t) \big\rangle_{H_{1}}    \notag 
    \\
    &  \indent \indent   + \frac{  \e_{n}^2  }{2}\big\langle \partial_u b_\gamma [t]^{*}P_\gamma (t)\partial_u b_\gamma [t]\d u(t), \d u(t)\big\rangle_{H_{1}} \Big]dt - \mathbb E \int_0^T  \e_{n}^2 \big\langle\big[\partial_{xu}\dbH_\gamma [t] + \partial_u a_\gamma [t]^* P_\gamma (t)   \notag 
    \\
    &   \indent \indent 
    + \partial_u b_\gamma [t]^*P_\gamma (t) \partial_x b_\gamma [t] + \partial_u b_\gamma[t]^* Q_\gamma (t) \big]y_\gamma (t), \d u(t)\big\rangle_{H_{1}} dt    + o( \e_{n}^2 )  \notag 
    \\
    &  \indent  =-\mathbb E \int_0^T  \Big[ \e_{n}  \big\langle \partial_u\dbH_\gamma (t), \d u(t) \big\rangle_{H_{1}} + \frac{  \e_{n}^2  }{2}\big\langle \partial_{uu}\dbH_\gamma [t]\d u(t), \d u(t) \big\rangle_{H_{1}}    \notag 
    \\
    &  \indent \indent   + \frac{  \e_{n}^2  }{2}\big\langle \partial_u b_\gamma [t]^{*}P_\gamma (t)\partial_u b_\gamma [t]\d u(t), \d u(t)\big\rangle_{H_{1}} \Big]dt  -   \mathbb{E} \int_0^T   \e_{n}^2  \big\langle y_\gamma(t),\dbS_\gamma(t)^* 
    \d  u(t) \big\rangle_{H_{1}} d t  + o( \e_{n}^2 ).  \notag 
\end{align}
Besides, take $\lambda^* \in \Lambda^{\bar u}$.  Then by the definition of $\Lambda^u$, we have 
\begin{align}
    &
J(\bar u(\cdot);\lambda^*) = \sup_{\lambda \in \Lambda}  J(\bar u(\cdot); \lambda )  =  \int_{\Gamma}\mathbb{E}\Big[\int_0^Tg_\gamma(t,\bar x_\gamma(t),\bar u(t)) {d}t +h_\gamma (\bar x_\gamma(T))\Big]\lambda^*( {d}\gamma), \notag
\\
&
\sup_{\lambda \in \Lambda} J(u^\varepsilon (\cdot); \lambda) \ge  J(u^\varepsilon (\cdot); \lambda^*)  = \int_{\Gamma} \mathbb{E}\Big[\int_0^T g_\gamma(t,x^\varepsilon_\gamma(t),u^\varepsilon(t)) {d}t  +h_\gamma(x^\varepsilon_\gamma(T)) \Big]\lambda^*( {d}\gamma).  \notag 
\end{align}
Subtracting separately the left and right hand sides of the  above equation and inequality, it deduces 
\begin{align*}
 \sup_{\lambda \in \Lambda} J(u^\varepsilon (\cdot); \lambda) - \sup_{\lambda \in \Lambda}  J(\bar u(\cdot); \lambda )\notag
&   \ge  \int_{\Gamma}  J(u^\varepsilon (\cdot); \gamma)-J(\bar{u} (\cdot); \gamma)\lambda^*( {d}\gamma) \notag 
\\
&  = \int_{\Gamma} \mathbb{E}\Big[\int_0^T g_\gamma(t,x^\varepsilon_\gamma(t),u^\varepsilon(t))-g_\gamma(t,\bar {x} _\gamma(t),\bar u(t)) {d}t \notag 
\\
& \indent \quad  +h_\gamma(x^\varepsilon_\gamma(T))-h_\gamma(\bar{x} _\gamma(T))\Big]\lambda^*( {d}\gamma) 
\\
&   \ge  \sup_{\lambda^* \in \Lambda^{\bar u}} \int_{\Gamma}  J(u^\varepsilon (\cdot); \gamma)-J(\bar{u} (\cdot); \gamma)\lambda^*( {d}\gamma) \notag ,
\end{align*}
where the last inequality results from the arbitrariness of $\lambda^*\in \Lambda^{\bar u}$. 
Moreover, together with \eqref{zczzdz}, and recalling Definition \ref{dscif}, we obtain 
\begin{align}\label{djdxc}
& \mathop{\overline{\lim}}_{n\to \infty} \frac{ 1}{\varepsilon_n^2} \big(\sup_{\lambda \in \Lambda }  J( u^{\varepsilon_n}(\cdot); \lambda ) - \sup_{\lambda \in \Lambda}  J(\bar u(\cdot); \lambda ) \big)
   \ge  - \int_\Gamma \mathbb{E} \int_0^T   {\big\langle y_\gamma(t), \dbS_\gamma(t)^* 
    (u(t)-\bar u(t))\big\rangle}_{H_{1}} d t \lambda^* (d\gamma ) .
\end{align}

On the other hand, 
for any $\varepsilon_n$ among the $\{\varepsilon_n\}$ in \eqref{tfceo1}, the corresponding optimal measure set  $\Lambda^{u^{\varepsilon_n}}$ is nonempty by Lemma \ref{jfkyu}. Taking one of its element $\lambda^{\varepsilon_n} \in \Lambda^{u^{\varepsilon_n}}$, it fulfills 
\begin{align*}
\sup_{\lambda \in \Lambda }  J( u^{\varepsilon_n}(\cdot); \lambda ) & =\int_{\Gamma} J( u^{\varepsilon_n}(\cdot); \gamma )  \lambda^{\varepsilon_n}( {d}\gamma)  
 = \int_{\Gamma}\mathbb{E}\Big[\int_0^T g_\gamma(t, x^{\varepsilon_n}_\gamma(t), u^{\varepsilon_n}(t)) {d}t +h_\gamma ( x^{\varepsilon_n}_\gamma(T))\Big]\lambda^{\varepsilon_n}( {d}\gamma) . \notag
\end{align*}
Besides, 
\begin{align*}
\sup_{\lambda \in \Lambda}  J(\bar u(\cdot); \lambda )  & \ge  
\int_{\Gamma} J(\bar u(\cdot); \gamma ) \lambda^{\varepsilon_n}( {d}\gamma)
= \int_{\Gamma}\mathbb{E}\Big[\int_0^Tg_\gamma(t,\bar x_\gamma(t),\bar u(t)) {d}t +h_\gamma (\bar x_\gamma(T))\Big]\lambda^{\varepsilon_n}( {d}\gamma). \notag
\end{align*}
Then subtract the left and right hand sides of the  above equation and inequality separately, 
\begin{align}\label{sjenb}
\sup_{\lambda \in \Lambda }  J( u^{\varepsilon_n}(\cdot); \lambda ) - \sup_{\lambda \in \Lambda}  J(\bar u(\cdot); \lambda )  \le  \int_{\Gamma}\mathbb{E}\Big[& \int_0^T g_\gamma(t, x^{\varepsilon_n}_\gamma(t), u^{\varepsilon_n}(t)) - g_\gamma(t,\bar x_\gamma(t),\bar u(t)) {d}t 
\\
& + h_\gamma ( x^{\varepsilon_n}_\gamma(T)) - h_\gamma (\bar x_\gamma(T))\Big]\lambda^{\varepsilon_n}( {d}\gamma). \notag
\end{align}
By taking a subsequence of  $\{\varepsilon_n\}$  if necessary, we can derive that 
\begin{align}\label{sjdjcf}
\mathop{\overline{\lim}}_{n\to \infty} \frac{ 1}{\varepsilon_n^2}\big(\sup_{\lambda \in \Lambda }  J( u^{\varepsilon_n}(\cdot); \lambda ) - \sup_{\lambda \in \Lambda}  J(\bar u(\cdot); \lambda )  \big) & = \lim_{n\to \infty}\frac{1}{\varepsilon_n^2}\big(\sup_{\lambda \in \Lambda }  J( u^{\varepsilon_n}(\cdot); \lambda ) - \sup_{\lambda \in \Lambda}  J(\bar u(\cdot); \lambda ) \big) 
\\
&  = \lim_{n\to \infty}\frac{1} {\varepsilon_n^2}\big( J( u^{\varepsilon_n}(\cdot); \lambda^{\varepsilon_n} ) -  J(\bar u(\cdot); \lambda^* ) \big), \label{riotvm}
\end{align}
where the limit in \eqref{riotvm} exists and is finite. 
Then 
\begin{align} \label{djslzh} 
&\mathop { \overline{\lim}}_{n\to \infty}\frac{1}{\varepsilon_n^2} \big(\sup_{\lambda \in \Lambda }  J( u^{\varepsilon_n}(\cdot); \lambda ) - \sup_{\lambda \in \Lambda} J(\bar u(\cdot); \lambda ) \big)  \notag
\\
& \indent \overset{\eqref{sjdjcf}}{=}\lim_{n\to \infty}\frac{1}{\varepsilon_n^2}\big(\sup_{\lambda \in \Lambda }  J( u^{\varepsilon_n}(\cdot); \lambda ) - \sup_{\lambda \in \Lambda}  J(\bar u(\cdot); \lambda ) \big)  \notag
\\
& \indent \overset{\eqref{sjenb}}{\le} \lim_{n\to \infty}\frac{1}{\varepsilon_n^2}  \int_\Gamma \big(
    J( u^{\varepsilon_n}(\cdot); \gamma  ) -   J(\bar u(\cdot); \gamma  ) \big)  \lambda^{\varepsilon_n} (d\gamma )  \notag 
    \\
& \indent \overset{\eqref{zczzdz}}{\le} \lim_{n\to \infty} \Big\{ \int_\Gamma \big( 
    - \mathbb E \int_0^T   \big\langle  y_\gamma(t), \dbS_\gamma(t)^*(u(t)-\bar u(t))\big\rangle_{H_{1}} dt   \big)\lambda^{\varepsilon_n} (d\gamma )  + \frac{1}{\varepsilon_n^2} \sup_{\gamma \in \Gamma } o( \e_{n}^2 ) \Big\} 
\notag
\\
&\indent \overset{}{=}  \int_\Gamma \big(\! 
- \mathbb E \int_0^T   \big\langle y_\gamma(t), \dbS_\gamma(t)^*(u(t)-\bar u(t))\big\rangle_{H_{1}} dt   \big)  \tilde \lambda (d\gamma ) ,  
\end{align}
where the last equality is from (A6), \eqref{cdjb},  \eqref{cdjg}, and \eqref{zczzdz}. 
On the other hand, for any $\lambda^{\varepsilon_n}  \in \{\lambda^{\varepsilon_n}\}_{n\in\mathbb{N} }$, it holds that 
\begin{align}\label{lycjue}
J(\bar u;\lambda^*)  
& \overset{\eqref{riotvm}}{=}  
\lim_{n\to\infty} \int_{\Gamma}  J(u^{\varepsilon_{n}} (\cdot); \gamma) \lambda^{\varepsilon_n} ( {d}\gamma) 
= \lim_{n\to\infty} \int_{\Gamma} \big( J(u^{\varepsilon_{n}} (\cdot); \gamma)-J(\bar{u} (\cdot); \gamma)+ J(\bar{u} (\cdot); \gamma)\big) \lambda^{\varepsilon_n} ( {d}\gamma)  \notag 
\\ 
& 
\overset{\eqref{cdju2}}{=} \lim_{n\to\infty}\int_{\Gamma} J(\bar{u} (\cdot); \gamma)\lambda^{\varepsilon_n} ( {d}\gamma) 
\overset{\eqref{cdjb}\eqref{cdjg} (A6)}{=} J(\bar{u} (\cdot); \tilde \lambda ). 
\end{align}
Note that \eqref{lycjue} derives  $\tilde\lambda\in \Lambda^{\bar u}$, where $\tilde \lambda$ is presented in \eqref{djslzh}.

Combining \eqref{djdxc} with \eqref{djslzh}, it can be derived that there is $\lambda^*\in \Lambda^{\bar u}$ (it is exactly the above $\lambda^*=\tilde \lambda $), such that 
\begin{align} 
&  \mathop  {\lim}_{n\to \infty}  \frac{1}{\varepsilon_n^2} \big( \sup_{\lambda \in\Lambda } J( u^{\varepsilon_n}(\cdot);  \lambda ) -  \sup_{\lambda \in\Lambda } J(\bar u(\cdot);  \lambda ) \big) 
= - \int_\Gamma  
  \mathbb E \int_0^T   \big\langle y_\gamma(t),  \dbS_\gamma(t)^* (u(t)-\bar u(t))\big\rangle_{H_{1}} dt \tilde \lambda (d\gamma ) .  \notag 
\end{align}

The calculus in \eqref{tfceo1} is dependent on the given perturbation $u(\cdot)$. Therefore, the above optimal uncertainty reference measure $\lambda^*$ is also dependent on the given  $u(\cdot)$, and the $\lambda^*=\tilde \lambda $ actually should be noted as  $\tilde \lambda (d\gamma ;u)$. 
By the optimality of $\bar u(\cdot)$, the proof is completed. 
\end{proof}

Now we can perform weak convergence arguments to present 
\begin{proof}[ The proof of Corollary \ref{ltbs}]
By the conclusions in the Theorem \ref{fmrt},  it derives 
\begin{align}\label{bdzh}
    \inf_{u(\cdot)\in \cal U^4[0,T]}\sup_{\tilde \lambda (d\gamma ) \in {\Lambda}^{\bar{u}}} \Big\{- \int_\Gamma \mathbb E \int_0^T   {\big\langle y_\gamma(t;\d u(t)), \dbS_\gamma(t)^* \d u(t)  \big\rangle}_{H_{1}} dt \tilde \lambda (d\gamma)\Big\}\ge0.
\end{align}
Utilizing \eqref{cdjb}-\eqref{cdju}, \eqref{bdzh}, (A8) (ii), and Lemma \ref{mimmaxthe}, it yields 
\begin{align*} 
    \sup_{\tilde \lambda (d\gamma) \in {\Lambda}^{\bar{u}}} \inf_{u(\cdot)\in \cal U^4[0,T]} \Big\{- \int_\Gamma \mathbb E \int_0^T   \big\langle y_\gamma(t;\d u(t)), \dbS_\gamma(t)^* \d u(t)  \big\rangle_{H_{1}} dt \tilde \lambda (d\gamma )\Big\}\ge0. 
\end{align*}
Then we can take $\{{\tilde \lambda}_N (d\gamma ) \in \Lambda^{\bar u}, N\in \dbN\}$, such that 
\[\inf_{u(\cdot)\in \cal U^4[0,T]} \Big\{- \int_\Gamma \mathbb E \int_0^T   \big\langle  y_\gamma(t;\d u(t)), \dbS_\gamma(t)^* \d u(t)  \big\rangle_{H_{1}} dt \tilde \lambda_N (d\gamma ) \Big\} \ge -\frac{1}{N} . \]
Recall Lemma \ref{jfkyu} and denote the weak limit of $\{{\tilde \lambda}_N (d\gamma ) \in \Lambda^{\bar u}, N\in \dbN\}$ by  $\tilde \lambda (d\gamma )$ again for simplicity,  then we have 
\[\inf_{u(\cdot)\in \cal U^4[0,T]} \Big\{ - \int_\Gamma  
\mathbb E \int_0^T   \big\langle  y_\gamma(t),  \dbS_\gamma(t)^* (u(t)-\bar u(t))\big\rangle_{H_{1}} dt \tilde \lambda (d\gamma ) \Big\} \ge 0.\]
Now the above $\tilde \lambda (d\gamma )$ is independent of the perturbation $u(\cdot)$. 
\end{proof}

\section{Pointwise second order optimality necessary conditions}\label{slpwocs}

Differentiation theorem of Lebesgue type plays essential role in obtaining pointwise necessary optimality conditions. For terms involving multiple deterministic and stochastic integral, there exist additional obstacles about the integrability order as shown in \cite{zhanghaisenzhangxu2015siam}, which also proposed a method with tools in Malliavin calculus to solve this kind of problem, and it is followed by \cite{luqizhanghaisenzhangxu2021siam,zhanghaisenzhangxu2017siam,zhanghaisenzhangxu2018siamreview}. The following proof also adheres to this framework.
Before performing the main proof, we give several technical lemmas. 

The following lemma allows approximating the bi-variables Malliavin derivative by the square integrable single variable function, utilizing the continuity of the Malliavin derivative in the neighborhood of $\{(t,t):t\in[0,T]\}$, while the Malliavin derivative helps overcome the difficulty in short of order when carrying out the Lesbegue differentiation theorem.  
\begin{lemma}\label{lxgz}
Let $\varphi_\gamma (\cdot)\in \dbL_{2,\dbF}^{1,2}(\wt  H)$ for any $\gamma \in \Gamma$, and $\lambda\in \Lambda$. Then, there exists a sequence of positive numbers $\{\epsilon_{n}\}_{n=1}^{\infty}$ such that $\epsilon_n\to 0^+$ as $n\to\infty$ and 
\begin{equation}\label{lxgzg}
\lim_{n\to  \infty}\frac{1}{\epsilon_{n}^2}\int_\Gamma \int_{\t}^{\t+\epsilon_{n}} \! \int_{\t}^{t}\dbE\big|\cal D_{s}\varphi_\gamma (t)-\nabla   \varphi_\gamma (s)\big|_{\wt H}^2dsdt \lambda(d\gamma )=0,\quad a.e.\ \t\in[0,T].
\end{equation}
\end{lemma}

\begin{proof}
    By exchange the order of integrals, displacement of intervals, substitution of variables, it can be proven by direct calculus 
    that for any $\gamma\in\Gamma$,
    \begin{align}\label{der-phi-nab-phi=0pf1}
    & \lim_{\epsilon\to0^+}  \! \frac{1}{\epsilon^2} \! \int_0^T \! \! \int_\tau^{\tau+\epsilon} \! \! \! \int_\tau^t\mathbb{E}   \left|\mathcal{D}_s\varphi_\gamma(t)-\nabla\varphi_\gamma(s)\right|_{\wt  H}^2  {d}s {d}t {d}\tau   \le \lim_{\epsilon\to0^+} \! \int_0^T   \! \! \sup_{t\in[\tau,\tau+\epsilon]\cap[0,T]} \! \mathbb{E}\left|\mathcal{D}_\tau\varphi_\gamma(t)-\nabla\varphi_\gamma(\tau)\right|_{\wt  H}^2  {d}\tau   \\
    &\indent =0,  \notag 
    \end{align}
    where the last equality can be deduced from the definition of $\nabla\varphi_\gamma(\cdot)$.
It further derives that 
    \begin{align}
    & \lim_{\epsilon\to0^+}\frac{1}{\epsilon^2}\int_\Gamma \int_0^T \int_\tau^{\tau+\epsilon}\int_\tau^t\mathbb{E}   \left|\mathcal{D}_s\varphi_\gamma(t)-\nabla\varphi_\gamma(s)\right|_{\wt  H}^2  {d}s {d}t {d}\tau \lambda( {d}\gamma)  \notag 
    \\
    &\indent  \overset{\eqref{der-phi-nab-phi=0pf1}}{ \le} \lim_{\epsilon\to0^+}\int_\Gamma\int_0^T \sup_{t\in[\tau,\tau+\epsilon]\cap[0,T]}\mathbb{E}\left|\mathcal{D}_\tau\varphi_\gamma(t)-\nabla\varphi_\gamma(\tau)\right|_{\wt  H}^2  {d}\tau \lambda( {d}\gamma)  \notag  
    \\
    &\indent   = \int_\Gamma\lim_{\epsilon\to0^+}\int_0^T \sup_{t\in[\tau,\tau+\epsilon]\cap[0,T]}\mathbb{E}\left|\mathcal{D}_\tau\varphi_\gamma(t)-\nabla\varphi_\gamma(\tau)\right|_{\wt  H}^2  {d}\tau \lambda( {d}\gamma)  \label{der-phi-nabruiuy-phi=0pf1b}  
 \overset{\eqref{der-phi-nab-phi=0pf1}}{=} 0,   
    \end{align}
    where the first equality in \eqref{der-phi-nabruiuy-phi=0pf1b} comes from the Lebesgue monotone convergence theorem, 
    since 
    \[\int_0^T\mathop{\sup}_{t\in[\tau,\tau+\epsilon ]\cap[0,T]}\mathbb{E}\left|\mathcal{D}_\tau\varphi_\gamma(t)-\nabla\varphi_\gamma(\tau)\right|_{\wt  H}^2 {d}\tau\]
    is non-negative and decreases as $\epsilon \to0^+$ for any $\gamma\in\Gamma$.
    By~\eqref{der-phi-nabruiuy-phi=0pf1b} and Fubini theorem, we obtain  
    \begin{equation*}
    \begin{aligned}
    & \lim_{\epsilon \to0^+}\frac{1}{\epsilon^2} \int_0^T \int_\Gamma\int_\tau^{\tau+\epsilon }\int_\tau^t\mathbb{E}   \left|\mathcal{D}_s\varphi_\gamma(t)-\nabla\varphi_\gamma(s)\right|_{\wt  H}^2  {d}s {d}t \lambda( {d}\gamma) {d}\tau =0.
    \end{aligned}
    \end{equation*}
    Then \eqref{lxgzg} follows.
\end{proof}

The following lemma is about a  version of multiple Lebesgue differentiation theorem under the model uncertainty context. 
\begin{lemma}\label{ldgsj}
Let $\phi_\gamma (\cdot),  \psi_\gamma (\cdot)\in  L_{\dbF}^{2}(0,T;H)$ for any $\g \in \Gamma $.  
Then for any $\lambda \in \Lambda $ and for a.e. $\t\in  [0,T)$,
\begin{align}\label{ldgsjg1}
& \lim_{\epsilon \to 0^+}\frac{1}{\epsilon^2}\int_\Gamma \dbE\int_{\t}^{\tau +\epsilon }\Big\langle\phi_\gamma (\t),  \int_{\t}^{t} e^{A(t-s)}\psi_\gamma (s)ds\Big\rangle_{H}  dt \lambda (d\g) =\frac{1}{2}\int_\Gamma \dbE\langle \phi_\gamma (\t),\psi_\gamma (\t)\rangle_{H}\lambda (d\g),
\\ 
& \label{ldgsjg2}  
\lim_{\epsilon\to 0^+}\frac{1}{\epsilon^2}\int_\Gamma \dbE\int_{\t}^{\tau +\epsilon } \Big\langle\phi_\gamma (t), \int_{\t}^{t}e^{A(t-s)}  \psi_\gamma (s)ds \Big\rangle_{H} dt \lambda (d\g) =\frac{1}{2}\int_\Gamma \dbE\langle\phi_\gamma (\t),\psi_\gamma (\t)\rangle_{H}\lambda (d\g).
\end{align}
\end{lemma}

\begin{proof}
From the fact that 
\begin{align*}
    \frac{1}{\epsilon^2} \int_\Gamma \mathbb{E}\int_\tau^{\tau+\epsilon}
    \Big\langle\phi_\gamma(\tau),\int_\tau^t\psi_\gamma(\tau) {d}s\Big\rangle_H {d}t\lambda ( {d}\gamma)
    =\frac{1}{2}\int_\Gamma \mathbb{E}{\big\langle\phi_\gamma(\tau),\psi_\gamma(\tau)\big\rangle}_H \lambda ( {d}\gamma) ,
\end{align*}
to prove  \eqref{ldgsjg1}, 
it is sufficient to prove 
\begin{align}\label{mllfa}
    \lim_{\epsilon\rightarrow0^+} \frac{1}{\epsilon^2} \Big| \int_\Gamma \mathbb{E}\int_\tau^{\tau+\epsilon}
 \Big\langle\phi_\gamma(\tau),\int_\tau^t e^{A(t-s)} \psi_\gamma(s)-\psi_\gamma(\tau) {d}s\Big\rangle_H {d}t\lambda ( {d}\gamma) \Big|=0.
\end{align}
Actually, 
\begin{align*}
 &   \lim_{\epsilon\rightarrow0^+} \frac{1}{\epsilon^2} \Big| \int_\Gamma \mathbb{E}\int_\tau^{\tau+\epsilon}
{ \Big\langle\phi_\gamma(\tau),\int_\tau^t e^{A(t-s)} \psi_\gamma(s)-\psi_\gamma(\tau) {d}s \Big\rangle}_H {d}t\lambda ( {d}\gamma) \Big|
 \\  & \indent 
 \le 
 \lim_{\epsilon\rightarrow0^+} \frac{1}{\epsilon^2}  \int_\Gamma 
 \int_\tau^{\tau+\epsilon} \mathbb{E} \Big(  |\phi_\gamma(\tau)|_H \int_\tau^t  \big| e^{A(t-s)}\psi_\gamma(s)-\psi_\gamma(\tau) \big|_H  {d}s \Big)  {d}t  \lambda ( {d}\gamma)  
 \\  &  \indent 
 \le 
    \lim_{\epsilon\rightarrow0^+} \frac{1}{ \epsilon^2} \sup_{\gamma \in \Gamma }(\mathbb{E}|\phi_\gamma(\tau)|^2_H)^{\frac{1}{2} }  \int_\tau^{\tau+\epsilon} (t-\tau)^{\frac{1}{2} } \Big( \int_\tau^t \mathbb{E} \int_\Gamma \big| e^{A(t-s)}\psi_\gamma(s)-\psi_\gamma(\tau) \big|^2_H \lambda ( {d}\gamma)  \mathrm{d}s \Big)^{\frac{1}{2} } {d}t   
    \\  &  \indent 
    \le  \lim_{\epsilon\rightarrow0^+} \frac{1}{ 2} \sup_{\gamma \in \Gamma } 
    (\mathbb{E}|\phi_\gamma(\tau)|^2_H)^{\frac{1}{2} }  \sup_{t\in (\tau,\tau+\epsilon)} \Big(\frac{1}{t-\tau } \int_\tau^t \mathbb{E} \int_\Gamma \big| e^{A(t-s)}\psi_\gamma(s)-\psi_\gamma(\tau) \big|^2_H \lambda ( {d}\gamma) \mathrm{d}s \Big)^{\frac{1}{2} } 
\end{align*}
where we have utilized the continuity of 
$ \int_\tau^t \mathbb{E} \int_\Gamma |e^{A(t-s)}\psi_\gamma(s)-\psi_\gamma(\tau) |^2_H \lambda ( {d}\gamma) \mathrm{d}s $
in $t$ due to the absolutely continuous of Lebesgue integration, and used the existence of the limit 
\begin{align}\label{lsjj}
    \lim_{t\to \tau^+}  \frac{1}{t-\tau } \int_\tau^t \mathbb{E} \int_\Gamma \big|e^{A(t-s)}\psi_\gamma(s)-\psi_\gamma(\tau)\big|^2_H \lambda ( {d}\gamma) {d}s  =0 \quad    a.e. \  \tau \in [0,T)
\end{align}
by Lebesgue differentiation theorem, both of which ensures the well-posedness of 
\begin{align}\label{mosulep}
    \sup_{t\in (\tau,\tau+\epsilon)} \Big\{\frac{1}{t-\tau } \int_\tau^t \mathbb{E} \int_\Gamma \big| e^{A(t-s)}\psi_\gamma(s)-\psi_\gamma(\tau) \big|^2_H \lambda ( {d}\gamma)  {d}s \Big\}.
\end{align}
Besides, the value of the term \eqref{mosulep} decreases as $\epsilon \to 0^+$, and the supremum in 
$\sup_{\gamma \in \Gamma } 
    (\mathbb{E}|\phi_\gamma(\tau)|^2_H)^{\frac{1}{2} }$ can be taken in a countable dense subset of $\Gamma$ to eliminate the common null measure set of $\tau \in [0,T]$ with proof by contradiction, owing to the continuity of $\phi_\gamma(\cdot)$ with respect to $\gamma$ in the moment sense with any high order from (A8). 
Together with \eqref{lsjj},  \eqref{mllfa} is proved, so is  \eqref{ldgsjg1}.

Now we are in the position to prove   \eqref{ldgsjg2}. Provided with  \eqref{ldgsjg1}, it is sufficient to prove  
\begin{align*}
    & \lim_{\epsilon\rightarrow0^+} \Big|\frac{1}{\epsilon^2} \int_\Gamma \mathbb{E}\int_\tau^{\tau+\epsilon}
     \Big\langle\phi_\gamma(t)-\phi_\gamma(\tau ),\int_\tau^t e^{A(t-s)} \psi_\gamma(s) {d}s\Big\rangle_H {d}t\lambda( {d}\gamma) \Big| =0. 
\end{align*}
By H\"{o}lder's inequality, Tonelli's theorem, and Lebesgue differentiation theorem, it holds that 
\begin{align*}
   & \lim_{\epsilon\rightarrow0^+} \Big|\frac{1}{\epsilon^2} \int_\Gamma \mathbb{E}\int_\tau^{\tau+\epsilon}
    \Big\langle\phi_\gamma(t)-\phi_\gamma(\tau ),\int_\tau^t e^{A(t-s)}\psi_\gamma(s) {d}s\Big\rangle_H {d}t\lambda ( {d}\gamma) \Big| 
    \\ & \quad 
    \le \lim_{\epsilon\rightarrow0^+}  \frac{1}{\epsilon^2} \int_\tau^{\tau+\epsilon} \int_\Gamma \mathbb{E} \Big( 
    \big|\phi_\gamma(t)-\phi_\gamma(\tau ) \big|_H  \Big| \int_\tau^t e^{A(t-s)} \psi_\gamma(s) {d}s \Big|_H\Big)  \lambda ( {d}\gamma)   {d}t 
    \\ & \quad 
    \le \lim_{\epsilon\rightarrow0^+}  \frac{1}{\epsilon^2} \Big(\int_\tau^{\tau+\epsilon} \int_\Gamma \mathbb{E} 
    \big|\phi_\gamma(t)-\phi_\gamma(\tau ) \big|_H^2 \lambda ( {d}\gamma)  {d}t\Big)^{\frac{1}{2}} \Big( \int_\tau^{\tau+\epsilon} \int_\Gamma \mathbb{E}  \Big| \int_\tau^t e^{A(t-s)} \psi_\gamma(s) {d}s \Big|_H^2 \lambda ( {d}\gamma)  {d}t \Big)^{\frac{1}{2} }
    \\ & \quad 
    \le \lim_{\epsilon\rightarrow0^+}  \frac{1}{ \sqrt{2} \epsilon} \Big(\int_\tau^{\tau+\epsilon} \int_\Gamma \mathbb{E} 
    \big|\phi_\gamma(t)-\phi_\gamma(\tau ) \big|_H^2 \lambda ( {d}\gamma)  {d}t\Big)^{\frac{1}{2}} 
    \Big( \int_\tau^{\tau+\epsilon} \int_\Gamma \mathbb{E} |\psi_\gamma(s)|_H^2 \lambda ( {d}\gamma) {d}s  \Big)^{\frac{1}{2} }
    \\ & \quad 
    =0. 
\end{align*}
\end{proof}

Note that by  (A8), for any $v\in U$, it follows  $\dbS_\gamma (t)^{*}(v-\bar{u}(t))\in \dbL_{2,\dbF}^{1,2}(H)$ and
\begin{equation}\label{cofp}
\dbS_\gamma (t)^{*}(v-\bar{u}(t)) =\dbE  \big[\dbS_\gamma (t)^{*}(v-\bar{u}(t))\big] +\int_{0}^{t} \dbE\big[\cal D_{s}\(\dbS_\gamma (t)^{*}(v-\bar{u}(t))\)  \big|  \cal F_{s}\big]  dW(s),\quad \dbP\mbox{-}a.s.
\end{equation}
Provided with the above preparations, we would perform  density arguments to present 
\begin{proof}[The proof of Theorem \ref{fmrt}]
Since $\{\cF_t\}_{t\ge0}$ is the natural filtration generated by the standard Brownian motion augmented by all the $\mathbb{P}$-null sets, one can find a sequence $\{F_l\}_{l=1}^{\infty}\subset\mathcal{F}_t, \forall t\in[0,T]$, such that for any $\O_t\in\mathcal{F}_t$, there exists a sub-sequence $\{F_{l_i}\}_{i=1}^{\infty}\subset\{F_l\}_{l=1}^{\infty}$,
satisfying
$\lim_{i\to\infty}\mathbb{P}(\O_t\vartriangle F_{l_i})=0$,
where $\O_t\vartriangle F_{l_i}:=(\O_t\backslash F_{l_i})\cup(F_{l_i}\backslash \O_t)$.
Denote by $\{t_i\}_{i=1}^\infty$ the sequence of rational numbers in $[0,T)$, by $\{v^k\}_{k=1}^\infty$ a dense subset of $U$. For any $i\in \mathbb{N}$, take a sequence $\{F_{ij}\}_{j=1}^{\infty}(\subset\mathcal{F}_{t_i})$ which can generate $\mathcal{F}_{t_i}$.
Fix $i,j,k$ arbitrarily.
For any $\tau\in[t_i,T)$ and $\epsilon \in(0,T-\tau)$, denote $E_\epsilon^i=[\tau,\tau+\epsilon )$,
\begin{equation*}
u_{ijk}^{\epsilon }(t,\omega)=
\left\{
\begin{aligned}
   v^k,  \indent\quad& (t,\omega)\in  E_\epsilon^i\times F_{ij},   \\
\bar{u}(t,\omega),  \indent &  (t,\omega)\in ([0,T]\times\Omega)\backslash (E_\epsilon^i\times F_{ij}),
\end{aligned}
\right.
\end{equation*}
and
\[v_{ijk}^{\epsilon }(t,\omega)=u_{ijk}^{\epsilon }(t,\omega)-\bar{u}(t,\omega) =(v^k-\bar{u}(t,\omega))\chi_{F_{ij}}(\omega)\chi_{E_\epsilon^i}(t),\quad (t,\omega)\in[0,T]\times\Omega.\]
Clearly $u_{ijk}^{\epsilon }(\cdot), v_{ijk}^{\epsilon }(\cdot)\in\mathcal{U}^4[0,T]$.
Denote by $y_{\gamma}^{ijk}$ the solution to \eqref{lisef} with $\d u(\cdot)$ replaced by $v_{ijk}^{\epsilon }(\cdot)$, recalling  \eqref{lbyxfg}, then actually
\begin{equation}\label{lbyxfe}
    \begin{aligned} 
    y_{\gamma,\epsilon }^{ijk}(t) = & \int_{0}^{t} e^{A(t-s)} \big( \partial_{x}a_\gamma[s] y_{\gamma,\epsilon }^{ijk} (s)+  \partial_{u}a_\gamma[s] (v^k-\bar{u}(s))\chi_{F_{ij}}(\omega)\chi_{E_\epsilon^i}(s) \big)ds 
    \\
    & + \int_{0}^{t}  e^{A(t-s)}  \big(\partial_{x}b_\gamma[s] y_{\gamma,\epsilon }^{ijk} (s)  + \partial_{u}b_\gamma[s] (v^k-\bar{u}(s))\chi_{F_{ij}}(\omega)\chi_{E_\epsilon^i}(s) \big) dW(s),
    \\
    & \quad \dbP\mbox{-}a.s.  \quad  t\in [0,T]. 
    \end{aligned}
\end{equation}
Moreover, replacing $\d u(\cdot)$ by $v_{ijk}^{\epsilon}(\cdot)$ in \eqref{bsyuiegz} results in
\begin{equation}\label{initial-second-necessary-condition-yijk}
\int_{\Gamma}  \mathbb{E}\int_\tau^{\tau+\epsilon } \big\langle  y_{\gamma,\epsilon }^{ijk}(t), \mathbb{S}_\gamma(t)^* (v^k-\bar{u}(t) )\big\rangle_H \chi_{F_{ij}} {d}t  \lambda^*( {d}\gamma)\le0.
\end{equation}
Then plug \eqref{lbyxfe} into \eqref{initial-second-necessary-condition-yijk} (in which $y_{\gamma,\epsilon }^{ijk}$ is regarded as zero for $t\in[0,\t)$) deriving 
\begin{equation}\label{nic123}
    I_1(\epsilon)+I_2(\epsilon )+I_3(\epsilon)  \le 0 ,
\end{equation}
where 
\begin{align*}
    I_1(\epsilon )&= \int_{\Gamma}  \mathbb{E}\int_\tau^{\tau+\epsilon } \Big\langle  \int_{\tau }^{t} e^{A(t-s)}   \partial_{x}a_\gamma[s] y_{\gamma,\epsilon }^{ijk} (s)   ds ,   \mathbb{S}_\gamma(t)^* (v^k-\bar{u}(t) )\Big\rangle_H \chi_{F_{ij}}(\omega) {d}t  \lambda^*( {d}\gamma),
\\
    I_2(\epsilon )&= \int_{\Gamma}  \mathbb{E}\int_\tau^{\tau+\epsilon } \Big\langle  \int_{\tau }^{t} e^{A(t-s)}   \partial_{u}a_\gamma[s] (v^k-\bar{u}(s))\chi_{F_{ij}}(\omega) ds ,  \\ & \indent\indent\indent\indent\indent\indent\indent\indent\indent\indent \mathbb{S}_\gamma(t)^* (v^k-\bar{u}(t) )\Big\rangle_H \chi_{F_{ij}}(\omega) {d}t  \lambda^*( {d}\gamma),
\\
    I_3(\epsilon) &= \int_{\Gamma}  \mathbb{E}\int_\tau^{\tau+\epsilon } \Big\langle  \int_{\tau }^{t}  e^{A(t-s)}\big(\partial_{x}b_\gamma[s] y_{\gamma,\epsilon }^{ijk} (s) + \partial_{u}b_\gamma[s] (v^k-\bar{u}(s))\chi_{F_{ij}}(\omega) \big) dW(s), \\
   & \indent\indent\indent\indent\indent \indent\indent\indent\indent\indent\indent\indent \mathbb{S}_\gamma(t)^* (v^k-\bar{u}(t) )\Big\rangle_H \chi_{F_{ij}}(\omega) {d}t  \lambda^*( {d}\gamma).
\end{align*}
For the term $ I_1(\epsilon )$, it can be estimated that 
\begin{align*}
\frac{1}{\epsilon^4} |I_1(\epsilon )|^2
& \le \frac{1}{\epsilon ^4}  \int_{\Gamma}  \mathbb{E}\int_\tau^{\tau+\epsilon } (t-\t) \int_{\tau }^{t} \big| e^{A(t-s)}   \partial_{x}a_\gamma[s] y_{\gamma,\epsilon }^{ijk} (s) \big|^2_H ds   {d}t  \lambda^*( {d}\gamma) 
\\ 
& \indent\indent\indent\indent\indent\indent\indent\indent \cdot \int_{\Gamma}  \mathbb{E}\int_\tau^{\tau+\epsilon } \big| \mathbb{S}_\gamma(t)^* (v^k-\bar{u}(t) )\big|^2_H  {d}t  \lambda^*( {d}\gamma)   
\\
& \le  C \sup_{\g\in\G}\sup_{t\in[0,T]} \dbE |y_{\gamma,\epsilon }^{ijk} (t)|_H^2 \cdot \frac{1}{\epsilon } \int_\tau^{\tau+\epsilon } \int_{\Gamma}  \mathbb{E} \big| \mathbb{S}_\gamma(t)^* (v^k-\bar{u}(t) )\big|^2_H  \lambda^*( {d}\gamma)  {d}t 
\\
& \to 0, \quad \epsilon  \to 0, \quad a.e.\ \t\in[t_i,T), 
\end{align*}
where the following convergence resulting  from Lemma \ref{mtuee} and (A8) is utilized: 
\begin{align}\label{syssxq}
\sup_{\g\in\G}\sup_{t\in[0,T]} \dbE |y_{\gamma,\epsilon }^{ijk} (t)|_H^2 \le C \dbE \int_{0}^{T}  |v^k-\bar{u}(s)|^2_{H_1}\chi_{F_{ij}}(\omega)\chi_{E_\epsilon^i}(s) ds \to 0 \quad \text{as} \ \epsilon \to 0. 
\end{align}
For the term $ I_2(\epsilon )$, by (A3) and Lemma \ref{ldgsj}, it follows that 
\begin{align*}
    \lim_{\epsilon \to0} \frac{1}{\epsilon ^2} I_2(\epsilon ) &=  \lim_{\epsilon \to0} \frac{1}{\epsilon ^2} \int_{\Gamma}  \mathbb{E}\int_\tau^{\tau+\epsilon } \Big\langle  \int_{\tau }^{t} e^{A(t-s)}   \partial_{u}a_\gamma[s] (v^k-\bar{u}(s))\chi_{F_{ij}}(\omega) ds ,  \\ & \indent\indent\indent\indent\indent\indent\indent\indent\indent\indent \mathbb{S}_\gamma(t)^* (v^k-\bar{u}(t) )\Big\rangle_H \chi_{F_{ij}}(\omega) {d}t  \lambda^*( {d}\gamma)  
    \\
    &  {=} \frac{1}{2} \int_{\Gamma}  \mathbb{E} \big(\big\langle \partial_{u}a_\gamma[\tau ] (v^k-\bar{u}(\tau )) ,   \mathbb{S}_\gamma(\tau )^* (v^k-\bar{u}(\tau ) )\big\rangle_H \chi_{F_{ij}}(\omega) \big) \lambda^*( {d}\gamma), \quad a.e.\ \t\in[t_i,T) .
\end{align*}
For the term $ I_3(\epsilon )$, by the Assumption (A8) and the Eq. \eqref{cofp}, it can be estimated that  
\begin{align*}
I_3(\epsilon) &= \int_{\Gamma}  \mathbb{E}\int_\tau^{\tau+\epsilon } \Big\langle  \int_{\tau }^{t}  e^{A(t-s)}\big(\partial_{x}b_\gamma[s] y_{\gamma,\epsilon }^{ijk} (s) + \partial_{u}b_\gamma[s] (v^k-\bar{u}(s))\chi_{F_{ij}}(\omega) \big) dW(s), \\
& \indent\indent   \dbE  \big[\dbS_\gamma(t)^{*}(v^k-\bar{u}(t))\big] +\int_{0}^{t} \dbE\big[\cal D_{s}\big(\dbS_\gamma(t)^{*}(v^k-\bar{u}(t))\big)  \big|  \cal F_{s}\big]  dW(s)\Big\rangle_H \chi_{F_{ij}}(\omega) {d}t  \lambda^*( {d}\gamma) 
   \\
&= \int_{\Gamma}  \mathbb{E}\int_\tau^{\tau+\epsilon } \Big\langle  \int_{\tau }^{t}  e^{A(t-s)}\big(\partial_{x}b_\gamma[s] y_{\gamma,\epsilon }^{ijk} (s) + \partial_{u}b_\gamma[s] (v^k-\bar{u}(s))\chi_{F_{ij}}(\omega) \big) dW(s), \\
& \indent\indent\indent\indent\indent\indent\indent\indent   \int_{0}^{t} \dbE\big[\cal D_{s}\big(\dbS_\gamma(t)^{*}(v^k-\bar{u}(t))\big)  \big|  \cal F_{s}\big]  dW(s)\Big\rangle_H \chi_{F_{ij}}(\omega) {d}t  \lambda^*( {d}\gamma) 
   \\
& = \int_{\Gamma} \int_\tau^{\tau+\epsilon } \int_{\tau }^{t}  \mathbb{E} \Big(\Big\langle  e^{A(t-s)}\big(\partial_{x}b_\gamma[s] y_{\gamma,\epsilon }^{ijk} (s) + \partial_{u}b_\gamma[s] (v^k-\bar{u}(s))\chi_{F_{ij}}(\omega) \big) , \\
& \indent\indent \indent\indent \indent\indent\indent\indent \indent   \cal D_{s}\big(\dbS_\gamma(t)^{*}(v^k-\bar{u}(t))\big)   \Big\rangle_H  \chi_{F_{ij}}(\omega) \Big) ds {d}t  \lambda^*( {d}\gamma)
\\
& =: I_4(\epsilon) + I_5(\epsilon) ,
\end{align*}
where 
\begin{align*}
& I_4(\epsilon) = \int_{\Gamma} \int_\tau^{\tau+\epsilon } \int_{\tau }^{t}  \mathbb{E} \big[{\big\langle  e^{A(t-s)} \partial_{x}b_\gamma[s] y_{\gamma,\epsilon }^{ijk} (s)   ,    \cal D_{s}\big(\dbS_\gamma(t)^{*}(v^k-\bar{u}(t))\big)   \big\rangle}_H  \chi_{F_{ij}}(\omega) \big] ds {d}t  \lambda^*( {d}\gamma), 
\\
& I_5(\epsilon) = \int_{\Gamma} \int_\tau^{\tau+\epsilon } \int_{\tau }^{t}  \mathbb{E} \big[\big\langle  e^{A(t-s)}  \partial_{u}b_\gamma[s] (v^k-\bar{u}(s))\chi_{F_{ij}}(\omega)  , \\
& \indent\indent \indent\indent\indent\indent\indent\indent \indent\indent \indent   \cal D_{s}\big(\dbS_\gamma(t)^{*}(v^k-\bar{u}(t))\big)   \big\rangle_H  \chi_{F_{ij}}(\omega) \big] ds {d}t  \lambda^*( {d}\gamma). 
\end{align*}
For the term $ I_4(\epsilon )$, we see that 
\begin{align*}
\frac{1}{\epsilon^4} |I_4(\epsilon )|^2 & \le  \frac{1}{\epsilon^4}  \int_{\Gamma} \int_\tau^{\tau+\epsilon } \int_{\tau }^{t}  \mathbb{E} \big|  e^{A(t-s)} \partial_{x}b_\gamma[s] y_{\gamma,\epsilon }^{ijk} (s) \big|^2_H   ds {d}t  \lambda^*( {d}\gamma) 
\\ & \indent\indent \cdot \int_{\Gamma} \int_\tau^{\tau+\epsilon } \int_{\tau }^{t} \mathbb{E} \big| \cal D_{s}\big(\dbS_\gamma(t)^{*}(v^k-\bar{u}(t))\big) \big|^2_H  ds {d}t \lambda^*( {d}\gamma)
\\
& \overset{\eqref{syssxq}}{\le}  C \sup_{\g\in\G}\sup_{t\in[0,T]} \dbE |y_{\gamma,\epsilon }^{ijk} (t)|_H^2 \cdot \frac{1}{\epsilon^2}  \int_{\Gamma} \int_\tau^{\tau+\epsilon } \int_{\tau }^{t} \mathbb{E} \big| \cal D_{s}\big(\dbS_\gamma(t)^{*}(v^k-\bar{u}(t))\big) \big|^2_H  ds {d}t \lambda^*( {d}\gamma) 
\\
& \to 0, \quad \epsilon  \to 0,  \quad a.e.\ \t\in[t_i,T).
\end{align*}
For the term $ I_5(\epsilon )$, by (A8), it deduces 
\begin{align}\label{dsvdsv}
\cal D_{s}\(\dbS_\gamma(t)^{*}(v-\bar{u}(t))\)= \cal D_{s} \dbS_\gamma(t)^{*}(v-\bar{u}(t)) -  \dbS_\gamma(t)^{*}\cal D_{s}(v-\bar{u}(t)).
\end{align}
Replacing one term in $I_5(\epsilon )$ using the above equality \eqref{dsvdsv}, it yields  
\[I_5(\epsilon )=:I_6(\epsilon ) - I_7(\epsilon ),\]
where 
\begin{align*}
    & I_6(\epsilon) = \int_{\Gamma} \int_\tau^{\tau+\epsilon } \int_{\tau }^{t}  \mathbb{E} \big[\big\langle  e^{A(t-s)}  \partial_{u}b_\gamma[s] (v^k-\bar{u}(s))\chi_{F_{ij}}(\omega)  , \\
    & \indent\indent \indent\indent\indent\indent\indent\indent \indent\indent \indent   \cal D_{s} \dbS_\gamma(t)^{*}(v^k-\bar{u}(t))   \big\rangle_H  \chi_{F_{ij}}(\omega) \big] ds {d}t  \lambda^*( {d}\gamma), 
    \\
    & I_7(\epsilon) = \int_{\Gamma} \int_\tau^{\tau+\epsilon } \int_{\tau }^{t}  \mathbb{E} \big[\big\langle  e^{A(t-s)}  \partial_{u}b_\gamma[s] (v^k-\bar{u}(s))\chi_{F_{ij}}(\omega)  , \\
    & \indent\indent \indent\indent\indent\indent\indent\indent \indent\indent \indent   \dbS_\gamma(t)^{*}\cal D_{s}(v^k-\bar{u}(t))  \big\rangle_H  \chi_{F_{ij}}(\omega) \big] ds {d}t  \lambda^*( {d}\gamma). 
\end{align*}
To approximate $\mathcal{D}_\cdot\mathbb{S}_\gamma(\cdot)^*$ by $\nabla\mathbb{S}_\gamma(\cdot)^*$ for any $\gamma\in\Gamma$, utilizing the continuity of $\mathcal{D}_\cdot\varphi_\gamma(\cdot)$ on some neighborhood of $\{(t,t)|t\in[0,T]\}$, we make a  simple interpolation and obtain 
\begin{align*}
    I_6(\epsilon) &= \int_{\Gamma} \int_\tau^{\tau+\epsilon } \int_{\tau }^{t}  \mathbb{E} \big[\big\langle  e^{A(t-s)}  \partial_{u}b_\gamma[s] (v^k-\bar{u}(s))\chi_{F_{ij}}(\omega)  , \\
    & \indent\indent \indent\indent\indent\indent\indent\indent \indent  \(\cal D_{s} \dbS_\gamma(t)^{*}-\nabla\mathbb{S}_\gamma(s)^{*}\)(v^k-\bar{u}(t))   \big\rangle_H  \chi_{F_{ij}}(\omega) \big] ds {d}t  \lambda^*( {d}\gamma)  
    \\
    & \indent \indent  + \int_{\Gamma} \int_\tau^{\tau+\epsilon } \int_{\tau }^{t}  \mathbb{E} \big[\big\langle  e^{A(t-s)}  \partial_{u}b_\gamma[s] (v^k-\bar{u}(s))\chi_{F_{ij}}(\omega)  , \\
    & \indent\indent \indent\indent\indent\indent\indent\indent \indent\indent \indent \nabla\mathbb{S}_\gamma(s)^{*}(v^k-\bar{u}(t))   \big\rangle_H  \chi_{F_{ij}}(\omega) \big] ds {d}t  \lambda^*( {d}\gamma)  
    \\
    & =: I_{6a}(\epsilon)+I_{6b}(\epsilon). 
\end{align*}
For the term $I_{6a}(\epsilon)$, it holds that 
\begin{align*}
\frac{1}{\epsilon^4} |I_{6a}(\epsilon )|^2 \le & C 
\frac{1}{\epsilon ^2}  \mathbb{E}  \int_{\Gamma} \int_\tau^{\tau+\epsilon } \int_{\tau }^{t}  |v^k-\bar{u}(s)|^2_{H_1} \cdot |v^k-\bar{u}(t)|^2_{H_1} ds {d}t  \lambda^*( {d}\gamma)  
\\
& \indent  \cdot \frac{1}{\epsilon ^2} \int_{\Gamma}  \mathbb{E}  \int_\tau^{\tau+\epsilon } \int_{\tau }^{t} \big|\cal D_{s} \dbS_\gamma(t)-\nabla\mathbb{S}_\gamma(s)\big|^2_{\cal L_2(H_1;H)} ds {d}t  \lambda^*( {d}\gamma)  . 
\end{align*}
By Lemma \ref{lxgz} with $\widetilde H={\cal L_2(H_1;H)}$, there exists a sequence  $\{\epsilon_{n}\}_{n=1}^{\infty}$ of positive numbers such that $\epsilon_n\to 0^+$ as $n\to\infty$ and 
\begin{equation*} 
\lim_{n\to  \infty}\frac{1}{\epsilon_{n}^2}\int_\Gamma \int_{\t}^{\t+\epsilon_{n}} \int_{\t}^{t}\dbE \big|\cal D_{s} \dbS_\gamma(t)-\nabla\mathbb{S}_\gamma(s)\big|^2_{\cal L_2(H_1;H)} dsdt \lambda^*(d\gamma )=0, \quad a.e.\ \t\in[t_i,T).
\end{equation*}
Together with the condition that  $\bar u(\cdot)\in \cal U^4[0,T]$, there exists a sequence  $\{\epsilon_{n}\}_{n=1}^{\infty}$ of positive numbers such that $\epsilon_n\to 0^+$ as $n\to\infty$ and 
\begin{equation*} 
    \lim_{n\to  \infty}\frac{1}{\epsilon_n^2} |I_{6a}(\epsilon_n )| =0,  \quad a.e.\ \t\in[t_i,T). 
\end{equation*}
Besides, for the term $I_{6b}(\epsilon)$, by (A3) and Lemma \ref{ldgsj},  for $a.e.  \t\in[t_i,T)$, we see that 
\begin{align*}
    \lim_{\epsilon \to0} \frac{1}{\epsilon ^2} I_{6b}(\epsilon) &=  \frac{1}{2} \int_{\Gamma}   \mathbb{E} \big(\big\langle   \partial_{u}b_\gamma[\tau ] (v^k-\bar{u}(\tau )) ,   \nabla\mathbb{S}_\gamma(\tau )^{*}(v^k-\bar{u}(\tau ))   \big\rangle_H  \chi_{F_{ij}}(\omega) \big) \lambda^*( {d}\gamma) .   
\end{align*}
The term $I_7(\epsilon)$ can be dealt with as that for $I_6(\epsilon)$.  There exists a subsequence of the above $\{\epsilon_{n}\}_{n=1}^{\infty}$ which is still denoted by itself for simplicity of positive numbers such that $\epsilon_n\to 0^+$ as $n\to\infty$, and 
\begin{align*}
    & \lim_{n\to  \infty}\frac{1}{\epsilon_n^2} I_7(\epsilon_n) = \frac{1}{2} \int_{\Gamma}  \mathbb{E} \big(\big\langle  \partial_{u}b_\gamma[\tau] (v^k-\bar{u}(\tau)) ,   \dbS_\gamma(\tau )^{*}  \nabla\bar{u}(\tau)  \big\rangle_H  \chi_{F_{ij}}(\omega) \big) \lambda^*( {d}\gamma), \quad   a.e.\ \t\in[t_i,T). 
\end{align*}

For the arbitrarily fixed $i,j,k$ in the beginning of the proof, from \eqref{nic123} and the above estimates for $I_1(\epsilon), I_2(\epsilon), I_3(\epsilon), I_4(\epsilon), I_5(\epsilon), I_6(\epsilon), I_7(\epsilon)$, there exists a Lebesgue measurable set $E_{i,j}^k\subset[t_i,T)$ satisfying $|E_{i,j}^k|=0$ and $|\cup_{i,j,k\in \dbN} E_{i,j}^k|=0$, such that 
\begin{align*}
&  \int_{\Gamma}  \mathbb{E} \big(\big\langle \partial_{u}a_\gamma[\tau ] (v^k-\bar{u}(\tau )) ,   \mathbb{S}_\gamma(\tau )^* (v^k-\bar{u}(\tau ) )\big\rangle_H \chi_{F_{ij}}(\omega) \big) \lambda^*( {d}\gamma)
    \\
& \indent    +   \int_{\Gamma}   \mathbb{E} \big(\big\langle   \partial_{u}b_\gamma[\tau ] (v^k-\bar{u}(\tau )) ,   \nabla\mathbb{S}_\gamma(\tau )^{*}(v^k-\bar{u}(\tau ))   \big\rangle_H  \chi_{F_{ij}}(\omega) \big) \lambda^*( {d}\gamma) 
\\
& \indent  -  \int_{\Gamma}  \mathbb{E} \big(\big\langle  \partial_{u}b_\gamma[\tau] (v^k-\bar{u}(\tau)) ,   \dbS_\gamma(\tau )^{*}  \nabla\bar{u}(\tau)  \big\rangle_H  \chi_{F_{ij}}(\omega) \big) \lambda^*( {d}\gamma) 
  \\
& \le 0, \quad   \forall\tau\in[t_i,T)\backslash \cup_{i,j,k\in \dbN} E_{i,j}^k. 
\end{align*}
Recall that   $\{F_{ij}\}_{j=1}^{\infty}(\subset\mathcal{F}_{t_i})$ are taken  generating $\mathcal{F}_{t_i}$ and $\{\cF_t\}_{t\ge0}$ is the natural filtration, together with the density of $\{v^k\}_{k=1}^\infty$ in $U$ and by the Assumption (A8) and Fubini's theorem, it yields 
\begin{align*}
    &  \int_{\Gamma}  \big\langle \partial_{u}a_\gamma[\tau ] (v -\bar{u}(\tau )) ,   \mathbb{S}_\gamma(\tau )^* (v -\bar{u}(\tau ) )\big\rangle_H   \lambda^*( {d}\gamma)
        \\
    & \indent    +   \int_{\Gamma} \big\langle   \partial_{u}b_\gamma[\tau ] (v -\bar{u}(\tau )) ,   \nabla\mathbb{S}_\gamma(\tau )^{*}(v -\bar{u}(\tau ))   \big\rangle_H    \lambda^*( {d}\gamma) 
    \\
    & \indent  -  \int_{\Gamma} \big\langle  \partial_{u}b_\gamma[\tau] (v -\bar{u}(\tau)) ,   \dbS_\gamma(\tau )^{*}  \nabla\bar{u}(\tau)  \big\rangle_H \lambda^*( {d}\gamma) 
      \\
    & \le 0,  \quad \dbP\mbox{-}a.s.\quad  \forall(\tau,v)\in([t_i,T)\backslash \cup_{i,j,k\in\mathbb{N}}E_{i,j}^k)\times U. 
\end{align*}
Recall the arbitrariness of $t_i \in[0,T)$, the proof is completed.
\end{proof}

\section{Motivating examples}\label{slme}

In \cite{jing23+}, we present a synthetic example to demonstrate the effectiveness of second order necessary conditions for optimal stochastic control under model uncertainty. These conditions help to narrow down the set of candidate optimal controls, especially when the maximum principle fails to provide useful information.
In detail, with one dimensional setting, we consider the case $[0,T]=[0,1], U=[-1,1]$, $\Gamma=\{1,2\}$, $\Lambda=\{\lambda^ \mathfrak{i} | \mathfrak{i} \in[0,1]\}$, and $\lambda^\mathfrak{i} (\{1\})= \mathfrak{i} , \lambda^\mathfrak{i} (\{2\})=1- \mathfrak{i} $.
The controlled systems corresponding to uncertainty parameter $\gamma=1,2$ are respectively
\begin{equation*}
\left\{
\begin{aligned}
&\mathrm{d}x_1(t)= u(t)\mathrm{d}t+u(t)\mathrm{d}W(t), \quad t\in[0,1],         \\
&x_1(0)=0,
\end{aligned}
\right.
\indent  \text{and} \indent  
\left\{
\begin{aligned}
&\mathrm{d}x_2(t)=u(t)\mathrm{d}t, \quad t\in[0,1],           \\
&x_2(0)=0.
\end{aligned}
\right.
\end{equation*}
Take 
    $f_1(t,x_1(t),u(t)) = \frac{1}{2}|u(t)|^2$, 
    $f_2(t,x_2(t),u(t)) =  \frac{1}{4}|u(t)|^4 $,  
    $h_1(x_1(1))=-\frac{1}{2}|x_1(1)|^2$, $h_2(x_2(1))=-\frac{1}{2}|x_2(1)|^2$. 
Then the robust cost functional is given by
\begin{align*}
  J(u(\cdot);\lambda^*) =\sup_{ \mathfrak{i} \in[0,1]}\Big\{   \mathfrak{i} \Big(\frac{1}{2} \mathbb{E}\int_0^1|u(t)|^2\mathrm{d}t-\frac{1}{2}\mathbb{E}|x_1(1)|^2\Big) 
  +(1- \mathfrak{i} ) \Big(\frac{1}{4}\mathbb{E}\int_0^1|u(t)|^4\mathrm{d}t-\frac{1}{2}\mathbb{E}|x_2(1)|^2\Big) \Big\}.
\end{align*}
It was shown that in this example, the function $\bar{u}(\cdot) \equiv 0$ defined on $t\in[0,1]$ can not be excluded from the candidate set of optimal controls satisfying classical maximum principle. 
However, by taking $u(t)\equiv1, t\in[0,1]$, it can be verified that $J( {u}(\cdot))\le -\frac{1}{4}< 0 =J(\bar{u}(\cdot))$, which indicates that $\bar{u}(\cdot) \equiv 0$ on $t\in[0,1]$ is not one of the optimal controls. 
This fact can be concluded by our results of second order necessary conditions.

Moreover, compared with \cite{jing23+}, the infinite dimensional setting in this paper can be adapted to obviously more general cases, such as  controlled  stochastic Schr\"odinger equations, stochastic Cahn-Hilliard equations, stochastic Korteweg de Vries equations,  and others. 

Currently, it is widely acknowledged that optimal controls in practice are determined through numerical calculations. The theoretical results serve as a foundation for such computations. 
Now we give another example to demonstrate that there are cases satisfying the assumptions imposed for the integral type necessary conditions, which is a reconstruction to some extent, but we still sketch some of the details for the convenience of readers. And actually numerous models satisfying the conditions assumed in this paper can be generated in the similar manner.

Let $H=L^{2}(0,1) \times L^{2}(0,1)$, $U=H_{0}^{1}(0,1)\times B_{H_{0}^{1}(0,1)}\subset H_1=H_{0}^{1}(0,1)\times H_{0}^{1}(0,1)$, where $B_{H_{0}^{1}(0,1)}$ is the closed unit ball in $H_{0}^{1}(0,1)$. 
Take $\Gamma=\{1,2\}$, $\Lambda=\{\lambda^ \mathfrak{i} | \mathfrak{i} \in[0,1]\}$, where  $\lambda^\mathfrak{i} (\{1\})= \mathfrak{i} , \lambda^\mathfrak{i} (\{2\})=1- \mathfrak{i} $.
Define $A: D(A)= [H^{2}(0,1) \cap H_{0}^{1}(0,1) ] \times [H^{2}(0,1) \cap H_{0}^{1}(0,1) ] \rightarrow L^{2}(0,1) \times L^{2}(0,1) $ and $A f=\partial_{\xi \xi} f$.
Consider the following coefficients 
\begin{align*}
    a_\gamma (t, x, u)=
\begin{pmatrix}
    a_{11,\gamma } & a_{12,\gamma} \\
    0 & a_{22,\gamma}
\end{pmatrix}
\begin{pmatrix}
    \varphi_{1} \\
    \varphi_{2}
\end{pmatrix}
+
\begin{pmatrix}
    u_{1} \\
    0
\end{pmatrix}, \quad 
    b_\gamma(t, x, u)= 
    \begin{pmatrix}
    b_{11,\gamma} & b_{12,\gamma} \\
    0 & 0
\end{pmatrix}
     \begin{pmatrix}
    \varphi_{1} \\
    \varphi_{2}
\end{pmatrix}
    +
    \begin{pmatrix}
    0 \\
    u_{2}^{2}
\end{pmatrix}, 
\end{align*}
where $a_{11,\gamma}, a_{12,\gamma}, a_{22,\gamma}, b_{11,\gamma}, b_{12,\gamma} \in L^{\infty}(0, T), \gamma \in \Gamma $, $x=\begin{pmatrix}
    \varphi_{1},
    \varphi_{2}
\end{pmatrix}^{\top},
u=
\begin{pmatrix}
    u_{1},
    u_2
\end{pmatrix}^{\top}$. 
Under the above setting, the state equation becomes 
\begin{equation*}
    \left\{
\begin{aligned}
&d \varphi_{1,\gamma }=\left(\partial_{\xi \xi} \varphi_{1,\gamma }+a_{11,\gamma } \varphi_{1,\gamma }+a_{12,\gamma } \varphi_{2,\gamma }+u_{1}\right) d t  
\\
& \indent \indent\indent\indent\indent +\left(b_{11,\gamma } \varphi_{1,\gamma }+b_{12,\gamma } \varphi_{2,\gamma }\right) d W(t)  &  \text{in}\  (0, T] \times(0,1), 
\\
&
d \varphi_{2,\gamma }=\left(\partial_{\xi \xi} \varphi_{2,\gamma }+a_{22,\gamma } \varphi_{2,\gamma }\right) d t+u_{2}^{2} d W(t)   &   \text{in}\  (0, T] \times(0,1), 
\\
&
\varphi_{1,\gamma }(t, 0)=\varphi_{1,\gamma }(t, 1)=\varphi_{2,\gamma }(t, 0)=\varphi_{2,\gamma }(t, 1)=0   &   \text{in}\  (0, T], 
\\
&
\varphi_{1,\gamma }(0, \xi)=\phi (\xi)  & \text{on}\  (0,1), 
\\
&
\varphi_{2,\gamma }(0, \xi)=0  &   \text{on}\  (0,1),
\end{aligned}
\right.
\end{equation*}
where 
$\phi =\sum_{n=1}^{\infty} a_{n } \sqrt{2} \sin n \pi \xi \in L^{2}(0,1)$. 
Further, take the cost functional with respect to  parameters $\gamma =1,2$ as follows:
\begin{align*}
     {J}(u;1)= \frac{1}{2}  \mathbb{E} \langle\varphi_{1,1}(T), \varphi_{1,1}(T) \rangle_{L^{2}(0,1)} ,
\quad 
     {J}(u;2)=   \frac{1}{2}  \mathbb{E} \langle\varphi_{2,2}(T), \varphi_{2,2}(T) \rangle_{L^{2}(0,1)}. 
\end{align*}
Then the corresponding cost functional ${J}(u;\lambda^*)$ has the following form
\begin{align*}
\sup_{\lambda^\mathfrak{i} \in \Lambda } \int_{\Gamma }  {J}(u;\gamma ) \lambda^\mathfrak{i} (d \gamma )
= \sup_{\mathfrak{i} \in[0,1]} \frac{1}{2}\big\{ \mathfrak{i} \mathbb{E} \langle\varphi_{1,1}(T), \varphi_{1,1}(T) \rangle_{L^{2}(0,1)} + (1-\mathfrak{i}) \mathbb{E} \langle\varphi_{2,2}(T), \varphi_{2,2}(T) \rangle_{L^{2}(0,1)} \big\}. 
\end{align*}

To apply the $V$-transposition solution, take $V=H_{0}^{1}(0,1)\times H_{0}^{1}(0,1)$ and $V^*=H^{-1}(0,1)\times H^{-1}(0,1)$. 
The two emdedding operators 
$H_0^1(0,1)\hookrightarrow L^2(0,1)\hookrightarrow H^{-1}(0,1)$
are Hilbert-Schmidt, which can be proved by taking feasible orthogonal basis. 

Take $f(t, \xi)=e^{\int_{0}^{t} a_{11,1} d s-\frac{1}{2} \int_{0}^{t} b_{11,1}^{2} d s+\int_{0}^{t} b_{11,1} d W(s)} \sum_{n=1}^{\infty} (-\frac{a_{n}}{T} e^{- (n^{2} \pi^{2}+1 / 2 ) t+W(t)}) \sqrt{2} \sin n \pi \xi$. 
It can be verified that 
$f \in L_{\mathbb{F}}^{2} (0, T ; H_{0}^{1}(0,1) )$ and further 
$f \in \mathbb{L}_{2, \mathbb{F}}^{1,2} (H_{0}^{1}(0,1) )$.
Besides, it can be verified that 
$(
    u_{1},   
    u_2
)^{\top}
=
 (
    f,  
    0
 )^{ \top}$
is an optimal control, since it  satisfies  
$\varphi_{1,1}(T)=0$ and $\varphi_{2,2}(T)=0$, which results in $ {J}((
    f,  
    0
)^{\top};\lambda^*)=0$, $\lambda^*\in \Lambda $. 

From the above formulation, it is obvious that 
$
\partial_x a_{\gamma }= \begin{pmatrix} 
    a_{11,\gamma } & a_{12,\gamma} \\
    0 & a_{22,\gamma}
    \end{pmatrix} 
    \in L_{\mathbb{F}}^{\infty} (0, T ; \mathcal{L}_{H V } ), 
    \partial_x b_{\gamma }= 
    \begin{pmatrix} 
    b_{11,\gamma} & b_{12,\gamma} \\
    0 & 0
    \end{pmatrix} 
    \in L_{\mathbb{F}}^{\infty} (0, T ; \mathcal{L}_{H V } ). 
$
The first order adjoint equations for any $\gamma \in \Gamma $ are 
\begin{equation*}
\left\{
\begin{aligned}
& d \rp_{1,\gamma}=- (\partial_{x x} \rp_{1,\gamma}+a_{11,\gamma} \rp_{1,\gamma}+b_{11,\gamma} \rq_{1,\gamma} ) d t+\rq_{1,\gamma} d W(t)   & \text{in}\ [0, T) \times(0,1), 
\\
& 
d \rp_{2,\gamma}=- (\partial_{x x} \rp_{2,\gamma}+a_{12,\gamma} \rp_{1,\gamma}+a_{22,\gamma} \rp_{2,\gamma}+b_{12,\gamma} \rq_{1,\gamma} ) d t+\rq_{2,\gamma} d W(t) & \text{in}\  [0, T) \times(0,1), 
\\
& 
\rp_{1,\gamma}(\cdot, 0)=\rp_{1,\gamma}(\cdot, 1)=\rp_{2,\gamma}(\cdot, 0)=\rp_{2,\gamma}(\cdot, 1)=0     & \text{on}\  [0, T), 
\\
&
\rp_{1,\gamma}(T, \cdot)=\rp_{2,\gamma}(T, \cdot)=0   & \text{in}\  (0,1) .
\end{aligned}
\right.
\end{equation*}
It can be deduced 
$(\rp_{1,\gamma},\rq_{1,\gamma})\equiv0$ and $(\rp_{2,\gamma},\rq_{2,\gamma})\equiv0$, 
which further implies that 
\begin{align*}
& \mathbb{H} (t, \begin{pmatrix}
    \bar\varphi_{1,\gamma} \\
    \bar\varphi_{2,\gamma}
\end{pmatrix}, \begin{pmatrix}
    f \\
    0
\end{pmatrix}, \begin{pmatrix}
    \rp_{1,\gamma} \\
    \rp_{2,\gamma}
\end{pmatrix}, \begin{pmatrix}
    \rq_{1,\gamma} \\
    \rq_{2,\gamma}
\end{pmatrix}) 
\\ 
& \indent  = \Big\langle \begin{pmatrix}
    \rp_{1,\gamma} \\
    \rp_{2,\gamma}
\end{pmatrix}, \begin{pmatrix}
    a_{11,\gamma } & a_{12,\gamma} \\
    0 & a_{22,\gamma}
\end{pmatrix}
\begin{pmatrix}
    \bar\varphi_{1,\gamma} \\
    \bar\varphi_{2,\gamma}
\end{pmatrix}
+
\begin{pmatrix}
    f \\
    0
\end{pmatrix} \Big\rangle + \Big\langle \begin{pmatrix}
    \rq_{1,\gamma} \\
    \rq_{2,\gamma}
\end{pmatrix}, \begin{pmatrix}
    b_{11,\gamma} & b_{12,\gamma} \\
    0 & 0
\end{pmatrix}
     \begin{pmatrix}
        \bar\varphi_{1,\gamma} \\
        \bar\varphi_{2,\gamma}
\end{pmatrix} \Big\rangle 
 \equiv0. 
\end{align*}
For any $\lambda \in \Lambda^{\bar u}, \bar u={\begin{pmatrix} f, 0 \end{pmatrix}}^{\top}$, the following conditions hold for $a.e. \ (t,\o)\in [0,T]\times \O$ and $\forall v\in U$:
\begin{equation*}
\left\{
\begin{aligned}
        &\int_\Gamma   \big\langle \partial_u\dbH_\gamma [t], v - \bar u(t) \big\rangle_{H_{1}}  \lambda (d\gamma )=0 , 
        \\
        &  \int_\Gamma  \big\langle \big(\partial_{uu} \dbH_\gamma [t] + \partial_u b_\gamma[t]^{*}P_\gamma (t)\partial_u b_\gamma[t]\big) ( v    -   \bar u(t) ),  v   -     \bar u(t)  \big\rangle_{H_{1}}   \lambda (d\gamma )=0, 
\end{aligned}
\right.
\end{equation*}
since $\mathbb{H} \equiv0$ and $\partial_u b_\gamma[t]\equiv0$ for any $t\in[0,T]$. 
Similarly, 
$ \dbS_\gamma(t)= \partial_{u}a_\gamma[t]^* P_\gamma(t) = \begin{pmatrix}
        P_{11,\gamma } & P_{12,\gamma} \\
        0 & 0 
    \end{pmatrix}$, 
which further indicates that $\nabla\mathbb{S}_\gamma(t)=0, t\in[0,T]$.

The second order operator-valued adjoint equations for any $\gamma\in \Gamma $ are repectively 
\begin{equation*}
    \left\{
\begin{aligned}
    & d P_\gamma (t) =- (A^{*} +\partial_x a_{\gamma }^{*}[t] ) P_\gamma (t) d t-P_\gamma (t) (A+\partial_x a_{\gamma }[t] ) d t-\partial_x b_{\gamma }^{*}[t] P_\gamma  (t) \partial_x b_{\gamma }[t] d t
    \\
    & \indent\indent\indent 
    - (\partial_x b_{\gamma }^{*}[t] Q_\gamma(t) +Q_\gamma (t) \partial_x b_{\gamma }[t] ) d t+Q_\gamma (t) d W(t) \quad    \text{in}\ [0, T),  \\
    & P_\gamma (T)= \begin{pmatrix} 
        (\gamma-2) I & 0 \\
        0 & (1-\gamma ) I
        \end{pmatrix}, \  \gamma \in \Gamma.  
\end{aligned}
    \right. 
\end{equation*}
Denote 
$ P_\gamma =
    \begin{pmatrix}
        P_{11,\gamma } & P_{12,\gamma} \\
        P_{21,\gamma} & P_{22,\gamma}
    \end{pmatrix},  
    Q_\gamma =
    \begin{pmatrix}
        Q_{11,\gamma } & Q_{12,\gamma} \\
        Q_{21,\gamma} & Q_{22,\gamma}
    \end{pmatrix}$. 
It is obvious that $(P_\gamma, Q_\gamma)=(\tilde P_\gamma, 0)$, where $\tilde P_\gamma$ is the mild solution to the following deterministic evolution equation  
\begin{equation*}
    \left\{
\begin{aligned}
    & d \tilde P_\gamma (t) =- (A^{*} +\partial_x a_{\gamma }^{*}[t] ) \tilde P_\gamma (t) d t - \tilde P_\gamma (t) (A+\partial_x a_{\gamma }[t] ) d t-\partial_x b_{\gamma }^{*}[t] \tilde P_\gamma  (t) \partial_x b_{\gamma }[t] d t
      \quad    \text{in}\ [0, T),  \\
    & \tilde P_\gamma (T)= \begin{pmatrix} 
        (\gamma-2) I & 0 \\
        0 & (1-\gamma ) I
        \end{pmatrix}, \  \gamma \in \Gamma,
\end{aligned}
    \right.
\end{equation*}
and moreover, $P_\gamma (\cdot) \in C_{\mathcal{S}}([0, T] ; \mathcal{L}(H)) $, $ P_{11,\gamma }(t)=\tilde{P}_{11,\gamma }(t)<0, t \in[0, T]$.  
Besides, it can be verified that 
$P_\gamma (\cdot) \in \mathbb{L}_{2, \mathbb{F}}^{1,2}(\mathcal{L}_{2}[ H_{0}^{1}(0,1) \times H_{0}^{1}(0,1) ;L^{2}(0,1) \times L^{2}(0,1)]) \cap L^{\infty}([0, T] \times \Omega ; \mathcal{L}_{2}[H_{0}^{1}(0,1) \times H_{0}^{1}(0,1) ;L^{2}(0,1) \times L^{2}(0,1)])$. 
After the detour, we have verified the conditions. 

\section{Appendix}\label{slyz}
\subsection{Elementary lemmas}
    
\begin{lemma}[Minimax Theorem \cite{pham09}]\label{mimmaxthe}
    Let $\mathcal{M}$  be a convex subset of a normed vector space,  and  $\mathcal{N}$  a weakly compact convex subset of a normed vector space. If  $\varphi$  is a real-valued function on $ \mathcal{M} \times  \mathcal{N} $ with 
    
    $x\to \varphi (x,y)$ is continuous and convex on $\mathcal{M}$ for all $y\in \mathcal{N}$,

    $y\to \varphi (x,y)$ is concave on $\mathcal{N}$ for all $x\in \mathcal{M}$.

    \noindent 
    Then 
\[\sup_{y\in \mathcal{N}}\inf_{x\in \mathcal{M}}\varphi (x,y) =\inf_{x\in \mathcal{M}}\sup_{y\in \mathcal{N}}\varphi (x,y).\]
\end{lemma}

\begin{lemma}[Partitions of Unity Theorem \cite{rudin06}]\label{partunith}
    Denote by $\{\mathcal{V}_{i}\}_{i=1}^n$ any open subsets of a locally compact Hausdorff space $ \mathfrak{M}  $. If a compact set $\mathcal{K} $ satisfies 
    $\mathcal{K} \subset \mathcal{V}_1 \cup \ldots \cup \mathcal{V}_n, $ 
    then there exists $\{\varphi_i\}_{i=1}^n$ satisfying 
    $\Sigma_{i=1}^n \varphi_i(x)=1,  x\in \mathcal{K},$ 
    where $\varphi_i$ is continuous on $\mathfrak{M}$, with its compact support lies in $\mathcal{V}_i$, $0\le \varphi_i \le 1$, $i=1,\ldots  n$.
\end{lemma}

\subsection{Preliminaries in Malliavin calculus}
A comprehensive introduction to Malliavin calculus can be found in the textbook \cite{NualartDavidEulalia}. Here, we present only a few elementary definitions and symbols to be used.

For any $h\in L^2(0,T)$, denote $W(h)=\int_0^T h(t) {d}W(t)$.
Let $\widetilde H$ be a separable Hilbert space and denote by $C_b^\infty (\mathbb R^j)$ the set of $C^\infty$-smooth functions with bounded partial derivatives. 
For any smooth $\widetilde{H}$-valued random variable of the form  
\[F=\sum_{j=1}^{k}f_j(W(h_1),W(h_2),\ldots,W(h_{j_m}))\mathfrak h_j, \]
where 
$\mathfrak h_j\in\widetilde{H}, j=1,\dots,k$, $f_j\in C_b^\infty(\mathbb{R}^{j_m}), h_{j_n}\in L^2(0,T), j_n=1,2,\ldots,j_m, k,j_m\in\mathbb{N}$, the Malliavin derivative of $F$ is defined as 
\[\mathcal{D}_t F=\sum_{j=1}^k\sum_{j_n=1}^{j_m}h_{j_n}(t)\frac{\partial f_j}{\partial x_{j_n}}\left(W(h_1),W(h_2),\ldots,W(h_{j_m})\right)\mathfrak h_j.\]
Then $\mathcal{D} F$ is a smooth random variable with values in $L^2(0,T;\widetilde H)$.
Denote by $ \mathbb{D}^{1,2}(\widetilde{H} )$  the completion of the set of the above smooth random variables with respect to the norm 
\[|F|_{\mathbb{D}^{1,2}}=\Big(\mathbb{E}|F|_{\widetilde{H} }^2+\mathbb{E} \int_0^T|\mathcal{D}_t F|_{\widetilde{H} }^2 {d}t \Big)^{\frac{1}{2}}.\]
For $\xi \in\mathbb{D}^{1,2}(\widetilde{H} )$, the following Clark-Ocone representation formula holds:
\begin{equation}\label{claocoori}
\xi =\mathbb{E} \xi +\int_0^T\mathbb{E}\left(\mathcal{D}_t \xi  | \mathcal{F}_t\right) {d}W(t).
\end{equation}

Define $\mathbb{L}^{1,2}(\widetilde{H} )$ to be the space of processes $\varphi\in L^2([0,T]\times \Omega;\widetilde{H})$, such that

(i) for a.e. $t\in[0,T]$, $\varphi(t,\cdot)\in\mathbb{D}^{1,2}(\widetilde{H})$,

(ii) the function $[0,T]\times[0,T]\times\Omega \ni (s,t,\omega )\to \mathcal{D}_s\varphi( t,\omega )\in \widetilde{H}$ admits a measurable version,

(iii) $\$\varphi\$_{1,2}:=\left(\mathbb{E}\int_0^T|\varphi(t)|_{\widetilde{H}}^2 {d}t +\mathbb{E}\int_0^T\int_0^T|\mathcal{D}_s\varphi(t)|_{\widetilde{H}}^2 {d}s {d}t\right)^{\frac{1}{2}}<\infty$.

The set of all adapted processes in $\mathbb{L}^{1,2}(\widetilde{H})$ is denoted by $\mathbb{L}_{\mathbb{F}}^{1,2}(\widetilde{H})$.

Moreover, define
\begin{equation*}
\begin{aligned}
\mathbb{L}_{2^+}^{1,2}(\widetilde{H}):=\Big\{& \varphi(\cdot)\in\mathbb{L}^{1,2}(\widetilde{H}) \mid \exists\ \mathcal{D}^+\varphi(\cdot)\in L^2([0,T]\times \Omega;\mathbb{R}^n) \ \text{such that}  
\\&   
g_\epsilon  (s):=\sup_{s<t<(s+\epsilon )\wedge T}\mathbb{E}|\mathcal{D}_s\varphi(t)-\mathcal{D}^+\varphi(s)|_{\widetilde{H}}^2<\infty\quad  a.e.\  s\in[0,T],   \\
&  g_\epsilon(\cdot)\ \text{is measurable on}\ [0,T]\ \text{for any}\ \epsilon>0,\ \text{and}\ \lim_{\epsilon\to0^+}\int_0^T g_\epsilon(s) {d}s=0 \Big\},
\end{aligned}
\end{equation*}
and
\begin{equation*}
\begin{aligned}
\mathbb{L}_{2^-}^{1,2}(\widetilde{H}):=\Big\{& \varphi(\cdot)\in\mathbb{L}^{1,2}(\widetilde{H}) \mid \exists\ \mathcal{D}^-\varphi(\cdot)\in L^2([0,T]\times \Omega;\widetilde{H}) \ \text{such that}  
\\&   
g_\epsilon(s):=\sup_{(s-\epsilon)\vee0<t<s}\mathbb{E}|\mathcal{D}_s\varphi(t)-\mathcal{D}^-\varphi(s)|_{\widetilde{H}}^2<\infty\quad  a.e.\  s\in[0,T],   
\\&  
g_\epsilon(\cdot)\ \text{is measurable on}\ [0,T]\ \text{for any}\ \epsilon>0,\ \text{and}\ \lim_{\epsilon\to0^+}\int_0^T g_\epsilon(s) {d}s=0 \Big\}.
\end{aligned}
\end{equation*}
Let $\mathbb{L}_{2}^{1,2}(\widetilde{H})=\mathbb{L}_{2^+}^{1,2}(\widetilde{H})\cap\mathbb{L}_{2^-}^{1,2}(\widetilde{H})$.
For any $\varphi\in\mathbb{L}_{2}^{1,2}(\widetilde{H})$, denote $\nabla\varphi(\cdot)=\mathcal{D}^+\varphi(\cdot)+\mathcal{D}^-\varphi(\cdot)$.
$\mathbb{L}_{2,\mathbb{F}}^{1,2}(\widetilde{H})$ consists of all the adapted processes in $\mathbb{L}_{2}^{1,2}(\widetilde{H})$.

{\small
\section*{Competing interests}
The authors have no competing interests to declare that are relevant to the content of this article.  
}

{\small
\section*{Funding}
Research supported by National Key R\&D Program of China (No. 2022YFA1006300) and the NSFC (No. 12271030).
}

{\small
\section*{Acknowledgments}
The authors are grateful to referees and editors for helpful comments that have significantly improved the paper.}

\bibliographystyle{IEEEtran}
\end{document}